\documentclass[vecarrow,chaprefs]{svmult} 

\usepackage{type1cm}        
\usepackage{graphicx}
\usepackage[bottom]{footmisc}

\usepackage{newtxtext}
\usepackage{newtxmath}
\usepackage{mathtools}

\newtheorem{assumption}[theorem]{Assumption}

\def\R{\mathbb{R}}  
\def\M{\mathcal{M}} 

\DeclareMathOperator{\Sq}{sq}
\DeclareMathOperator{\Sqrt}{sqrt}
\DeclareMathOperator{\Cov}{Cov}
\DeclareMathOperator{\Deriv}{D}

\DeclareMathOperator{\Dom}{dom}
\DeclareMathOperator{\Entropy}{\mathcal H}
\DeclareMathOperator{\Expectation}{\mathbb E} 
\DeclareMathOperator{\Grad}{grad}
\DeclareMathOperator{\Maxexp}{\mathcal E}
\DeclareMathOperator{\eDeriv}{D_{\text{e}}}
\DeclareMathOperator{\hDeriv}{D_0}
\DeclareMathOperator{\mDeriv}{D_{\text{m}}}

\newcommand{\KH}[2]{\operatorname{DH}\left(#1 \middle| #2 \right)}

\newcommand{\Derivby}[1]{\frac{\Deriv}{d#1}}
\newcommand{\KL}[2]{\operatorname{D}\left(#1 \middle| #2 \right)}
\newcommand{\acceleration}[1]{\overset{\star\star}{#1}}
\newcommand{\avalof}[1]{\left\vert#1\right\vert}
\newcommand{\coshtwo}{\cosh_2}
\newcommand{\covat}[3]{\Cov_{#1}\left(#2,#3\right)}

\newcommand{\derivby}[1]{\frac{d}{d#1}}

\newcommand{\displacement}{\operatorname{\mathbb S}}
\newcommand{\domof}[1]{\Dom\left(#1\right)}
\newcommand{\eDerivby}[1]{\frac{\eDeriv}{d#1}}

\newcommand{\entropyof}[1]{\Entropy\left(#1\right)}
\newcommand{\etransport}[2]{\prescript{\text{e}}{} {\mathbb U} _ {#1} ^ {#2}}
\newcommand{\htransport}[2]{\prescript{0}{} {\mathbb U} _ {#1} ^ {#2}}
\newcommand{\euler}{\mathrm{e}}
\newcommand{\expbundleat}[1]{S\!\maxexpat{#1}}
\newcommand{\expectat}[2]{{\Expectation}_{#1}\left[#2\right]}
\newcommand{\expfiberat}[2]{S_{#1}\maxexpat{#2}}

\newcommand{\expof}[1]{\exp\left(#1\right)}

\newcommand{\fullbundleat}[1]{\prescript{1}{}S^1\maxexpat{#1}}
\newcommand{\fullfiberat}[2]{\prescript{1}{}S^1_{#1}\maxexpat{#2}}
\newcommand{\gaussdensity}{\gamma}

\newcommand{\hDerivby}[1]{\frac{\hDeriv}{d#1}}

\newcommand{\mDerivby}[1]{\frac{\mDeriv}{d#1}}
\newcommand{\macc}{\prescript{m}{}\Deriv^2}
\newcommand{\maxexpat}[1]{\Maxexp\left(#1\right)}
\newcommand{\mixbundleat}[1]{{}^*\!S\!\maxexpat{#1}}
\newcommand{\mixfiberat}[2]{{}^*S_{#1}\maxexpat{#2}}

\newcommand{\mtransport}[2]{\prescript{\text{m}}{} {\mathbb U} _ {#1} ^ {#2}}

\newcommand{\normat}[2]{\left\Vert #2 \right\Vert_{#1}}
\newcommand{\normof}[1]{\left\Vert #1 \right\Vert}
\newcommand{\orliczof}[2]{L_{#1}\left(#2\right)}
\newcommand{\orliczpof}[3]{L_{#1}^{#2}\left(#3\right)}
\newcommand{\pderivby}[1]{\frac {\partial} {\partial #1}}
\newcommand{\probabilities}{\mathbb P}
\newcommand{\pstar}{\prescript{*}{}}

\newcommand{\scalarat}[3]{\left\langle #2 , #3 \right\rangle_{#1}}
\newcommand{\scalarof}[2]{\left\langle#1,#2\right\rangle}
\newcommand{\setof}[2]{\left\{ #1 \,\middle|\, #2 \right\}}
\newcommand{\set}[1]{\left\{ #1 \right\}}
\newcommand{\smeasures}{\mathbb M}
\newcommand{\tensorat}[3]{\prescript{#1}{}S^{#2}\maxexpat{#3}}

\newcommand{\transport}[2]{{\mathbb U}_{#1}^{#2}}
\newcommand{\tricovat}[4]{\Cov_{#1}\left(#2,#3,#4\right)}
\newcommand{\tvnorm}[1]{\normat {\operatorname{TM}} {#1}}
\newcommand{\vectorat}[2]{S_{#1}(#2)}
\newcommand{\velocity}[1]{\overset{\star}{#1}}

\usepackage{bm}
\usepackage{url}

\AtBeginDocument{%
}

\begin{document}


\title*{Dually affine Information Geometry modeled on a Banach space}

\titlerunning{Affine statistical bundles}

\author{Goffredo Chirco and Giovanni Pistone}
\authorrunning{G. Chirco, G. Pistone}
\institute{Goffredo Chirco \at Dipartimento di Fisica ``Ettore Pancini'', Universit\`a degli Studi di Napoli Federico II, Complesso Universitario di Monte Sant'Angelo, Via Cinthia, 21, 80126 Napoli, Italy  \\
\email{goffredo.chirco@unina.it}
\and Giovanni Pistone \at de Castro Statistics, Collegio Carlo Alberto, piazza Vincenzo Arbarello 8, 10122  Torino, Italy \email{giovanni.pistone@carloalberto.org}}

\maketitle

\abstract{In this review paper, we present dually affine Information Geometry in terms of the mathematical structure of a statistical bundle $SM$, that is, a set of couples $(q,u)$, where the probability density $q$ belongs to an affine space $M$ and $u$ is a random variable with $\expectat q u  = 0$.}

\section{Overview}

Information Geometry (IG) is one of the modern outcomes of various lines of research. The first non-parametric version of IG dates back to the work of J. W. Gibbs. In his 1902 monograph titled \emph{Elementary Principles in Statistical Mechanics}, Gibbs presented an innovative approach using the ``theory of error'' to describe uncertainty in mechanical systems. His program deploys new concepts intended to describe the time evolution of probability distributions. We review these concepts below in contemporary mathematical language.

Strictly positive probability densities describe uncertainty in some measurable space. Because of the strict positivity, any such density $q$ is of exponential form
\begin{equation*}
    q = \euler^{-v} \quad \text{where} \quad v = \log \frac 1 q \ .
\end{equation*}

The logarithmic expression is related to the Boltzmann entropy, while the exponential expression refers to the Maxwell distribution of densities. Gibbs assumes $v$ to be bounded below and rewrites as
\begin{equation*}
    q = \euler^{c - w} \quad \text{where $c$ is a constant and $\inf w = 0$} \ .
\end{equation*}

The next step is the mechanical interpretation of $w$, the computation of its mean value $\expectat q w$, and the variance expression as $w - \expectat q w$. In conclusion, Gibbs obtains the key equation
\begin{equation*}
    q = \euler^{u - H(q)}, \text{where $\expectat q u = 0$, hence $H(q) = - \expectat q {\log q}$.}
\end{equation*}

The expression above has many consequences. First, the "typical value" of $u$ is 0. Hence, the typical value of $- \log q$ is $H(q)$. This observation is the starting point of the applications of this formalism to Information Theory made by C. E. Shannon in his \emph{A Mathematical Theory of Communication} (1948). Second, the random variable $u$ is uniquely defined and, in turn, uniquely defines the probability density:
\begin{equation*}
u = \log q - \int q \log q \quad \text{and} \quad q = \frac{\euler^u}{\int \euler^u} = \euler^{u - \log \int \euler^u} \ .
\end{equation*}
Third, an interesting time evolution $t \mapsto q(t)$ is given by a linear model in $u$ as 
\begin{equation*}
    t \mapsto \euler^{tu - \log \int \euler^{tu}} \ .
\end{equation*}

The last equation is a case of what was subsequently called an exponential family. See the monograph  \cite{brown:86}. An exponential family is a family of probability densities of the form
\begin{equation*}
    q = \exp\left(\sum_{j=1}^d \theta_j u_j - \psi(\theta)\right) \cdot p \ .
\end{equation*}
If the random variables $u_1,\dots,u_d$ are affinely independent, then the map $q \mapsto \theta$ is well defined and provides an example of a chart in the geometrical sense of the term. This opens the way to the idea of considering statistical models as manifolds. 

R. A. Fisher has provided a further ingredient in his ``Mathematical Foundations of Theoretical Statistics'' (1922). He considers a statistical model $\theta \mapsto q_\theta$ versus a reference probability density $p$ and studies the likelihood $q_\theta/p$ in the logarithmic scale $\theta \mapsto \log q_\theta - \log p$. In such a scale, the velocity of variation is $\derivby \theta \log q_\theta = \dot q_\theta/q_\theta$ and is named Fisher's score, while its squared intensity is the variance 
\begin{equation*}
    \expectat{q_\theta}{\left(\frac{\dot q_\theta}{q_\theta}\right)^2} = \int \frac{\dot q_\theta ^2}{q_\theta} \ ,
\end{equation*}
which is named Fisher's Information. It is relevant to note that $\expectat {q_\theta}{\derivby \theta \log q_\theta} = 0$.

The use of differential geometry in the study of statistical models was first devised by C.R. Rao \cite{rao:45} who recognized that the Information matrix, that is, the variance matrix of $\nabla \log q_\theta$ defines a Riemannian manifold on the model parameters. It was later recognized by B. Efron \cite{efron:1975,efron:1978}  that the affine geometry of the exponential families provides a more interesting geometric setup that connects with such fundamental topics as entropy, information, and Gibbs distribution. See also \cite{dawid:75,dawid:1977as}. It was the work of \v Cenkov \cite{cenkov:1982} and Amari \cite{amari:85} to connect the various branches in the topic we call Information Geometry.

 The modern form of this line of research is due to A. Amari and I. Nagaoka; see the monograph \cite{amari|nagaoka:2000}. Their presentation of the affine geometry of statistical models depends on the standard formalism of differential manifolds of finite dimension. We present here a version that holds in a non-parametric case, as it was the original approach by Gibbs.

\section{Non-parametric manifold}
\label{sec:non-param-manif}

 This section concisely presents those parts of differential geometry that are useful in our context. We specifically chose to avoid the use of parametric statistical models and refer systematically to statistically meaningful natural operations. For this reason, geometric notions such as the tangent space of the manifold and parallel transports are dealt with peculiarly. Due to the specific issues of the non-parametric setup, we adopt Bourbaki's definition of a differentiable manifold. For non-parametric differential geometry, see \cite[sec.~5]{bourbaki:71variete}, \cite[Ch~II]{lang:1995}, and \cite[Ch~3]{abraham|marsden|ratiu:1988}. 
 In particular, sec.~\ref{subsec:generalities} defines the notions of differentiable manifold and vector bundle.

 \subsection{Generalities} \label{subsec:generalities}
A \emph{chart} on a set $M$ is a triple $(s,U,B)$, where $s$ is a 1-to-1 mapping from its domain $U \subset M$ to an \emph{open} subset $s(U)$ of a topological vector space $B$. The space $B$ is called coordinates' space or model space. A \emph{topological vector space} is a (real) vector space endowed with a topology such that all the vector space operations are continuous, see \cite[Ch~2 \S1]{lang:1995}. We will call top linear mappings all continuous linear mappings between topological linear spaces. Normed vector spaces and Banach spaces are instances of topological vector spaces, but there are other instances of interest. In our applications, the topological vector spaces of interest will be Banach spaces of random variables. We are not going to use any especially sophisticated notion from Functional Analysis. For basic reference, see \cite[Ch.~2]{abraham|marsden|ratiu:1988} or \cite{rudin:1987-3rd}. In general, the definitions below are a specialization of those in \cite{lang:1995}. The monograph \cite[Ch.~10]{kass|vos:1997} was the first to mention the specific interpretation of IG we adopt.

The manifold structure consists of a given set of compatible charts. We do not assume that the model space $B$ is finite-dimensional, nor is it the same for each chart.

\begin{definition}[$C^k$-atlas]\label{def:manifold} Let $M$ be a set and let $B_\alpha$, $\alpha \in A$, be a family of Banach spaces. For each $\alpha \in A$, $(s_\alpha,U_\alpha,B_\alpha)$ is a \emph{chart} on $M$, that is, $s_\alpha \colon U_\alpha \to B_\alpha$ is 1-to-1 from $U_\alpha \subset M$ to the open set $s_\alpha(U_\alpha) \subset B_\alpha$.  We assume that each couple of charts, say $(s_\alpha,U_\alpha,B_\alpha)$ and $(s_\beta,U_\beta,B_\beta)$, either have disjoint domains, $U_\alpha \cap U_\beta = \emptyset$, or the \emph{change of chart} (or transition mapping)
  \begin{equation} \label{eq:manifold}
    s_{\beta} \circ s^{-1}_{\alpha} \colon s_\alpha(U_\alpha \cap U_\beta) \to s_\beta(U_\alpha \cap U_\beta) 
  \end{equation}
  is a 1-to-1 $C^k$ mapping, $k \geq 0$, between open sets. The set $\mathcal A$ of all charts is a $C^k-$\emph{atlas}.
\end{definition}

In the definition above, for each couple of overlapping domains, there is a homeomorphism between open sets of the corresponding model spaces, say $B_1, B_2$. If $k \geq 1$, the derivative $d(s_2 \circ s_1^{-1})$ of change of chart provides a 1-to-1 continuous linear mappings of $B_1$ onto $B_2$. Hence, all the model spaces of overlapping domains are toplinear isomorphic. This remark is critical for model building in the infinite-dimensional case, where the equality of the respective dimensions does not ensure such an isomorphism.

\begin{definition}[$C^k$-manifold] Two atlases are equivalent if their union is again an atlas. A class of equivalent atlases on the set $M$ is a \emph{manifold} $\M$. 
\end{definition}

From now on, we assume our manifold to be differentiable, that is, $k \geq 1$.

\begin{definition}[Velocity in a chart]
If $s \colon U \to B$ is a chart on the set $M$, and $t \mapsto x(t) \in s(U) \subset B$ is a curve, then $t \mapsto s(x(t))$ is the \emph{expression} of the curve in the chart $s$ and $\derivby t s(x(t)) \in B$ is the \emph{expression of the velocity} of the curve in the chart $s$.
\end{definition}

 If the curve has values in the intersection of the domains of two charts $s_1$ and $s_2$, the expression of the velocity in the chart $s_1$ is mapped to the expression of the velocity in the chart $s_2$ by the derivative of the change of chart,
\begin{equation*}
  \derivby t s_2(x(t)) = \derivby t s_2 \circ s_1^{-1} (s_1(x(t)) = d \left(s_2 \circ s_1^{-1}\right) \left[\derivby t  s_1(x(t))\right] \ .
\end{equation*}

In particular, we are interested in atlases with one distinguished chart for each point of $M$.

\begin{definition}[Frame bundle]\label{def:frame-bundle}
A \emph{frame bundle} for the $C^k$-manifold $\M$ is a defining atlas $(U_x,s_x,B_x)$, $x \in M$, such that $s_x(x) = 0$ together with a system of toplinear isomorphism $\transport x y \colon B_x \to B_y$ such that $\transport y x \transport x y = \transport x x$ is the identity, $x,y \in M$. We assume that the vector bundle $N = \setof{(x,v)}{x \in M, v \in B_x}$ is a $C^k$-manifold for the charts
\begin{equation}
    S_x \colon (y,w) \mapsto (s_x(y), \transport y x w) \in B_x \times B_x \ .
\end{equation}The isomorphisms $\transport x y$ are the \emph{parallel transport} of the frame bundle. 
\end{definition}

Different parallel transports can exist on the same atlas.  Some parallel transports have the \emph{cocycle property}, $\transport y z \transport x y = \transport x z$.

The frame bundle provides a \emph{moving frame}, a smooth system of charts associated with each reference point.

\begin{definition}[Velocity and auto-parallel curves]\label{def:velocity-parallel}
 Let $(M,(s_x)_{x \in M},(B_x)_{x \in M},(\transport x y)_{x,y \in M})$ is a given frame bundle according to def.~\ref{def:frame-bundle}. If $t \mapsto x(t) \in M$ is a smooth curve its \emph{velocity} is 
\begin{equation}\label{eq:velocity-general}
\velocity x(t) = \left. \derivby t s_x(x(t)) \right|_{x = x(t)} \ . 
\end{equation}
The curve is \emph{auto-parallel} if
\begin{equation}
    \velocity x(t) = \transport {x(s)} {x(t)} \velocity x(s) \ .
\end{equation}
\end{definition}

\begin{remark}
Many models in IG involve the use of a distance or divergence. Riemannian manifolds have a distance induced by the metric and a notion of geodesic, that is, a curve between two points of minimal length. However, affine manifolds do not have a distance, and the notion of an auto-parallel curve replaces the notion of geodesic. Geodesics or auto-parallel curves provide a notion of a ``straight line'', hence a geometry. Moreover, we shall see that the computation of the velocity of the lifted curve $t \mapsto (x(t), \velocity x(t))$ defines a notion of acceleration so that auto-parallel curves have zero acceleration.
\end{remark}

In sec.~\ref{subsec:L2-shere-as-a manifold} below, we introduce the first example of differential geometric construction in statistics. The original form of the argument is the classical Hellinger's divergence between probability densities,
\begin{equation*}
    \frac{1}{2} \int \left(\sqrt q - \sqrt p \right)^2 = 1 - \int \sqrt {pq} \ .
\end{equation*}
In Hellinger divergence, each density maps into its square root, a random variable belonging to the unit sphere of the $ L^2$ space, and half the squared $ L^2$-norm provides a measure of the divergence.

\subsection{The unit sphere of $L^2(\mu)$ as a manifold}
\label{subsec:L2-shere-as-a manifold}
Let $(\Omega,\mathcal F,\mu)$ be a probability space. We refer to \cite{malliavin:1995} for background material in measure theory. The unit sphere of the Hilbert space $L^2(\mu)$ with norm $\normof \rho ^2 = \int \avalof \rho ^2 \ d\mu$ and inner product $\scalarof u v = \int uv \, d\mu$, is our base set,
\begin{equation}
    \label{eq:unit-sphere}
    M = \setof{\rho \in L^2(\mu)}{ \normof \rho = 1} \ .
    \end{equation}

The unit shere is a sub-manifold of $L^2(\mu)$ on general principles; see, for example, \cite[II, \S2]{lang:1995}. Here, we prefer to define our charts directly to obtain a frame bundle. If $p$ is a probability density, then $\sqrt p \in M$. This remark provides a 1-to-1 mapping from probability densities into $M$, which suggests a way to transfer the sphere's geometry to the set of probability densities. See a more general case in \cite{ay|jost|le|schwachhofer:2017}. The geometry of unit spheres of $L^p$ spaces, $p > 1$, is discussed in \cite{gibilisco:2020}.

\subsubsection{Orthogonal projection}
We construct a frame bundle according to def.~\ref{def:frame-bundle} by orthogonally projecting the unit spere $M$ to its tangent hyperplanes.  For each $\alpha \in M$, the chart's domain is the open half-sphere $U_\alpha = \setof{\rho \in M}{\scalarof \rho \alpha > 0}$ and the coordinates' space is the space of random variables in $L^2(\mu)$ which are orthogonal to $\alpha$, $B_\alpha = \setof{u \in L^2(\mu)}{\scalarof u \alpha = 0}$. The chart centred at $\alpha$, and its inverse are, respectively,
\begin{gather}
  s_\alpha \colon U_\alpha \ni \rho \mapsto \rho-\scalarof{\rho}{\alpha} \alpha = u \in B_\alpha \ , \label{eq:op-chart} \\
  s_\alpha^{-1} \colon S_\alpha \ni u \mapsto u + \sqrt{1 - \normof u ^2} \alpha = \rho  \in U_\alpha \ , \label{eq:inverse-op-chart}
\end{gather}
with chart domain $s_\alpha(U_\alpha) = \setof{u \in B_\alpha}{\normof u < 1}$. Eq.~\eqref{eq:inverse-op-chart} follows from eq.~\eqref{eq:op-chart} by observing that $\normof u ^2 = 1 - \scalarof \alpha \rho ^2$.

The change of chart and its derivative in the direction $h$ are, respectively,
\begin{gather}
  \label{eq:1}
  s_\beta \circ s_\alpha^{-1}(u) = (u - \scalarof u \beta \beta) + \sqrt{1 - \normof u ^2} (\alpha - \scalarof \alpha \beta \beta) \ , \\
  d s_\beta \circ s_\alpha^{-1}(u)[h] = (h - \scalarof h \beta \beta) - \frac{\scalarof h u}{\sqrt{1 - \normof u ^2}} (\alpha - \scalarof \alpha \beta \beta) \ . 
\end{gather}
Note that we use the notation $df(u)[k]$ to denote the derivative of the function $f$ in the direction $k$. Alternative notations are $D_hf(u)$, $Df(u) \cdot k$, $f'(u) k$. For example, the notation with square brackets is used in \cite{absil|mahony|sepulchre:2008}. 

The change-of-chart map depends non-linearly on $u$ in the term $\sqrt{1 - \normof u^2}$, which is infinitely differentiable on the given domain. Hence, our atlas defines a $C^k$-manifold for all $k$.

Especially we can define a linearization of the change-of-chart by computing the derivative at $u = 0$,
\begin{gather}
  d s_\beta \circ s_\alpha^{-1}(0) \colon B_\alpha \ni h \mapsto  d s_\beta \circ s_\alpha^{-1}(0)[h] = h - \scalarof h \beta \beta \in B_\beta \ , \\
    \left(d s_\beta \circ s_\alpha^{-1}(0)\right)^{-1} \colon B_\beta \ni   k \mapsto  k - \frac {\scalarof k \alpha}{\scalarof \beta \alpha} \beta \in B_\alpha \ . \end{gather}
  Such mapping is a topological vector space isomorphism of the Hilbert spaces $B_\alpha$ and $B_\beta$.
  
  We now look for a natural parallel transport, one which is a Hilbert isometry. For each $\alpha,\beta \in M$, define
  \begin{equation}
    \label{eq:sphere-transport}
    \transport \alpha \beta \colon B_\alpha \ni v \mapsto v - (1 + \scalarof \alpha \beta)^{-1}\scalarof \beta v (\alpha+\beta) \ .
  \end{equation}

By noting that $\normof{\alpha+\beta}^2 = 2(1 + \scalarof \alpha \beta)$, one easily verifies that
\begin{enumerate}
\item $\transport \alpha \beta \colon B_\alpha \to B_\beta$;
\item $\transport \beta \alpha = \left(\transport \alpha \beta\right)^{-1}$ and $\transport \alpha \alpha = \operatorname{Id}$;
\item $\normof {\transport \alpha \beta v} ^2 = \normof v ^2$.
\end{enumerate}
  
The velocity of a curve $t \mapsto \rho(t) \in M$ in the chart centred at $\alpha$ is
\begin{equation}
  \label{eq:2}
  \derivby t s_\alpha(\rho(t)) = \derivby t \left( \rho(t) - \scalarof {\rho(t)} {\alpha} \alpha \right) = \dot \rho(t) - \scalarof {\dot\rho(t)} \alpha \alpha \ , 
\end{equation}
where $\dot \rho$ is the derivative computed in $L^2(\mu)$. The velocity in the moving chart, that is at $\alpha = \rho(t)$, is
\begin{equation}
  \label{eq:3}
  \velocity \rho(t) = \left. \derivby t s_\alpha(\rho(t)) \right|_{\alpha=\rho(t)} =  \dot \rho(t) - \scalarof {\dot\rho(t)} {\rho(t)} \rho(t) = \dot \rho(t) \ .
\end{equation}

We use the family $\left(\transport \alpha \beta\right)_{\alpha,\beta}$ as a system of parallel transports, meaning that the vector $w \in B_\beta$ is parallel to the vector $v \in B_\alpha$ if $w = \transport \alpha \beta v$. The curve $t \mapsto \rho(t)$ is auto-parallel if 
\begin{equation}
  \label{eq:7}
  \transport {\rho(s)}{\rho(t)} \velocity \rho(s) = \velocity \rho(t) \ .
\end{equation}

The auto-parallel curve such that $\rho(0) = \rho$ and $\velocity \rho(0) = v$ satisfies the differential equation
\begin{equation}
  \label{eq:9}
  \dot \rho(t) = v - (1+ \scalarof \rho {\rho(t)})^{-1} (\rho+\rho(t)) \scalarof {\rho(t)} v \ .
\end{equation}

Let us show that
\begin{equation}
  \label{eq:riemannian-geodesic}
\rho(t) = \rho \cos \normof v t + \normof v ^{-1} v \sin \normof v t  
\end{equation}
is an auto-parallel curve with $\rho(0) = \rho$ and $\velocity \rho(t) = v$. The velocity is
\begin{equation}
  \label{eq:19}
  \velocity \rho(t) = - \rho \normof v \sin \normof v t + v \cos \normof v t \ .
\end{equation}

The equations $\scalarof {\rho(t)}\rho=\cos \normof v t$ and $\scalarof {\rho(t)} v = \normof v \sin \normof v t$ follow from eq.~\eqref{eq:riemannian-geodesic}. A simple computation shows that it satisfies eq.~\eqref{eq:9}.

Conversely, the auto-parallel \eqref{eq:riemannian-geodesic} connecting $\rho_0 = \rho(0)$ to $\rho(1) = \rho_1$ satisfies 
\begin{equation*}
\rho_1 = \rho_0 \cos \normof {\velocity \rho(0)} + \normof {\velocity \rho(0)} ^{-1} \velocity \rho(0) \sin \normof {\velocity \rho(0)} \ ,  
\end{equation*}
hence $v = \velocity \rho(0)$ is
\begin{equation}\label{eq:18}
  v = \normof v \left(1 - \normof v ^2\right)^{-1/2}(\rho_1 - \scalarof {\rho_1}{\rho_0} \rho_0) \quad \ , \quad \normof v = \arccos \scalarof {\rho_0}{\rho_1} \ . 
\end{equation}

\begin{remark}
  The auto-parallel curves \eqref{eq:riemannian-geodesic} are geodesics; that is, they have a minimal  squared length
  \begin{equation}
    \label{eq:20}
    \int_0^1 \normat {L^2(\mu)} {\velocity \rho(t)} ^ 2 dt \ .
  \end{equation}
  We do not discuss the metric properties of the geometry here because we focus on the dually affine geometry. For the metric properties, we refer to general treatises on Riemannian geometry and geometric analysis, such as \cite{lang:1995}, \cite{klingenberg:1995}, and \cite{jost:2017textbook}.
  \end{remark}

\begin{remark}\label{rem:other-projections}
The classical device of stereographic projection produces the same manifold as the orthogonal projection, but the atlas differs. For each reference point $\alpha \in M$, the chart's domain is $U_\alpha = M \setminus \set{-\alpha}$ and the coordinates' space is $B_\alpha = \setof{u \in L^2(\mu)}{\scalarat \mu u \alpha = 0}$. The chart centered at $\alpha$, $s_\alpha$, maps each $\rho \in M$ to the unique point of the tangent affine hyperplane at $\alpha$ --- identified to $B_\alpha$ by taking $\alpha$ itself as the origin --- which aligns with $-\alpha$ and $\rho$. The equations for $s_\alpha$ and its inverse are, respectively,
  \begin{gather*}
    s_\alpha \colon U_\alpha \ni \rho \mapsto \frac 2 {1+\scalarof \alpha \rho} \rho - \frac {2\scalarof \alpha \rho} {1+\scalarof \alpha \rho} \alpha
    = s_\alpha(\rho) \in B_\alpha \ , \\
    s^{-1}_\alpha \colon B_\alpha \ni u \mapsto \frac 1 {1 + \normof{u/2}^2}  u + \frac{1 - \normof{u/2} ^2}{ 1+ \normof{u/2}^2} \alpha = s^{-1}_\alpha(u) \in U_\alpha \ .
  \end{gather*}

   Projection from the origin gives the atlas
   \begin{gather*}
       \label{eq:origin}
       s_\alpha(\rho) = \scalarof \rho \alpha ^{-1} \rho - \alpha \\
       s_\alpha^{-1}(v) = \left(1+\normof u ^2)\right)^{-1/2}(u + \alpha)
   \end{gather*}
  \end{remark}

\subsection{Square root  embedding}

Unit sphere geometry has been one of the first attempts to give a geometrical structure to the set of probability densities via the square root embedding $p \mapsto \sqrt p$. The argument applies to probability measures
\cite{kakutani:1948}. See the monograph \cite{ay|jost|le|schwachhofer:2017} for a presentation of IG along these lines. The $L^p$ case is discussed in \cite{gibilisco:2020}.

The sphere $M$ of $L^2(\mu)$ maps onto the set of probability densities by the square function, $\Sq \colon M \ni \rho \mapsto \rho^2 \in L^1(\mu)$. In fact, $p = \rho^2$ is non-negative and has $\int p \, d\mu = \int \rho^2 \, d\mu = 1$. Conversely, given any probability density $p \in P(\mu)$, it holds $\sqrt p \in M$. Unfortunately, this transformation is not smooth unless the probability space has a finite number of atoms. In the infinite case, the set $P(\mu) \subset L^1(\mu)$ and the set of non-negative points on the sphere have an empty interior, so it is impossible to restrict the transformation to a homeomorphism. In other words, the square mapping is smooth but not locally invertible. This obstruction leaves two options: using the full sphere manifold or restricting the manifold with a model vector space with a stronger topology. 

We can rephrase the argument above in the language of manifolds. The expression of $\Sq$ in the orthogonal projection atlas is
\begin{equation}
  \label{eq:combe-presentation}
  \Sq_\alpha = \Sq \circ s_\alpha^{-1} \colon u \mapsto \left(u + \sqrt{1 - \normof u ^2}\alpha\right)^2 \ , 
\end{equation}
with derivative in the direction $h$
\begin{equation}
  \label{eq:22}
  d_h \Sq_\alpha(u) = 2 s_\alpha^{-1}(u) d_h s_\alpha^{-1}(u) \ .
\end{equation}
If $u=0$, the derivative is $h \mapsto \alpha h$, and the conditions for the inverse theorem do not hold but in the finite case. The presentation of probability densities in eq.~\eqref{eq:combe-presentation} appeared first in \cite{burdet|combe|nencka:2001}in the case $\alpha=1$. 

On the positive side, every statistical model $\theta \mapsto p_\theta$, with $p_\theta \cdot \mu$ equivalent to $\mu$ (that is, $p_\theta > 0$ $\mu$-a.s.) maps to a curve $\theta \mapsto \rho(\theta) = \sqrt{p_\theta}$ and studied as a smooth curve on the unit sphere.

Let us introduce a new equivalent presentation of the Hilbert bundle $(B_\alpha)$ restricted to $\alpha = \sqrt p$ with $p$ in the set of positive probability densities $P_>(\mu)$. For each $p \in P_>(\mu)$, let us write, for short, the square root of the chart in eq.~\eqref{eq:op-chart} as
\begin{multline}
  \label{eq:24}
  s_p \circ \Sqrt \colon P_>(\mu) \ni q \mapsto \sqrt q \mapsto s_{\sqrt p}(\sqrt q) = \sqrt q - \sqrt p \int \sqrt {pq} \, d\mu = \\
  \left(\sqrt\frac q p - \int \sqrt \frac q p \, p d\mu\right) \, \sqrt p = \left(\sqrt\frac q p - \expectat p {\sqrt\frac q p}\right) \, \sqrt p \in B_{\sqrt p} \ .
\end{multline}
The mapping is well defined because $\scalarof {\sqrt p}{\sqrt q} = \int \sqrt {pq} \, d\mu > 0$. In Fisher's statistical terms, the quantity in brackets is the centered square root transformation of the likelihood $q / p$.

To find an even better presentation, we change the Hilbert spaces. The 1-to-1 mapping 
\begin{equation}
  \label{eq:twice-isometry}
  J_{\sqrt p} \colon B_{\sqrt p} \ni u \mapsto \frac2{\sqrt p} u \in L^2_0(p \cdot \mu)  
\end{equation}
is twice an isometry,
\begin{gather}
  \label{eq:12}
  \int \frac2{\sqrt p} u \, p d\mu = 2 \int u \sqrt p\, d\mu = 0 \ , \\
  \int \left(\frac2{\sqrt p} u\right)^2 p d\mu = 4 \int u^2 \, d\mu = 4 \normof u ^2 \ .
\end{gather}

We have so introduced as a frame bundle the Hilbert bundle introduced by \cite[sec.~5]{amari:87}, \cite{amari:87dual}, and \cite[sec.~10.3]{kass|vos:1997}. Simple computations define the Hilbert bundle, the charts, and their inverse,
\begin{gather}
  \label{eq:Hilbert-bundle}
  H P_>(\mu) = \setof {(q,u)}{q \in P_>(\mu), u \in L^2(\mu), \expectat p u = 0} \ , \\
\label{eq:Hilbert-chart}
  s_p(q) = J_{\sqrt p} \circ s_{\sqrt p}(\sqrt q) = 2\left(\sqrt\frac q p - \expectat p {\sqrt\frac q p} \right) \ , \\
  \label{eq:Hilbert-inverse-chart}
  s_p^{-1}(v) = \left(\frac 12 v + \sqrt{1 - \normat p {\frac v 2} ^2} \right)^2 \cdot p \ .
\end{gather}

On the Hilbert bundle~\eqref{eq:Hilbert-bundle}, we have defined the pseudo-charts~\eqref{eq:Hilbert-chart} that are re-parameterized restrictions of the orthogonal projection charts defined on the unit sphere. Moreover, we can construct the Hilbert transports
\begin{equation}
  \label{eq:h-transport}
  \htransport p q = J_{\sqrt q} \circ \transport {\sqrt p} {\sqrt q} \circ J_{\sqrt p}^{-1} \colon L^2_0(p \cdot \mu) \to L^2_0(q \cdot \mu) \ ,
\end{equation}
that is, from eq.~\eqref{eq:sphere-transport},
\begin{multline}
  \label{eq:28}
  \htransport p q u = \frac 2 {\sqrt p} \left(\frac {\sqrt p} 2 u - (1 + \scalarof {\sqrt p} {\sqrt q})^{-1}\scalarof {\sqrt q} {\frac {\sqrt p} 2 u} ({\sqrt p} + {\sqrt q}) \right) = \\
u - \left(1 + \expectat p {\sqrt\frac q p} \right)^{-1} \scalarat p {\sqrt\frac q p} u \left(1 + \sqrt\frac q p\right) \ . \end{multline}
One can easily check by computation that eq.~\eqref{eq:h-transport} holds and $\htransport p q$ is isometric. Cf. the direct computation in \cite{pistone:2020-NPCS}.

Let $t \mapsto q(t) \in P_>(\mu)$ be such that $t \mapsto \sqrt {q(t)} \in M$ is smooth. Then, the velocity in the chart at $p$ is
\begin{equation}
  \label{eq:15}
\derivby t s_p(q(t)) =    
  \frac{\dot q(t)}{\sqrt{pq(t)}} - \expectat p {\frac{\dot q(t)}{\sqrt{pq(t)}}} \ ,
    \end{equation}
    while the velocity in the moving frame $p = q(t)$ is
    \begin{equation}   \label{eq:velocity}
   \velocity q(t) =   \frac {\dot q(t)}{q(t)} - \expectat {q(t)} {\frac {\dot q(t)}{q(t)}} = \frac {\dot q(t)}{q(t)} \ .
    \end{equation}

    The quantity $\velocity q(t) = \derivby t \log q(t)$ is called the Fisher score and is a natural way to score the rate of change in statistical models. The natural way to compute the intensity of change is the variance of the velocity,
    \begin{equation}
      \label{eq:17}
    \normat {q(t) \cdot \mu} {{\velocity q(t)}^2} = \expectat {q(t)} {{\velocity q(t)}^2} = \int \frac {{\dot q(t)}^2}{q(t)} \, d\mu \ ,
    \end{equation}
    which is called Fisher's information.

    The auto-parallel curve starting at $p(0) = p$ with initial velocity $\velocity p(0) = v$ is
    \begin{equation}
      \label{eq:29}
      t \mapsto p(t) = \left(\sqrt p \cos\left(\normat p v \frac t 2\right) + 2 v \sin\left(\normat p v \frac t 2\right) \right)^2 \ .
    \end{equation}

    \begin{remark}
      The same argument applies to the atlases of rem.~\ref{rem:other-projections}, for example, the stereographic projection. In such a case, the representation of probabilities is
      \begin{equation*}
        \label{eq:6}
        s_p^{-1}(v) = \left(1 + \expectat p {\normof {v/4} ^2} \right)^{-2}
        \left( \frac v 2 + ( 1 - \expectat p {\normof {v/4}^2}\right)^2 p \ .
      \end{equation*}
    \end{remark}

    \begin{remark}
The Hilbert bundle naturally produces the definition of velocity as Fisher's score and provides a coherent setup for IG. Precisely, the set of probability densities $\mathcal P(\mu)$ is in 1-to-1 correspondence with the subset of the unit sphere of non-negative functions, $p \leftrightarrow \sqrt p$. In turn, the set of positive elements of the unit sphere is a subset of an open set that supports a manifold structure,
\begin{equation*}
    \setof{\rho \in M}{\rho \geq 0} \subset \setof{\rho \in M}{\scalarof{\rho}{1} > 0} \ .
\end{equation*}

However, as it was suggested initially by \cite[sec.~5]{amari:87} and \cite{amari:87dual}, the actual setup of IG is a dually affine theory. We will show in the following that the affine and the Hilber approach are related to each other by a system of model spaces $B_p'  \hookrightarrow L^2_0(p) \hookrightarrow B_p$, where $\hookrightarrow$ is a dense continuous injection, and the Hilbert inner product extends to a duality pairing of $B'_p$ and $B_p$. On each Banach space, we will define an affine structure such that the definition of velocity is again the Fisher's score ~\eqref{eq:velocity}.
    \end{remark}
    
\subsection{\emph{Affine} space}
\label{sec:affine-space}
We turn now to our main subject, \emph{affine manifolds}. We begin with a general definition and shall turn to the specific case of statistical manifolds in sec.~\ref{sec:exponential-affine-manifold}.

The word "affine" above refers to the geometrical construction of vectors associated with displacements according to classical H.~Weyl's axioms, see \cite[sec.~I.2]{weyl:1952}. The L.~Schwartz' textbook \cite[sec.~III.1]{schwartz:1981} has the same presentation in a slightly different language.

Let $M$ be a set and $V$ a real finite-dimensional vector space. A \emph{displacement} mapping is a mapping
\begin{equation}
M \times M \ni (P,Q) \mapsto \overrightarrow{PQ} \in V \ ,
\end{equation}
such that
\begin{enumerate}
\item for each fixed $P$ the partial mapping $s_P \colon Q \mapsto \overrightarrow{PQ}$ is \emph{1-to-1 and onto}, and
\item the \emph{parallelogram law} or \emph{Chasles rule}, $\overrightarrow{PQ} + \overrightarrow{QR} = \overrightarrow{PR}$, holds true.
\end{enumerate}

The notation $Q = P + \overrightarrow{PQ}$ shows the action of the vector space $V$ on the set $M$. From the parallelogram law, it holds $\overrightarrow{PP} = 0$ and $\overrightarrow{PQ}+\overrightarrow{QP}=0$. The structure $(M,V,\overrightarrow{\phantom{pq}})$ is an \emph{affine space}. The corresponding affine manifold is derived from the atlas of charts $s_P \colon M \to V$, $P \in M$. Notice that the change of chart is the choice of a new origin. Such a structure supports a full geometrical development. See, for example, the monograph \cite{nomizu|sasaki:94}.

Because of our non-parametric perspective, we re-define the object to fit our def.~\ref{def:manifold} of a manifold. We consider a generalization of Weyl's axioms that allows for a family of (possibly infinite-dimensional) topological vector spaces instead of a single finite-dimensional vector space. Cf. \cite[p. 42]{lang:1995}.

\begin{definition}[Affine space] \label{def:affine-space}Let $M$ be a set and  let $B_\mu$, $\mu \in M$, be a family of toplinear spaces. Let $(\transport {\nu}{\mu})$, $\nu,\mu \in M$ be a family of toplinear isomorphisms $\transport {\nu} {\mu} \colon B_{\nu} \to B_{\mu}$ satisfying the \emph{cocycle conditions} AF0 below.
  \begin{description}
  \item[(AF0)] \label{item:affine-space-cocycle}  $\transport \mu \mu = I$ and $\transport \nu \rho \transport \mu \nu  = \transport \mu \rho$. 
  \end{description}
  We call $\transport \nu \mu$ the \emph{parallel transport} from $B_\nu$ onto $B_\mu$. 
  Consider a \emph{displacement} mapping
\begin{equation}
\displacement \colon  (\nu,\mu) \mapsto s_\nu(\mu)\in B_{\nu}
\end{equation}
defined on a subset of the product space $\domof \displacement \subset M \times M$. We assume
\begin{description}\label{desc:affineaxiom}
\item[(AF1)] \label{item:affineaxiom1}
  For each fixed $\nu$ the mapping $M_\nu \ni \mu \mapsto s_\nu(\mu) = \displacement(\nu,\mu)$ is 1-to-1.
\item[(AF2)] \label{item:affineaxiom2}
  $\displacement (\mu_1,\mu_2) + \transport {\mu_2}{\mu_1} \displacement (\mu_2,\mu_3) = \displacement (\mu_1,\mu_3)$.
  \end{description}
We will say that the structure $(M,(B_\mu)_{\mu \in M},(\transport {\nu}{\mu})_{\mu,\nu\in M},\displacement)$ is an \emph{affine space}.
\end{definition}

When the vector space does not depend on the point $\mu$, $B_\mu = B$, and the parallel transport maps are the identity, we recover Weyl's definition. Note that in our definition, the domain of $s_\mu$ map could be smaller than $M$, and the image could be smaller than the vector space.

Def.~\ref{def:affine-space}(AF2) with $\mu_1=\mu_3=\nu$ and $\mu_2=\mu$ becomes
\begin{equation}
  \displacement(\nu,\mu) + \transport \mu \nu \displacement(\mu,\nu) = 0 \ .
\end{equation}

Let us compute, where defined, the \emph{change-of-origin} map $s_\mu \circ s_\nu^{-1}$ in an affine space. At $\rho = s_\nu^{-1}(w)$, $w \in B_\nu$, it holds
\begin{equation}\label{eq:change-of-origin}
  s_\mu \circ s_\nu^{-1} (w) = s_\mu(\rho) = s_\mu(\nu) + \transport \nu \mu s_\nu(\rho) = s_\mu(\nu) + \transport \nu \mu w \ .
\end{equation}
Notice that the change-of-origin map extends to an affine map, a top linear isomorphism.

\subsection{\emph{Affine} manifold}

An affine space provides a family of candidates to charts $s_\nu \colon M_\nu \to B_\nu$, $\nu \in M$, that we could use as an atlas. We want to add a smoothness condition.

\begin{definition}[Affine manifold]\label{def:affine-manifold} 

Let $\left(M,(B_\mu)_{\mu \in M},(\transport {\nu}{\mu})_{\mu,\nu\in M},\displacement \right)$ be an affine space as in def.~\ref{def:affine-space}. \begin{description}
\item[(AF3)]We assume that for each $\nu$, the image set $s_\nu(M_\nu)$ is a neighborhood of 0 in $B_\mu$.
\end{description} That is, its interior $s_\nu(M_\nu)^\circ$ is an open set containing $s_\nu(\nu)=0$. Define the coordinates domains as $U_\nu = s_\nu^{-1} \left(s_\nu(M)^\circ\right)$, so that $(s_\nu, U_\nu, B_\nu)$ is a chart on $M$. Such a chart is said to have \emph{origin} $\nu$. Such charts are compatible, and the resulting manifold is, by definition, the \emph{affine manifold} associated with the given affine space.
\end{definition}

\begin{proof}
   All assumptions of def.~\ref{def:manifold} hold, but the fact that the domains in eq.~\eqref{eq:manifold} are both open. For each couple $\mu,\nu \in M$, we have defined $U_\mu = s_\mu^{-1} \left(s_\mu(M)^\circ\right)$ and $U_\nu = s_\nu^{-1} \left(s_\nu(M)^\circ\right)$. Use the change-of-origin equation \eqref{eq:change-of-origin} to see that
   \begin{multline}
     s_\nu(U_\nu \cap U_\mu) = s_\nu(M)^\circ \cap s_\nu \circ s_\mu^{-1} (s_\mu(M)^\circ) = \\  s_\nu(M)^\circ \cap (s_\nu(\mu) + \transport \mu \nu s_\mu(M)^\circ)
   \end{multline}
   is open. \smartqed
 \end{proof}
 
 Given the affine atlas, we can locally express the displacement mapping $\displacement$ with respect to any origin $\sigma \in M$. From the parallelogram law for the points $\sigma,\nu,\mu$, we write
 \begin{equation}
s_\sigma(\nu) + \transport \nu \sigma \displacement(\nu,\mu) = s_\sigma(\mu) \ .
\end{equation}
Because of the cocycle property \ref{def:affine-space}(AF0), we can write
\begin{equation}
    \displacement(\nu,\mu) = \transport \sigma \nu \left(s_\sigma(\mu) - s_\sigma(\nu)\right) \ ,
\end{equation}
which in turn implies the expression in the chart at $\sigma$ of the
displacement is given by
\begin{equation}\label{eq:expressio-displacement}
  \displacement_\sigma(u,v) =      \displacement(S^{-1}_\sigma(u),S^{-1}_\sigma(v)) = \transport \sigma {S^{-1}_\sigma(u)} (v - u) \ .
\end{equation}

The expression $\displacement_\sigma$ is affine in the second variable, hence the derivative
\begin{equation}\label{eq:d2}
 \left. \derivby t \displacement_\sigma(u,v+tk) \right|_{t=0} = D_2 \displacement_\sigma(u,v)[k] = \transport \sigma {S^{-1}_\sigma(u)} k \ .
\end{equation}
Using \eqref{eq:d2} and \eqref{eq:expressio-displacement}, we find that the expression of the displacement solves a differential equation, namely
\begin{equation}
\displacement_\sigma(u,v) = D_2 \displacement_\sigma(u,v)[v-u] \ .
\end{equation}

The previous equation provides a differential equation for the derivative in the first variable. The first partial derivative in the direction $h$ is
\begin{equation}
  \label{eq:d1}
  D_1\displacement_\sigma(u,v)[h] = D_1 D_2 \displacement_\sigma(u,v)[h,v-u] - D_2\displacement_\sigma(u,v)[h] \ .
\end{equation}

\begin{example}[Compact sample space, continuous positive densities] \label{ex:running-1} 
Let $\Omega$ be a compact metric space and $\mathcal B$ be its Borel $\sigma$-algebra. $m$ is a reference finite measure on $(\Omega,\mathcal B)$. $C(\Omega)$ is the Banach space of continuous functions on $\Omega$ with the uniform norm $\normat {\infty} f = \sup _{x \in \Omega} \avalof {fx)}$. Let $M$ be the set of \emph{strictly positive continuous probability densities}, that is, positive functions $f \in C(\Omega)$ such that $\int f \ dm = 1$. It is easy to see that $M$ is an open convex subset of the affine space $A = \setof{f \in C(\Omega)}{\int f \ dm =1}$. The previous set-up applies to the case where $\Omega$ is finite and $C(\Omega)$ is the space of all real vectors with indices in $\Omega$.

Many exciting ways exist to give the set $M$ an affine geometry. We introduce here three basic cases.

\begin{description}
\item[\emph{Flat:}] \label{item:flat} As $M$ is an open convex subset of an affine space $A$, it inherits the affine geometry of the larger space and the displacement vector is simply $\displacement(p,q) = q - p \in B_1$, where $B_1$ is the vector space parallel to $A$, $B_1 = \setof{f \in C(\Omega)}{\int f \ dm = 0}$. That is $s_p(q) = q - p$ and $s_p^{-1}(u) = u + p$ for all $u \in B_1$ such that $u + p > 0$. 
\item[\emph{Mixture:}] \label{item:mixture} If $B_p = \setof{f \in C(\Omega)}{\int f \ p \cdot dm = 0}$ and $\mtransport p q u = \frac p q u$, then we can define
  $s_p(q) = \frac q p -1$ for all $q \in M$. The parallelogram law follows from
  \begin{equation}
    \left(\frac q p -1\right) + \frac q p \left( \frac r q - 1\right ) = \frac r p - 1 \ .
      \end{equation}
The inverse mapping is $s^{-1}_p(u) = (1+u) \cdot p$ and it is defined on the open set $U_p = \setof{u \in B_p}{ u > -1}$. The expression of the displacement is $\displacement_p(u,v) = (1+u)^{-1}(v-u)$ and the partial derivatives are $D_1\displacement_p(u,v)[h] = -(1+u)^{-2}(1+v)h$ and $D_2\displacement_p(u,v)[k] = (1+u)^{-1}k$.
\item[\emph{Exponential:}] \label{item:exponential} With the same $B_p$ as above, we define $\etransport p q u = u - \int u \ q \cdot dm$ and define $s_p(q) = \log \frac q p - \int \log \frac q p \ p \cdot dm$. The parallelogram law follows from
  \begin{multline}
  \left(\log \frac q p - \int \log \frac q p \ p \cdot dm\right) + \\  \left(\log \frac r q - \int \log \frac r q \ p \cdot dm - \int  \left(\log \frac r q - \int \log \frac r q \ p \cdot dm\right)  \ dm \right) = \\ \left(\log \frac r p - \int \log \frac r p \ p \cdot dm\right)    
\end{multline}
The inverse of the chart is 
\begin{equation}
  s_p^{-1}(u) = \expof{u - K_p(u)} \cdot p \ , \quad K_p(u) = \log \int \euler^u \ p \ dm \ , \quad u \in B_p \ .
\end{equation}
The expression of the displacement is
\begin{equation}
  \displacement_r(u,v) = (v-u) - \int (v-u) \euler^{u - K_r(u)} \ r \ dm \ . 
\end{equation}
\end{description}
\end{example}

\subsection{\emph{Affine} bundle}

The specific form of the atlas defining the affine manifold allows the extension of the same atlas to define the affine bundle.
 
\begin{definition}[Affine bundle]\label{def:affine-bundle} Given the affine manifold $\M$ of def.~\ref{def:affine-manifold}, consider the vector bundle
 \begin{equation}
   SM = \setof{(\mu,v)}{\mu \in M, v \in B_\mu} \ .
 \end{equation}
 The equation
 \begin{equation}
   SM \times SM \ni ((\nu,u),(\mu,v)) \mapsto (s_\nu(\mu),\transport \mu \nu v) \in B_\nu \times B_\nu
 \end{equation}
 defines a displacement on the bundle. For each $\nu$ define the chart
 \begin{equation}\label{eq:charts-bundle}
     s_\nu \colon SM \ni (\mu,v) \mapsto (s_\nu(\mu), \transport \mu \nu v) \in B_\nu \times B_\nu 
 \end{equation}
 to define the affine bundle $S\M$ as a manifold. Equivalently, we can say that $S\M$ is a linear bundle with trivialization
 \begin{equation}
     s_\nu \colon (\mu,v) \mapsto (s_\nu(\mu), \transport \mu \nu v) \ .
 \end{equation}
\end{definition}

If we define the velocity as follows, the affine bundle is a convenient expression of the tangent bundle of the affine manifold.
 
\begin{definition}[Velocity] The \emph{velocity} of the smooth curve $t \mapsto \gamma(t)$ of the affine manifold is the curve $t \mapsto (\gamma(t), \velocity \gamma(t))$ of the affine bundle whose second component is
\begin{equation}\label{eq:velocity-1} \velocity \gamma(t) = \lim _{h \to 0} h^{-1} (s_{\gamma(t)}(\gamma(t+h)) = \left. \derivby h s_{\gamma(t)}(\gamma(t+h)) \right|_{h=0} \ .
\end{equation}
\end{definition}

By eq.~\eqref{eq:charts-bundle} and def.~\ref{def:affine-space}(AF2) applied to the points, the expression in the chart centered at $\nu$ of $\velocity \gamma(t)$ is
  \begin{multline}\label{eq:velocity-in-chart}
\transport {\gamma(t)} \nu \velocity \gamma(t) =  \lim_{h \to 0} h^{-1} \transport {\gamma(t)} \nu s_{\gamma(t)}(\gamma(t+h)) = \\ \lim_{h \to 0} h^{-1} \left(s_\nu(\gamma(t+h)) - s_\nu(\gamma(t)\right) = \derivby t s_\nu(\gamma(t))\ ,
  \end{multline}
and, conversely,
  \begin{equation}\label{eq:star}
    \velocity \gamma(t) = \transport {\nu} {\gamma(t)} \derivby t s_\nu(\gamma(t)) \ .
  \end{equation}
  
  \begin{definition}[Integral curve and flow of a section] Let $F$ be a section of the affine bundle, that is, $(\mu,F(\mu)) \in S \M$. An \emph{integral curve} of the section $F$ is a curve $t \mapsto \gamma(t)$ such that $\velocity \gamma(t) = F(\gamma(t))$. A \emph{flow} of the section $F$ is a mapping
\begin{equation}
    M \times I \ni (\nu,t) \mapsto \Gamma_t(\nu)
\end{equation}
such that for each $\nu$ the curve $t \mapsto \Gamma_t(\nu)$ is an integral curve and $\gamma(0,\nu) = \nu$. 
\end{definition}

\begin{definition}[Auto-parallel curve] A curve $I \colon t \mapsto \gamma(t)$ is \emph{auto-parallel} in the affine bundle if 
\begin{equation}\label{eq:auto-parallel-1}
    \velocity \gamma(t) = \transport {\gamma(s)} {\gamma(t)} \velocity \gamma(s) \quad s,t \in I \ .
  \end{equation}
\end{definition}

\begin{proposition}
  The following conditions are equivalent.
  \begin{enumerate}
  \item The curve $\gamma$ is autoparallel.
  \item The expression of the curve in each chart is affine.
  \item For all $s,t$  
\begin{equation}\label{eq:auto-parallel-2}
  \gamma(t) = S^{-1}_{\gamma(s)}\left((t-s) \velocity \gamma(s)\right) \ .    
\end{equation}
\end{enumerate}
\end{proposition}

\begin{proof}
  From eq.~\eqref{eq:auto-parallel-1} we see that the velocity is constant in each chart, $\transport {\gamma(t)} \nu \velocity \gamma(t) = \transport {\gamma(s)} \nu \velocity \gamma(s)$, and, by \eqref{eq:velocity-in-chart}, we have that $\derivby t s_\nu(\gamma(t) = \derivby s s_\nu(\gamma(s)$. Hence $t \mapsto s_\nu(t)$ is an affine curve, $s_\nu(\gamma(t)) - s_\nu(\gamma(s)) = (t - s) \derivby s s_\nu(\gamma(s)) = (t-s) \transport {\gamma(s)} \nu \velocity \gamma(s)$. Now put $\nu = \gamma(t)$ to get $s_{\gamma(s)}(\gamma(t)) =  (t-s) \velocity \gamma(s)$ hence \eqref{eq:auto-parallel-2}. And conversely. \smartqed
\end{proof}

\begin{proposition} The affine bundle is an affine manifold for the displacement mapping
 \begin{equation}
     ((\nu,v),(\mu,w)) \mapsto \left(\displacement(\mu,\nu), \transport \mu \nu w - v\right) \in B_\nu \times B_\nu
   \end{equation}
   and the transports $\transport \nu \mu \times \transport \nu \mu$.
 \end{proposition}

 \begin{proof}
   Check all the properties AF0--3. \smartqed
 \end{proof}
 
 \begin{definition}[Acceleration] Consider the curve $t \mapsto \mu(t)$ with velocity $t \mapsto \velocity \mu(t)$. The acceleration $t \mapsto \acceleration \mu(t)$ is the velocity $t \mapsto (\mu(t),\velocity \mu(t))$.
 \end{definition}

 \begin{equation}
     (\velocity \mu(t),\acceleration \mu(t)) = \lim_{h \to 0} h^{-1} s_{\mu(t),\velocity \mu(t)}(\mu(t+h),\velocity \mu(t+h)) \ .
   \end{equation}
Especially, for all $\mu \in M$,
\begin{equation}
  \label{eq:starstar}
  \acceleration \mu(t) = \transport \mu {\mu(t)} \derivby t \transport {\mu(t)} \mu \velocity \mu(t) \ .
\end{equation}
From this equation, it follows that
 \begin{proposition}
   A curve with 0 acceleration is auto-parallel.
 \end{proposition}
 
\begin{example} [Sequel of Example~\ref{ex:running-1}] \label{ex:running-2}
 If $t \mapsto p(t)$ is a curve, let us compute the velocity by implicitly assuming enough smoothness to justify all the computations. The velocity in the mixture manifold is
\begin{multline}
  \velocity p(t) = \lim_{h \to 0} h^{-1} \mtransport {p(t+h)}{p(t)} \displacement(p(t),p(t+h)) = \\ \lim_{h \to 0} h^{-1} \left(\frac {p(t)}{p(t+h)} \left(\frac{p(t+h)}{p(t)}-1\right)\right) = \frac {\dot p(t)}{p(t)} \ .
\end{multline}
The velocity in the exponential manifold is
\begin{multline}
  \velocity p(t) = \lim_{h \to 0} h^{-1} \etransport {p(t+h)}{p(t)} \displacement(p(t),p(t+h)) = \\ \lim_{h \to 0} h^{-1} \left(\log \frac {p(t+h)}{p(t)} - \int \log \frac {p(t+h)}{p(t)} \ p(t) \ dm\right)  = \frac {\dot p(t)}{p(t)} \ .
\end{multline}
Remarkably, the expression of the velocity is the same in both cases. In the statistical literature, the quantity $\velocity p(t)$ is called the \emph{Fisher's score} of the 1-dimensional model $p(t)$. The exponential velocity for a curve of the form of a Gibbs model $p(t) \propto \euler^{\alpha(t) U} \cdot p$, that is $p(t) = \euler^{\alpha(t) U - \psi(t)} \cdot p$, is
\begin{equation}
  \velocity p(t) = \derivby t \left(\alpha(t) U - \psi(t)\right) = \dot \alpha(t) U - \dot \psi(t) = \dot \alpha(t) \left(U - \int U \ p(t) \ dm\right) \ .
\end{equation}
In this case, the quantity $\velocity p(t)$ is seen as $\dot \alpha(t)$ times the \emph{fluctuation} $\left(U - \int U \ p(t) \ dm\right)$. Cf. any textbook of Statistical Physics, for example \cite{landau|lifshits:1980}.

Let us compute the acceleration in both cases. In the mixture case,
\begin{equation}\label{eq:m-acceleration}
  \acceleration p(t) = \mtransport {p}{p(t)} \derivby t \mtransport {p(t)} p \velocity p(t) = 
  \frac{p}{p(t)} \derivby t \frac {p(t)} p \frac {\dot p(t)} {p(t)} = \frac {\ddot p(t)}{p(t)} \ . 
\end{equation}
In the exponential case,
\begin{multline}\label{eq:e-acceleration}
  \acceleration p(t) = \etransport {p}{p(t)} \derivby t \etransport {p(t)} p \velocity p(t) = 
  \etransport p {p(t)} \derivby t \left(\frac{\dot p(t)}{p(t)} - \int \frac{\dot p(t)}{p(t)} \ p \ dm\right) = \\
  \frac {\ddot p(t)}{p(t)} - \left(\frac {\dot p(t)}{p(t)}\right)^2 - \int \left(\frac {\ddot p(t)}{p(t)} - \left(\frac {\dot p(t)}{p(t)}\right)^2\right) \ p(t) \ dm = \\
  \frac {\ddot p(t)}{p(t)} - \left(\frac {\dot p(t)}{p(t)}\right)^2 + \int \left(\frac {\dot p(t)}{p(t)}\right)^2 \ p(t) \ dm\ . 
\end{multline}
For the Gibbs model above, the exponential acceleration is proportional to the velocity, namely
\begin{equation}
  \acceleration p(t) = \ddot \alpha(t) \left(U - \int U \ p(t) \ dm\right) = \frac {\ddot \alpha(t)}{\dot \alpha(t)} \velocity p(t) \ .
\end{equation}

The auto-parallel curves in the mixture geometry are of the form $\gamma(t) = \gamma(0) + \dot \gamma(0)t = (1 + \velocity \gamma(0)) \gamma(0) = (1-t) \gamma(0) + t \gamma(1)$. The last expression explains the name. In the exponential geometry, the form of the auto-parallel curve follows from Eq.~\eqref{eq:auto-parallel-2}, $\gamma(t) = S^{-1}_{\gamma(0)}(t \velocity(\gamma(0)) = \euler^{t \velocity \gamma(0) - K_{\gamma(0)}(\velocity \gamma(0))} \cdot \gamma(0)$, that is, it is an exponential family. The auto-parallel interval is the Hellinger arc $\gamma(t) \propto \gamma(0)^{1-t}\gamma(1)^t$.
\end{example}

 \begin{definition}[Duality] Let be given two affine manifolds on the same base set $M$, $\M_i = (M,(B^i_\mu)_{\mu \in M},(\prescript{i}{}{\transport {\nu}{\mu}})_{\mu,\nu \in M},\prescript{i}{}\displacement)$, $i = 1,2$, and let be given for each $\mu \in M$ a duality pairing
   \begin{equation}
     B^1_\mu \times B^2_\mu \ni (u_1,u_2) \mapsto \scalarat \mu {u_1}{u_2} \ .
        \end{equation}
The affine manifolds $\M_1$ and $M_2$ are in duality if for all $\mu,\nu \in M$, $u \in B_{\mu}$, $v \in B^2_{\nu}$, it holds
        \begin{equation}
          \scalarat \mu {u}{\prescript{2}{}{\transport \nu \mu v}} = \scalarat \nu {\prescript{1}{}{\transport \mu \nu u}}{v} \ .
        \end{equation}
      \end{definition}

\begin{example}[Duality. Follows from Examples~\ref{ex:running-1} and \ref{ex:running-2}] \label{ex:running-3}
In the present case, the mixture and the exponential fibers are equal, $\prescript{\text{m}}{}B_p=\prescript{\text{e}}{}B_p=B_p$, and there is a separating pairing $\scalarat p u v = \int u \ v \ p \ dm$. The mixture affine manifold and the exponential affine manifold are dual. Let us check this. For $u \in B_p$ and $v \in B_q$
      \begin{multline}
        \scalarat q {\mtransport p q u} v = \int \frac p q u \ v \ q \ dm = \int u \ v \ p \ dm = \\ \int u \left(v - \int vp \, dm\right) \ p \ dm = \scalarat p y {\etransport q p v} \ .
      \end{multline}
\end{example}

 \begin{definition}[Gradient] Consider a $M$ which is base of two dual affine manifolds $\M_1$ and $\M_2$. A real function $\phi$ on $\M_1$ has a gradient $\Grad \phi$ if $\Grad \phi$ is a section of the affine bundle $S\M_2$ and for all smooth curve $t \mapsto \gamma(t) \in M$ it holds
   \begin{equation}
     \label{eq:grad}
  \derivby t \phi(\gamma(t)) = \scalarat {\sigma} {\Grad \phi (\gamma(t))}{\velocity \gamma(t)} \ .
   \end{equation}
 \end{definition}

 As defined above, the gradient is related but does not coincide with the \emph{natural gradient} of S.I.~Amari. Let us express the gradient $\Grad$ in a chart with origin $\sigma$ with the ordinary gradient $\nabla_\sigma$. In the 1-chart, it holds $\phi \circ \gamma(t) = \left(\phi \circ  \prescript{1}{} S^{-1}_\sigma\right) \circ \left( \prescript{1}{} s_\sigma \circ \gamma(t)\right) = \phi_\sigma \circ \left( \prescript{1}{} s_\sigma \circ \gamma(t)\right)$, where $\phi_\sigma \colon U_\sigma \to \R$ is the expression of $\phi$, so that
 \begin{multline}
   \derivby t \phi \circ \gamma(t) = d\phi_\sigma \left[\derivby t \left( \prescript{1}{} s_\sigma \circ \gamma(t)\right)\right] = d\phi_\sigma \left[\prescript{1}{}{\transport {\gamma(t)} \sigma} \velocity \gamma(t) \right] = \\
   \scalarat \sigma {\nabla_\sigma \phi_\sigma}{\prescript{1}{}{\transport {\gamma(t)} \sigma} \velocity \gamma(t)} = \scalarat {\gamma(t)} {\prescript{2}{}{\transport \sigma {\gamma(t)}}\nabla_\sigma \phi_\sigma}{ \velocity \gamma(t)} \ ,
   \end{multline}
   where we have used \eqref{eq:star} and $\nabla_\sigma$ denotes the gradient computed in the duality of $\prescript{1}{}B_\sigma$ with $\prescript{2}{}B_\sigma$. In conclusion, the gradient of $\phi \colon M$ equals the gradient of $\phi_\sigma \colon \prescript{1}{}B_\sigma$, 
   \begin{equation}
     \Grad \phi(\mu) = \prescript{2}{}{\transport \sigma \mu}\nabla_\sigma \phi_\sigma \circ \prescript{1}{}s_\sigma(\mu) = \nabla_\mu \phi_\mu (0) .
   \end{equation}
   
   \begin{example}[Gradient of the entropy. Follows from Examples \ref{ex:running-1},\ref{ex:running-2}, and \ref{ex:running-3}]\label{ex:running-4}
     The entropy is $\entropyof q = - \int q \log q \ dm$. The expression of the entropy of $q$ at $p$ is $\Entropy_p(v) = - \int (1+v) p \log ((1+v)p) \ dm$, $v \in B_p$. Let $t \mapsto v(t)$ be a smooth curve in $B_q$ with $v(0) = 0$. We have
   \begin{multline}
     \left. \derivby t \Entropy_p(v(t)) \right|_{t=0} = \left. \int (1 + \log ((1+v(t))p) \dot v(t) \ p \ dm \right|_{t=0} = \\ \left. \scalarat p {- \log((1+v(t))p)+ \int \log((1+v(t) )p) \ dm}{\dot v(t)} \right|_{t=0} = \\ \scalarat p {-\log q - \entropyof q}{\dot v(0)} \ .  
   \end{multline}
   In conclusion $\Grad_\text{m} \entropyof q  = - \log q + \entropyof q$. The same result holds for $\Grad_\text{e}$.

   In the exponential geometry, as $\velocity \gamma(t) = \derivby t \log \gamma(t)$, the gradient flow equation for the entropy $\velocity \gamma(t) = - \Grad \entropyof {\gamma(t)}$ becomes
   \begin{equation}
     \derivby t \log \gamma(t) = \log \gamma(t) - \int \gamma(t) \log \gamma(t) \ dm
   \end{equation}
   If we compute the acceleration, we find the remarkable result $\acceleration \gamma(t) = \velocity \gamma(t)$. The curve $\gamma(t) = \euler^{a(t)v - K_p(v)} \cdot p$ has $\velocity \gamma(t) = \dot a(t) (v - \int v \ \gamma(t) \ dm)$, and acceleration $\acceleration \gamma(t) = \ddot a(t) (v - \int v \ \gamma(t) \ dm) = \frac {\ddot a(t)}{\dot a(t)} \velocity \gamma(t)$. We want $\dot a(t) = \ddot a(t)$, that is $a(t) = c\euler^t + b$.   
 \end{example}

 \begin{example}[Differentiable densities] In this example, we discuss how to construct an affine manifold for differentiable densities. For example, such a class of densities is used in functional data analysis. See, for example, assumption A1 in \cite{petersen|mueller:2016}. Let $\Omega$ be a bounded domain in $\R^d$ and let $m$ denote the Lebesgue measure on $\Omega$. Let $C_b^n(\Omega)$ be the Banach space of continuous functions on $\overline \Omega$ and $n$-times differentiable on $\Omega$, with bounded partial derivatives. The norm is $\normat {d} f = \normat \infty f + \sum_{i=1}^d \normat \infty {\partial_i f}$. Let $M_1$ be the affine subspace where $\int f \ dm = 1$. The convex subset
 \begin{equation}
     M = \setof{f \in C_b^n(\Omega)}{f > 0. \int f \, dm = 1}
 \end{equation}
 is an open subset of $M_1$. It is the set of positive differentiable density functions on $\Omega$.
   
 In \cite{hivarinen:2005} we find the divergence
 \begin{equation}
  \KH{q}{p} = \frac 12 \int \avalof{\nabla q - \nabla p} ^2 \, p dm = \frac 12 \normat p {\nabla q - \nabla p} ^2 \ .   
 \end{equation}
 
 Let us show that it is the squared norm of a displacement involving derivatives. For $q,p \in M$, define
   \begin{equation}
     \label{eq:hivarinen}
       \displacement (p,q) = \nabla \log \frac q p = \frac {\nabla q} q - \frac {\nabla p} p
   \end{equation}
Let us define the fibers as
\begin{equation}
   \bm B_p = \setof{\bm u \in C_b^{n-1}(\Omega;\R^d)}{\bm u = \nabla U, U \in C_b^n(\Omega), \int U \ p \ dm = 0} \ .
\end{equation}
All the fibers are equal, and we can assume trivial transports. The semi-group property is clear. Let us compute the inverse of a chart. If
\begin{equation}
    B_p = \setof{U \in C_b^{n}(\Omega)}{ \int U \ p \ dm = 0} \ ,
    \end{equation}
    then $s_p(\bm u)^{-1} = q$ if $\log \frac q p = U \in B_p$ and $\bm u = \nabla U$. It follows that, in the notations of Example~\ref{ex:running-1}, we can write
\begin{equation}
    s_p^{-1}(\bm u) = \euler^{U - K_p(U)} \cdot p \ , \quad \bm u = \nabla U \ .
\end{equation}

The bundle on $M$ with fibers $\bm B_p$ has an interesting inner product introduced by \cite{otto:2001}.
\begin{equation}
    \bm\langle \bm u , \bm v \bm\rangle _ p = \int \nabla U \cdot \nabla V \ p \ dm \ .
    \end{equation}
We do not further develop this example but refer to the relevant chapter in this Handbook. See \cite{lott:2008calculations}, and the application to IG in \cite{ogouyandjo|wadagni:2020}.

    \end{example}
 
\section{Non-parametric \emph{statistical} affine manifolds}\label{sec:nonparametric-statistical-affine-manifolds}

Let us focus now on affine manifolds whose base set is the vector space of signed measures $\smeasures$ on a given measurable space $(\Omega,\mathcal F)$. We refer to \cite{ay|jost|le|schwachhofer:2017} for a full treatment of Information Geometry in $\mathcal M$. We refer to \cite{rudin:1987-3rd} for basic Functional Analysis and Measure theory. More advanced textbooks are  \cite{aliprantis|border:2006} and \cite{brezis:2011fasspde}. We recall that the set of finite measures on a measurable space $(X, \mathcal X)$ is a lattice and a convex pointed cone; see \cite[Ch.~10-11]{aliprantis|border:2006}. 

A \emph{signed measure} is the numerical difference of two finite measures, $\mu = \mu_1 - \mu_2$. There exists a unique minimal decomposition $\mu = \mu_+ - \mu_-$, the \emph{Jordan decomposition}, where the positive part $\mu_+$ and the negative part $\mu_-$ have disjoint supports, that is, $\mu_- \wedge \mu_+  = 0$. The measure $\avalof \mu = \mu_+ + \mu_-$ is the absolute value of $\mu$ and $\mu \mapsto \tvnorm \mu = \int d\avalof \mu$ is the \emph{total variation norm}. The affine subspace of signed measures with total mass $\tau$ is $\smeasures_\tau$. In particular, we are interested in the affine subspace $\smeasures_1 = \setof{\mu \in \smeasures}{\mu(X)=1}$ and in the vector subspace $\smeasures_0=\setof{\mu \in \smeasures}{\mu(X)=0}$. The affine subspace $\smeasures_1$ contains the closed convex set of probability measures $\probabilities$, while $\smeasures_0$ is its tangent subspace.

The integral induces a natural pairing on $\smeasures \times \mathcal L^\infty$, that is, $(\mu,f) \mapsto \int f \ d\mu$. The space $\smeasures_0$ is closed in the weak topology. We have $\scalarof \mu f \leq \normat \infty f \tvnorm f$, and the weak convergence induced on $\mathcal L^\infty$ implies the point-wise convergence. 

For each curve $t \mapsto \mu(t)$ in $\smeasures_1$ and for any topology on $\smeasures$ compatible with the operations of vector space such that $\smeasures_0$ is closed, it holds $\derivby t \mu(t) = \dot\mu(t) \in \smeasures_0$ provided the curve is differentiable at $t$, that is
\begin{equation}
    \lim _ {h \to 0} \left(h^{-1}(\mu(t+h) - \mu(t)) - \dot\mu(t)\right) = 0
\end{equation}. 

The following special case is of high interest. Assume that the curve is smooth and stays in $\probabilities$. In such a case, $\mu_s(A) = 0$ implies that $s$ is a minimum point of $t \mapsto \mu_t(A)$, hence $\dot\mu_s(A) = 0$ so that the absolute continuity $\dot\mu_t << \mu_t$ holds. The \emph{Fisher's score} $\velocity \mu_t = \frac{d\dot\mu_t}{d\mu_t} \in L^1(\mu_t)$ is defined for smooth statistical models. Notice that $\int \velocity \mu_t \ d\mu_t = \dot \mu_t(X) = 0$.

For any topological vector space on $\smeasures$ such the the mapping $\mu \mapsto \int f \ d\mu = \scalarof f \mu$ is continuous, the \emph{Fisher-Rao equation} holds,
\begin{multline}
  \derivby t \int f \ d\mu_t = \int f \ d\dot\mu_t =  \int \left(f-\int f \ d\mu_t\right) \
  d\dot\mu_t = \\ \int \left(f-\int f \ d\mu_t\right) \ \frac{d\dot\mu_t}{d\mu_t} \ d\mu_t = \int \left(f-\int f \ d\mu_t\right) \ \velocity \mu_t \ d\mu_t = \\ \scalarat {\mu_t} {f-\int f \ d\mu_t} {\velocity \mu_t} \ .
\end{multline}
Notice that $\velocity \mu_t \in L^1(\mu_t)$, possibly a different space for each $t$. 

The existence of a common dominating measure is an option to be considered, that is, $\mu(t) = p_t \cdot \mu$. Assume moreover $p(t) > 0$ $\mu$-almost surely. In such a case,
\begin{equation}
\velocity \mu(t) = \frac{\dot p_t \cdot \mu}{p_t \cdot \mu} = \frac {\dot p_t}{p_t} = \derivby t \log p_t \quad \text{$\mu$-a.s.}     
\end{equation}

The basic scheme above has various qualifications in statistical application. A few examples, not further developed in this chapter, are listed below. The cases of interest are in the following sections.

\begin{example}[$\probabilities$ on a measurable space with the total variation norm]\label{ex:TV} Let the base manifold be $\M = \smeasures_1$ and let the fibers be $B_\mu = \smeasures_0$ for all $\mu \in \probabilities$. In this case, the affine structure is simply the affine structure of the affine space $\smeasures_1$. Recall that $\smeasures$ is a Banach space for the total variation norm. This space has even more structure. It is a \emph{Dedekind complete Banach lattice} for the total variation norm; See \cite[Ch.~8--9]{aliprantis|border:2006}. Any bounded above set (respectively, below) in the natural order of signed measures has an upper (respectively, a lower) limit. In particular, $\mu \vee \nu$ and $\mu \wedge \nu$ exist, and open intervals are an open set. We define a displacement by $\displacement(\nu,\mu) = \mu - \nu$, that is, $s_\nu \colon \mathcal P \ni \mu \mapsto \mu-\nu \in M_0$. For each $\nu$ the mapping $\mu \mapsto \vectorat \nu \mu$ is 1-to-1 and the image is the set of all $\xi \in \smeasures_0$. The dual space is the space of bounded measurable functions with the topology of bounded convergence. This space supports many special structures, particularly the geometry of scores and an embedded Riemannian geometry. This construction has been detailed in the monograph \cite{ay|jost|le|schwachhofer:2017}.
\end{example}

\begin{example}[Positive probability densities with the $L^1(m)$ topology]\label{ex:L1}
Let us take $M = \setof{\rho \in L^1(m)}{\rho \geq 0, \int \rho \ dm =1}$ and $B = L^1_0(m)$. These are subsets of the cases in ex.~\ref{ex:TV}. The mapping $s_\eta \colon M \ni \rho \mapsto \rho - \eta \in B$ is 1-to-1. Consider that the image of $s_\eta$ has an empty interior in the $L^1(m)$ topology if the sample space is not finite. This remark is a counter-example showing that no trivial construction is feasible in infinite dimensions.
\end{example}   

\begin{example}[Continuous probability densities for a Borel probability measure on a compact space] This has already been used in the previous section as an introductory example. We assume $\Omega$ is metric and compact. This assumption is satisfied in the finite state space case. Most of the non-parametric Information Geometry literature rests on this assumption. The tutorial  \cite{pistone:2020-NPCS} is a presentation along such lines. We consider a reference Radon measure $m$, a positive, hence continuous, linear functional on $C(\Omega)$. We construct an affine space with base $M = \setof{p \in C(\Omega)}{ p > 0, \int p \ d\mu = 1}$ by setting $\displacement p q = q-p$. In this case, $s_p \colon M \to U_p = \setof{u \in C(\Omega)}{u + p > 0, \int u \ dm =0}$ The space of vectors is $B = \setof{u \in C(\Omega)}{\int u \ dm =0}$. We want to show that $U_p$ is open in $B$. In fact, if $u \in U_p$, there is an $\epsilon > 0$ such that $u -\epsilon > p$, $\epsilon = \min (u + p)$. If we define $B_p = \setof{u \in C(\Omega)}{\int u \ p\ dm = 0}$, then the velocity is $\velocity p_t = \dot p_t / p_t \in B_{p_t}$.
\end{example}     

\begin{example}[Arens-Eells] 
Let $(\Omega,d)$ be a metric space with Borel measurable space $(\Omega, \mathcal B)$. Let $B$ be the vector space of all  signed measures $\xi$ of the form
\begin{equation}
  \xi = \int (\delta_x - \delta_y) \ a(dx,dy) \quad \text{with $a$ a signed measure on $\Omega \times \Omega$} \ .
\end{equation}
That is, $f \in \mathcal L^\infty(\mathcal B)$,
\begin{equation}
\int f \ d\xi = \int (f(x) - f(y)) \ a(dx,dy)  ,     
\end{equation}
in particular, $\xi(\Omega) = 0$. 
The Arens-Eells norm is 
\begin{equation}
 \xi \mapsto \sup \setof{\int d(x,y) a(dx,dy) \,}{ \,\xi = \int (\delta_x - \delta_y) \ a(dx,dy)}
\end{equation}
is its Arens-Eells norm. Take $M$ as a maximal set of probability measures such that $\nu,\mu \in M$ implies $\mu - \nu \in B$. This is clear on a finite state space; otherwise, see \cite[Ch.~15]{aliprantis|border:2006} and \cite[sec.~3.1]{weaver:2018-la-2nd-ed}.
\end{example}

\subsection{\emph{Exponential} affine manifold} \label{sec:exponential-affine-manifold}

This section sets the exponential affine geometry already described in Examples \ref{ex:running-1} to \ref{ex:running-3} in a larger framework. There are many feasible choices for the Banach spaces to act as coordinated spaces. However, not all settings will work, and we consider it essential to spell out general requirements. The following sections will describe two specific choices of Banach spaces.

Let the base set $M$ be the set of all probability measures equivalent to a reference $\sigma$-finite measure $m$. That is, $\mu = p \cdot m$ and $p > 0$ $m$-a.s. The maximum possible base set is the set of all positive $m$-desities. Many models we will define below apply to a smaller set of densities. 

As all the measures $p \cdot m$, $p \in M$, are equivalent,  the vector spaces of $p \cdot m$-equivalent classes of real random variables are equal,  $L^0(p \cdot m) = L^0(m)$. $L^0(m)$ is a topological vector space with the convergence in $m$-measure. However, the bundle $M \times L^0(m)$ seems too big to support an exponential geometric structure because the random variable $U$ and $\euler ^ U$ will not always be $m$-integrable unless the state space is finite. 

Let us assume that $M$ is a set of positive $m$-probability densities possibly smaller than the maximal one. Each probability density $p\in M$ defines its own Banach space of integrable random variables $L^1(p) = L^1(p \cdot m)$. In general, the spaces are not equal for different densities. A sufficient condition for $L^1(p) = L^1(q)$ for all of all $p,q \in M$ is that the density ratio is bounded above and below for all couples, $k \leq q/p \leq K$ for some $0 < k \leq K$. In fact,
\begin{equation}\label{eq:aboveandbelow}
  \int \avalof f \ q \ dm = \int \avalof f \frac q p \ p\ dm \leq
  K \int \avalof f \ p \ dm \ .
\end{equation}

We want a set of integrable random variables for all densities in our base set $M$. That is, we look for a topological vector space $B$ of $m$-classes of random variables such that
\begin{equation}\label{eq:B-space}
  B \hookrightarrow \cap _ {p \in M} L^1(p) \hookrightarrow L^0(m) \ .
\end{equation}

\begin{example}[Compact sample space, continuous densities] If $\Omega$ is compact and $p,q$ are assumed to be continuous, then a bound in Eq.~\eqref{eq:aboveandbelow} always exists.
\end{example}

\begin{example}[Bounded random variables] Clearly, that $\cap_{p \in M} L^1(p) = L^\infty(m)$, so that one could restrict the attention to a bounded random variable only. Such an assumption does not seem to produce a model with sufficient applicability. Many useful random variables are unbounded.
\end{example}

Given a vector space $B$ satisfying Eq.~\eqref{eq:B-space}, we define the family of spaces
\begin{equation}
  \label{eq:B-p-spaces}
  B_p = \setof{u \in B}{ \int u \ p \ dm = 0}
\end{equation}
together with the transports
\begin{equation}\label{expart}
  \etransport p q \colon B_p \ni  \mapsto u - \expectat q u \in B_q \ .
\end{equation}

Notice that the transports are composed correctly in a cocycle,
\begin{equation}
  \etransport p p u = u \quad \text{and} \quad \etransport p r = \etransport q r \etransport p q
\end{equation}

Moreover, we assume $M$ is such that for each couple $p,q \in M$ the $\log$-ratio is well defined in $B$, $\log \frac q p \in B$. The displacement
\begin{equation}
  \displacement \colon (p,q) \mapsto \log \frac q p  - \expectat p {\log \frac q p} \in B_p \subset L^1_0(p)  
\end{equation}
defines an \emph{affine space}. The parallelogram law holds,
\begin{multline}
\left(\log \frac q p  - \expectat p {\log \frac q p}\right) + \etransport q p \left(\log \frac r q  - \expectat q {\log \frac rq}\right) = \\ \log \frac q p  - \expectat p {\log \frac q p}  + \log \frac r p - \log \frac q p - \expectat p {\log \frac r p - \log \frac q p} \\ \log \frac r p - \expectat p {\log \frac rp}
\end{multline}

Finally, we want to check that the charts
\begin{equation}
  s_p \colon M \ni q \mapsto \log \frac q p  - \expectat p {\log \frac q p}
\end{equation}
with
\begin{equation}
  s_p^{-1} \colon u \mapsto \euler^{u - K_p(u)} \cdot p \ , \quad K_p(u) = \expectat p {\euler^u} \ .
\end{equation}
are such that the image of each $s_p$ is a neighborhood of 0 in $B_p$. 

The velocity \eqref{eq:velocity} of the curve
\begin{equation}
  t \mapsto p(t) = \euler^{u(t) - K_p(u(t))} \cdot p
  \end{equation}
  is
\begin{equation}
  \label{eq:velocity-exponential}
  \velocity p(t) = \derivby t \log p(t) = \dot(u(t)) - \derivby t K(u(t)) = \etransport p {p(t)} \dot u(t) \ ,
\end{equation}
where the last equality follows from a well-known property of the derivative of the cumulant generating function. A classical reference for exponential families is \cite{brown:86}.

The auto-parallel curves are the exponential families
\begin{equation}
  t \mapsto \euler^{tu-K_p(tu)} \cdot p \quad u \in B_p \ .
\end{equation}

Let us compute the acceleration with Eq.~\eqref{eq:starstar}
\begin{multline}\label{eq:exponential-acceleration}
  \acceleration p(t) = 
  \etransport p {p(t)} \derivby t \etransport {p(t)} p \velocity p(t) = \\ \etransport p {p(t)} \derivby t \left(\frac{\dot p(t)}{p(t)} - \int \frac{\dot p(t)}{p(t)} \ p \ bm\right) = \\ \etransport p {p(t)} \left(\frac {\ddot p(t)}{p(t)} - \left(\frac{\dot p(t)}{p(t)}\right)^2 - \int \frac {\ddot p(t)}{p(t)} - \left(\frac{\dot p(t)}{p(t)}\right)^2 \ p \ dm\right) = \\ \frac{\ddot p(t)}{p(t)} - \left(\frac{\dot p(t)}{p(t)}\right)^2 + \int \left(\frac{\dot p(t)}{p(t)}\right)^2 \ p(t) \ dm \ .
\end{multline}

\begin{example}
We have seen that $\acceleration p(t) = 0$ implies $p(t) = \euler^{tu - K_p(tu)} \cdot p$, $p(0) = p$, $u \in B_p$. In a different time scale, $q(t) = \exp(a(t)v - K_p(a(t)v) \cdot p$, $\dot a > 0$, $v \in B_p$. We have
\begin{equation}
  \velocity q(t) = \dot a(t) \left(v - \int v \ q(t)  \ dm\right) \quad \acceleration q(t) = \ddot a(t) \left(v - \int v \ q(t) \ dm\right)
\end{equation}
so that
\begin{equation}
    \acceleration q(t) = \frac {\ddot a(t)}{\dot a(t)} \velocity q(t) = \derivby t \log \dot a(t) \velocity q(t) \ . 
\end{equation}
In particular, $a(t) = -1/t$, yields the equation $t \acceleration q(t) + 2 \velocity q(t) = 0$. A classical reference for Boltzmann-Gibbs is \cite{landau|lifshits:1980}. \end{example}

\section{Banach spaces of random variables as coordinate spaces}\label{sec:Banach}

In our definition of affine statistical manifold, we consider a set of probability measures $M$ and, for each $\nu \in M$, a mapping $s_\nu \colon M \to B_\nu$, where $B_\nu$ is a topological vector space, a \emph{toplinear space}. The specific needs of the modeling dictate the choice of the displacement map, the only restriction being the parallelogram law. The choice of the family of toplinear spaces $B_\nu$ and the family of parallel transport $\transport \nu \mu$, $\nu,\mu \in M$, could be challenging. There are two topological requirements:

\begin{enumerate}
    \item $\transport \nu \mu$ is an toplinear isomorphism of $B_\nu$ onto $B_\mu$, and
    \item The set image of $M$ with $s_\nu$, the set of all coordinates, is open in $B_\nu$.
\end{enumerate}

The previous requirements are strict, but there is a wide choice of possible setups. The number of possible variations is too large to provide an exhaustive list. Functional analysis offers a toolbox for adapting to each specific case, just as other fields do. For the reader's convenience and future reference, we provide the displacement maps we introduced in the previous sections in Table~\ref{tab:displacements}.

\begin{table}[t]
\centering
\begin{tabular}{| c | c | c |}
\hline
\quad  \text{name} \qquad & \text{chart} & $\text{patch}$ \\
\hline
\,&\, &\,\\
   \text{mixture}     & $u = \frac{q}{p} - 1$  &  $q =  (1 + u) \cdot p$ \\
   \,&\, &\,\\
\quad   \text{exponential}\quad \quad  & \quad $u = \log \frac{q}{p} - \int \log \frac{q}{p} \, p \, dm $\quad \quad & $q = \euler^{u - K_p(u)} \cdot p$\\
\,&\, &\,\\
   \quad  \text{Hyv\"arinen} \quad  & ${\bm u} = \nabla \log \frac{q}{p}$ & \quad $q = \euler^{U - K_p(U)} \cdot p \quad {\bm u} = \nabla U$ \quad \quad\\
   \,&\, &\,\\
       \hline
       \end{tabular}
\caption{Basic types of affine charts}
\label{tab:displacements}
\end{table}

We have already discussed one instance of the base set $M$ in examples \ref{ex:running-1}, \ref{ex:running-2}, \ref{ex:running-3}. In applications where one can assume bounded densities for a standard reference measure, one can focus on specific bounded functions, such as bounded measurable or continuous functions on a compact metric space. Such a setup is not suitable for many statistical applications. For example, Gaussian probability measures on $\R^d$ do not transform via bounded factors on a bounded domain. However, there is a considerable scope of application in cases where the sample space is naturally bounded, for example, in some Statistical Physics or Data Science applications.

A more general assumption follows from the observation that a model of exponential form $q_\theta \propto \euler^{\theta v}$ imposes a strong integrability condition on the random variable $v$, precisely, the moment generating function $\hat v(\theta) = \int \euler^{\theta v}$ must be finite for $\theta$ in an open interval. We shall show in the following sec.~\ref{sec:Orlicz-space} that such a condition implies the existence of a family of Banach spaces, and this space provides a statistical bundle. This approach to IG has been introduced in \cite{pistone|sempi:95}.

\subsection{Probability densities and Orlicz spaces} \label{sec:Orlicz-space}
\label{sec:stat-grad-model}
In Lebesgue spaces, one evaluates a function $f$ in some power scale $\avalof f ^ \alpha$, $1 \leq \alpha < \infty$. Then, one computes the norm as $\normat \alpha f ^ \alpha = \int \avalof f ^ \alpha$. In Orlicz spaces, one uses a more general scale $\Phi(f)$, for example, $\Phi(f) = \cosh f - 1$.  The review below does not cover the most general case, but it is general enough for our scope. The monograph by Adams and Fournier \cite[Ch. 8]{adams|fournier:2003} and the Musielak monograph \cite[Ch.~I-II]{musielak:1983} provide basic references to this topic.

If $\phi \in C[0,+\infty[$ satisfies:
\begin{enumerate}
\item  $\phi(0)=0$,
\item $\phi$ is strictly increasing, and
\item $\lim_{u \to +\infty} \phi(u) = +\infty$,
\end{enumerate}
its primitive function
\begin{equation}
  \Phi(x) = \int_0^x \phi(u) \ du \ , \quad x \geq 0 \ ,
\end{equation}
is strictly convex and a diffeomorphism of $]0,\infty[$. The function $\Phi$ is extended to $\R$ by symmetry, $\Phi(x) = \Phi(\avalof x)$ and it is called \emph{Young function}. Notice that $\avalof x ^\alpha$, $\alpha > 1$ is included in our definition, while $\avalof x$ is not. If more generality is needed, see \cite[sec,~8.2]{adams|fournier:2003} and \cite[sec.~7.1]{musielak:1983}.

The inverse function $\psi = \phi^{-1}$ has the same properties 1) to 3) as $\phi$, so that its primitive
\begin{equation}
  \Psi(y) = \int_0^y \psi(v) \ dv \ , \quad y \geq 0 \ ,
\end{equation}
is again a Young function. The couple $(\Phi, \Psi)$ is a couple of \emph{conjugate} Young functions. The relation is symmetric, and we write both $\Psi=\Phi^*$ and $\Phi = \Psi^*$.

The following properties are easy to check.  The \emph{Young inequality} holds true,
\begin{equation}\label{eq:Young-inequality}
  \Phi(x) + \Psi(y) \geq xy \ , \quad x,y \geq 0 \ ,
\end{equation}
and the \emph{Legendre equality} holds true ,
\begin{equation}\label{eq:Legendre-equality}
  \Phi(x) + \Psi(\phi(x)) = x \phi(x) \ , \quad x \geq 0 \ ,
\end{equation}
that is, the \emph{Legendre transform} coincides with the convex conjugate, 
\begin{equation}
  \label{eq:legendre-transform}
  \Psi(y) = xy - \Phi(\phi^{-1}(y)) = \inf_x \left(xy - \Phi(x)\right) \ ,
\end{equation}
Table~\ref{tab:Young-functions} collects the examples we will use in the following. The previous theory is just an exceptional case of convex duality; see, for example, \cite[Ch.~I]{ekeland|temam:1999convex2nd}.

\begin{table}[t]
\centering
\begin{tabular}{|c | c |}
      \hline
     \quad  $\Phi = \Psi_* $& \quad $\Psi = \Phi_*$ \\
      \hline
     \,&\, \\
  \quad $x^{\alpha}/\alpha$ &\quad  $\int_0^y v^ {1/(\alpha-1)} \ dv = y^\beta/\beta \ , \quad 1/\alpha + 1/ \beta = 1$\\
   \,&\, \\
  \quad $\exp_2(x) = \euler^x - 1 - x$ & \quad $(\exp_2)_*(y) = \int_0^y \log(1+v) \ dv = (1+y)\log(1+y) - y$ \quad  \\
   \,&\, \\
  \quad $\cosh_2(x) = \cosh x - 1 $ &\quad $(\cosh_2)_*(y) = \int_0^y \sinh^{-1}(v) \ dv = y \sinh^{-1} y - \sqrt {1 + y^2}$ \quad \\
   \,&\, \\
     \quad  $\operatorname{gauss}_{(x)} = \expof{\frac12x^2}-1 $ \quad & \quad\text{no closed form} \\ 
     \,&\, \\
      \hline     
\end{tabular}
\caption{Examples of Young functions. In the table first line $\alpha, \beta > 1$.}
\label{tab:Young-functions} 
\end{table}

Given a probability space $(X,\mathcal X,\mu)$, we denote by $L^0(\mu)$ the space of $\ mu$ classes of real random variables.
\begin{definition}\label{def:Orlicz-space}Given a Young function $\Phi$ and a probability measure $\mu$, the
\emph{Orlicz space} $\orliczof \Phi \mu$ is the vector subspace of $f \in L^0(\mu)$ such that $\int \Phi(\rho^{-1} f) \ d\mu$ is finite for some $\rho > 0$.
\end{definition}

\begin{proposition}
  $\orliczof \Phi \mu$ is a Banach space for the norm whose closed unit ball is $\setof{f \in L^0(\mu)}{\int \Phi(\avalof f) \ d\mu \leq 1}$.
\end{proposition}

The vector space property follows from the Young function's convexity $\Phi$.  The above-mentioned norm is called \emph{Luxemburg norm}. Explicitly,
\begin{equation}
  \normat {\orliczof \Phi \mu} f \leq \rho \quad \text{if, and only if,} \quad \int \Phi(\rho^{-1} \avalof f) \ d\mu   \leq 1 \ ,
\end{equation}
that is,
\begin{equation}
\normat {\orliczof \Phi \mu} f = \inf \setof{\rho}{\int \Phi(\rho^{-1} \avalof f) \ d\mu   \leq 1} \ .
\end{equation}
We refer to the standard monographs on Orlicz spaces for detailed proofs of the proposition above. See \cite[sec.~8.7-11]{adams|fournier:2003} and \cite[\S{I.1}]{musielak:1983}. See a completeness proof in \cite[Th.~7.7]{musielak:1983}. Some special features of this class of Banach spaces are listed below.

  \begin{enumerate}
  \item If $c$ is a constant function, then $\normat {\orliczof \Phi \mu} c = c$ if, and only if $\Phi(1) = 1$, which is the case for power functions, but is not the case for the other examples in the table.
  \item In the case $\Phi(x) = \avalof x ^\alpha$, $1 < \alpha < \infty$, the Luxemburg norm equals the Lebesgue norm. If $\Phi(x) = \alpha^{-1} \avalof x ^\alpha$, then the Luxemburg norm equals $\alpha^{-1/\alpha}$ $\times$ the Lebesgue norm.
    \item We have assumed the reference measure $\mu$ to be a probability measure. Reference to a probability measure is not part of the general theory of Orlicz spaces, it is a specific feature of the application we seek. 
  \item The convergence of a sequence $(f_n)$ to zero in $\orliczof \Phi \mu$, that is $\lim_{n \to \infty 0} \normat {\orliczof \Phi \mu} {f_n} = 0$, is not equivalent $\lim_{n\to\infty} \int \Phi(f_n) \ d\mu = 0$. In fact, it is required that, for all $\epsilon > 0$, it holds $\normat {\orliczof \Phi \mu} {\epsilon ^ {-1} f_n} \leq 1$ definitively. The condition of norm convergence in terms of integrals is 
    \begin{equation}
      \int \Phi(\epsilon ^ {-1} f_n) \ d\mu \leq 1 \quad \text{definitively for all $\epsilon > 0$.}
    \end{equation}
    Now, for all $0 < \lambda < 1$ it holds $\Phi(\lambda x) \leq \lambda \Phi(x)$, so that
    \begin{equation}
      \int \Phi(\epsilon ^ {-1} f_n) \ d\mu \leq \lambda \int \Phi((\lambda\epsilon) ^ {-1} f_n) \ d\mu \leq \lambda \quad \text{definitively for all $\epsilon > 0$.}
    \end{equation}
  In conclusion,
    \begin{equation}
      \label{eq:convergence}
  \lim_{n \to \infty 0} \normat {\orliczof \Phi \mu} {f_n} = 0 \quad \Leftrightarrow \quad \lim_{n\to\infty} \int \Phi(\epsilon ^ {-1} f_n) \ d\mu = 0 \ , \quad \epsilon > 0 \ .
    \end{equation}
\item If a growth condition of the form $\Phi(ax) \leq C(a) \Phi(x)$, $a > 0$, holds, then the condition $\lim _ {n \to \infty} \int \Phi(f_n) \ d\mu = 0$ clearly implies \eqref{eq:convergence}. The result is true the case of the power functions $\avalof{ax}^ \alpha = a^\alpha \avalof x ^ \alpha$, but it is not the case for $\exp_2$ of Table~\ref{tab:Young-functions} because $\exp_2(cx)/\exp_x(x)$ is unbouded for $x \geq 0$. This issue is important for the duality between conjugate spaces.
\end{enumerate}

For each couple of conjugate Young function $\Phi$ and $\Psi=\Phi^*$, we have a couple of conjugate Orlicz spaces with a duality pairing. Integration of the Young inequality \eqref{eq:Young-inequality} gives
\begin{equation} \label{eq:Young-inequality-integrated}
  \int \avalof{uv} \ d\mu \leq \int \Phi(\avalof u) \ d\mu + \int \Psi(\avalof v) \ d\mu \ .
\end{equation}
The duality pairing is
\begin{equation}
  \label{eq:Orlicz-pairing}
\orliczof {\Phi} \mu \times \orliczof {\Psi} \mu \ni (u,v) \mapsto \scalarat \mu u v = \int uv \ d\mu \ .
\end{equation}
If the norms of $u$ and $v$ in \eqref{eq:Young-inequality-integrated} are both 1, the LHS is bounded by 2, that is,
\begin{equation}\label{eq:continuity-Orlicz-pairing}
  \scalarat \mu u v \leq 2 \normat {\orliczof{\Phi} \mu} u  \normat {\orliczof {\Phi_*} \mu} v \ .
\end{equation}

Each element of an Orlicz space is associated via the duality pairing \eqref{eq:Orlicz-pairing} to a linear continuous functional of the conjugate. However, an Orlicz space is the dual Banach space of its conjugate in particular cases only; see below. However, an equivalent norm follows from the duality pairing, namely, the \emph{Orlicz norm},
\begin{equation}\label{eq:orlicz-norm}
  \normat {{\orliczof {\Psi} \mu }^*} f = \sup \setof{\scalarat \mu f
    g}{\normat {\orliczof {\Psi} \mu} f \leq 1} \ .
\end{equation}
By bounding the pairing with \eqref{eq:continuity-Orlicz-pairing}, we see that $\normat {{\orliczof \Phi \mu}^*} f \leq 2 \normat {\orliczof \Psi \mu} f$. Because of this inequality, \eqref{eq:orlicz-norm} defines a continuous norm on $\orliczof \Psi \mu$ and  $\scalarat \mu f g \leq \normat {\orliczof \Psi \mu ^*} f \normat {\orliczof \Psi \mu} g$. Moreover, the Luxembourg norm and the Orlicz norm are equivalent. Let us show that $\normat {\orliczof \Phi \mu ^*} f \leq 1$ whenever $\normat {\orliczof \Psi \mu ^*} f \leq 1$. The conjugation relation extends to integrals,
\begin{equation}\label{eq:integrated-conjugation}
  \int \Phi(f) \ d\mu = \sup \setof{\scalarat \mu f g - \int \Psi(g) \ d\mu}{g \in \orliczof \Psi \mu} \ .
\end{equation}
but we can compute the $\sup$ on a smaller set because
  \begin{multline}
  \sup \setof{\scalarat \mu f g - \int \Psi(g) \ d\mu}{\int \Psi(g)  d\mu > 1} \leq  \\ \sup \setof{\normat {\orliczof \Psi \mu} g - \int \Psi(g) \ d\mu}{\int \Psi(g)  d\mu > 1} \leq 0 \ .
\end{multline}
With that, \eqref{eq:integrated-conjugation} becomes
\begin{equation}
 \int \Phi(f) \ d\mu = \sup \setof{\scalarat \mu f
    g - \int \Psi(g) d\mu}{\normat {\orliczof {\Psi} \mu} f \leq 1} \leq 1
\end{equation}
and the bound is proved.

Other equivalent norms are of interest and will be discussed later in specific instances of the Young function or the base measure $\mu$.

The domination relation between Young functions implies continuous injection properties for the corresponding Orlicz spaces. We will say that $\Phi_2$ \emph{eventually dominates} $\Phi_1$, written $\Phi_1 \prec \Phi_2$, if there are positive constants $a,b$ and a non-negative $\bar x$ such that $\Phi_1(x) \leq a\Phi_2(bx)$ for all $x \geq \bar x$. As, in our case, $\mu$ is a probability measure, the continuous embedding $\orliczof {\Phi_2} \mu \to \orliczof {\Phi_1} \mu$ holds if, and only if, $\Phi_1 \prec \Phi_2$. If $\Phi_1 \prec \Phi_2$, then $(\Phi_2)_* \prec (\Phi_1)_*$. See \cite[Th.~8.12]{adams|fournier:2003} or \cite[Th.~8.5]{musielak:1983}.

When there exists a function $C$ such that $\Psi(ax) \leq C(a) \Psi(x)$ for all $a \geq 0$. In such a case, the conjugate $\orliczof {\Psi^*} \mu$ is the dual Banach space of $\orliczof \Psi \mu$, and bounded functions are a dense
set. We do not discuss this classical topic further because it is not relevant to our application to IG. See details in \cite[sec.~8.17-20]{adams|fournier:2003}.

We now discuss the examples we use in our version of the affine statistical manifold. See the list in tab.~\ref{tab:Young-functions}.

The spaces corresponding to the power functions coincide with the ordinary Lebesgue spaces. The norm is related by
\begin{equation}
  \normat {\orliczof {\Phi_\alpha} \mu } f = \alpha^{1/\alpha} \normat {L^\alpha(\mu)} f \ .
\end{equation}
The embedding conditions hold. The spaces are dual of each other.

The Young function $\exp_2$ and $\cosh_2$ are equivalent, and the Orlicz spaces are isomorphic equal as vector spaces and isomorphic, $\orliczof {\cosh_2} \mu \leftrightarrow \orliczof {\exp_2} \mu$. This example is of special interest to us as it provides the model spaces for a non-parametric version of Information Geometry; see sec.~\ref{sec:exponential-bundle} below.  They both are eventually dominated by $\operatorname{gauss}_2$ and eventually dominate all powers, that is,
\begin{multline}
L^\infty(\mu) \hookrightarrow \orliczof {\operatorname{gauss}_2} \mu \hookrightarrow \orliczof {\exp_2} \mu \approx \orliczof {\cosh_2} \mu \hookrightarrow \\ L^\alpha(\mu) \hookrightarrow L^2(\mu) \hookrightarrow L^\beta(\mu) \\ \hookrightarrow \orliczof {\exp_2^*} \mu \approx \orliczof {\cosh_2^*} \mu \hookrightarrow \orliczof {\operatorname{gauss}_2^*} \mu \hookrightarrow L^1(\mu) \ ,
\end{multline}
where $\alpha > 2$ and $1 < \beta  <2$ are conjugate, $1/\alpha+1/\beta = 1$. Each space at the left of $L^2(\mu)$ is the dual of one space at the right.

The Orlicz space $\orliczof {\exp_2} \mu = \orliczof {\cosh_2} \mu$ is known by many different names in various chapters of Statistics. The proposition below provides such definitions with references.
\begin{proposition}\ 
  \begin{enumerate}
  \item A function belongs to the space $\orliczof {\cosh_2} \mu$ if, and only if, its \emph{moment generating function $\lambda \mapsto \int \euler^{\lambda f}$ is finite in a neighborhood of 0}. This implies that the moment-generating function is analytic at 0. See \cite[Ch. 2]{brown:86}.
  \item The same property is equivalent to a large deviation inequality. See \cite[Ch. 2]{wainwright:2019-hds}. Precisely, a function $f$ belongs to $\orliczof {\cosh_2} \mu$ if, and only if, it is \emph{sub-exponential}, that is, there exist constants $C_1,C_2 > 0$ such that
\begin{equation}
  \mu (\avalof{f} \geq t) \leq C_1 \expof{-C_2 t} \ , \quad t \geq 0 \ .
\end{equation}
\end{enumerate}
\end{proposition}
\begin{proof} The first statement is immediate. If $\normat {\orliczof {\cosh_2}
  \mu} f = \rho$, then $\int \euler^{\rho^{-1} \avalof f} \ d\mu \leq 4$. It
  follows that
  \begin{equation}
    \mu(\avalof f > t) =
    \mu\left(\euler^{\rho^{-1}\avalof f} >
      \euler^{\rho^{-1}t}\right) \leq
    \left(\int \euler^{\rho^{-1}\avalof f} \ d\mu\right) 
    \euler^{- \rho^{-1}t} \leq
    4 \euler^{- \rho^{-1}t} \ .  
  \end{equation}
The sub-exponential inequality holds with $C_1=4$ and $C_2 = \normat
{\orliczof {\cosh_2} \mu} f ^{-1}$. Conversely, for all $\lambda > 0$,
\begin{equation}
  \int \euler^{\lambda f} \ d\mu \leq \int_1^\infty \mu\left(\euler^{\lambda
      f^+}> t\right) \ dt \leq C_1 \int_0^\infty \euler^{-(C_2
    \lambda^{-1} - 1)s} \ ds \ .  
\end{equation}
The right-hand side is finite if $\lambda < C_2$ and the same
bound holds for $-f$. \smartqed
\end{proof}

Further useful facts are listed below.

\begin{enumerate}
  \item A sub-exponential random variable is exciting in applications because it admits an explicit exponential bound in the Law of Large Numbers. Another class of interest consists of the \emph{sub-Gaussian} random variables, that is, those random variables whose square is sub-exponential. See \cite{wainwright:2019-hds}.
\item
The theory of sub-exponential random variables provides an
\emph{equivalent norm for the space} $\orliczof {\cosh_2} \mu$. See \cite{buldygin|kozachenko:2000} or
\cite{siri|trivellato:2021-SPL}. The norm is
\begin{equation}
  f \mapsto \sup_k \left((2k)!^{-1} \int f^{2k} \ d\mu
  \right)^{1/2k} = \left\bracevert f \right\bracevert _ {\cosh_2} \ .
\end{equation}
Let us prove the equivalence. If
$\normat {\orliczof {\cosh_2} \mu} f \leq 1$, then
\begin{equation}
  1 \geq \int \cosh_2 f \ d\mu \geq \frac1{(2k)!} \int f^{2k} \ d\mu
    \quad \text{for all $k = 1,2,\dots$,}
\end{equation}
so that $1 \geq  \left\bracevert f \right\bracevert _
{\cosh_2}$. Conversely, if the latter inequality holds, then
\begin{equation}
  \int \coshtwo (f/\sqrt 2) \ d\mu = \sum_{k=1}^\infty \frac1{(2k)!}
  \int f^{2k} \ d\mu \left(\frac12\right)^k \leq 1 \ , 
\end{equation}
so that $\normat {\orliczof {\coshtwo} \mu} f \leq \sqrt 2$.
\item 
It is convenient to introduce a further notation. For each Young
function $\Phi$, the function $\overline \Phi(x) = \Phi(x^2)$ is again
a Young function such that
$\normat {\orliczof {\overline \Phi} \mu} f \leq \lambda$ if, and only
if, $\normat {\orliczof \Phi \mu} {\avalof f ^2} \leq
\lambda^2$. \emph{We will denote the resulting space by}
$\orliczpof \Phi 2 \mu$. For example, $\operatorname{gauss}_2$ and
$\overline{\cosh_2}$ are $\prec$-equivalent , hence the isomorphism
$\orliczof {\operatorname{gauss}_2} \mu \leftrightarrow \orliczpof
{\cosh_2} 2 \mu$. As an application of this notation, consider that
for each increasing convex $\Phi$ it holds
$\Phi(fg) \leq \Phi((f^2+g^2)/2) \leq (\Phi(f^2) + \Phi(g^2))/2$. It
follows that when the $\orliczpof \Phi 2 \mu$-norm of $f$ and of $g$
is bounded by one, the $\orliczof \Phi \mu$-norm of $f$, $g$, and
$fg$, are all bounded by one. The space $\orliczof {\coshtwo} \mu$ has
a continuous injection in the Fr\'echet space
$L^{\infty-0}(\mu) = \cap_{\alpha>1} L^\alpha(\mu)$, which is an
algebra. When we need the product, we can either assume the factors
are both sub-Gaussian or move up the functional framework to the
Lebesgue spaces' intersection.
\end{enumerate}

\begin{example}[Gaussian exponential Orlicz space]
\label{ex:calc-gauss-space}
Let us now discuss other special issues of Orlicz spaces by focusing on a case of specific interest in IG that is, the \emph{Gaussian exponential Orlicz space} $\orliczof {\cosh_2} \gaussdensity$, with $\gaussdensity$ the standard $n$-variate Gaussian density. We note that that \emph{dominated convergence does not hold in this space}. In fact, the squared-norm function $f(x) = \avalof x ^ 2$ belongs to the Gaussian exponential Orlicz space $\orliczof {\cosh_2} \gaussdensity$ because
\begin{equation}
  \int \cosh_2(\lambda f(x)) \ \gaussdensity(x)dx < \infty \quad
  \text{for all} \quad \lambda < 1/2 \ .
\end{equation}
The sequence $f_N(x) = f(x)(\avalof x \leq N)$ converges to $f$
point-wise and in all $L^\alpha(\gamma)$, $1 \leq \alpha < \infty$.
However, the convergence does not hold in the Gaussian exponential
Orlicz space. We see that, for all $\lambda \geq 1/2$,
\begin{multline*}
 \int  \cosh_2(\lambda (f(x)-f_N(x))) \ \gaussdensity(x)dx =
 \int_{\avalof x > N} \coshtwo (\lambda f(x)) \ \gaussdensity(x)dx =
 \infty \ ,
\end{multline*}
while convergence would imply
\begin{equation}
  \limsup_{N \to \infty} \int  \cosh_2(\lambda (f(x)-f_N(x))) \ \gaussdensity(x)dx \leq 1 \quad \text{for all $\lambda >0$}\ .
\end{equation}
In the same spirit, one must observe that the closure in $\orliczof {\cosh_2} \gaussdensity$ of the vector space of bounded functions is called \emph{Orlicz class} $M_{\cosh_2}(\gaussdensity)$ and is strictly smaller than the full Orlicz space. Precisely, one can prove that $f \in M_{\cosh_2}(\gaussdensity)$ if, and only if. the moment generating function $\lambda \mapsto \int \euler^{\lambda f(x)} \ \gaussdensity(x)dx$ is finite for all $\lambda$. See \cite{pistone:2018}. An example is $f(x) = x$. Bounded convergence holds in the Orlicz class. Assume $f \in M_{\cosh_2}(\gaussdensity)$ and consider the sequence $f_N(x) = (\avalof x \leq N) f(x)$. Now,
\begin{equation}
  \int \coshtwo(\lambda(f(x) - f_N(x)) \ \gaussdensity(x) dx =
  \int_{\avalof x \geq N} \coshtwo(\lambda f(x)) \ \gaussdensity(x) dz
\end{equation}
converges to 0 as $N \to \infty$.
\end{example}

\section{Exponential statistical \emph{bundles}} \label{sec:exponential-bundle}

We now show how to apply the construction of the exponential and mixture affine bundles when the fibers are Orlicz spaces, as defined in the previous sections. Precisely, we consider two equivalent conjugate couples of Young functions from Table~\ref{tab:Young-functions},  
\begin{align}
  \exp_2(x) &= \euler^x - 1 - x = \int_0^x (x-s) \exp s \ ds \ . \label{eq:int-1} \\
  \exp_2^*(y) &= (1+y) \log(1+y) - y = \int_0^y \frac {y - s}{1+s} \ ds \ . \label{eq:int-2} \\
  \cosh_2(x) &= \cosh x -1 = \int_0^x (x-s) \cosh s \ ds \ . \label{eq:int-3}\\
  \cosh_2^*(y) &= y \sinh^{-1} y - \sqrt {1 + y^2} = \int_0^y \frac {y - s}{\sqrt{1+s^2}} \ ds \ .
\end{align}

The integral form is convenient in proving useful inequalities. The identity of the Orlicz spaces from the Young functions \eqref{eq:int-1} and \eqref{eq:int-3} follows from the inequalities $\cosh_2(x) \leq \exp_2(x) \leq 2 \cosh_2(x)$, for  $x \geq 0$, which, in turn, follow from $\cosh s \leq \exp s \leq 2 \cosh s$. For all $a,y > 0$, 
\begin{equation}
  \exp_2^*(ay) = \int_0^{ay} \frac{ay - s}{1+s} \ ds = a \int_0^y \frac {a(y -t)}{1+at} \ dt \leq \max(a,a^2) \int_0^y \frac{y - t}{1+t} \ dt \ , 
  \end{equation}
and similarly for $\cosh_2^*$. The growth bounds 
\begin{equation}\label{eq:Delta-2}
\exp_2^*(ay) \leq a \max(1,a) \exp_2^*(y) \quad \text{and} \quad \cosh_2^*(ay) \leq  a \max(1,a)\cosh_2^*(y)\end{equation}
imply a bound on the Luxemburg norm. If $\int \exp_2^*(f) \ dm < \infty$, then
\begin{equation} \label{eq:Delta-2-class}
 \int \exp_2^*(\rho^{-1} f) \ dm \leq \rho^{-1} \max(1, \rho^{-1}) \int \exp_2^*(f) \ dm < \infty \ , \quad \rho > 0 \ .   
\end{equation}
It follows that 
\begin{multline}
\normat {\orliczof {\exp_2^*} m} f ^{-1} \max\left(1,\normat {\orliczof {\exp_2^*} m} f ^{-1}\right)  \int \exp_2^*\left( f\right) \ dm \geq \\ \int \exp_2^*\left(\normat {\orliczof {\exp_2^*} m} f ^{-1} f\right) \ dm = 1 \ ,  
\end{multline}
and
\begin{equation}
  \normat {\orliczof {\exp_2^*} m} f  \min\left(1,\normat {\orliczof {\exp_2^*} m} f \right) \leq \int \exp_2^*\left( f\right) \ dm  \ .
\end{equation}

The bounds \eqref{eq:Delta-2} imply that the conjugate space equals the dual Banach space. This is a general result, see \cite[\S13]{musielak:1983} or \cite[\S8.17--20]{adams|fournier:2003}. Here, we sketch the argument in our special case. From the bound in Eq.~\eqref{eq:Delta-2-class}, the norm is follows with equality, that is, $\normat {\orliczof {\exp_2^*} m} { f} = 1$ if, and only if, $\int \exp_2^*(f) \ dm = 1$. Now, Eq.~\eqref{eq:int-2} shows that $\exp_2^*(y)$ is smaller than $y^2/2 = \int_0^y (y-s) \ ds$. Because of that, we have the injection $L^2(m) \hookrightarrow \orliczof {\exp_2^*} m$, and hence the dual injection $\left(\orliczof {\exp_2^*} m\right)^* \hookrightarrow (L^2(m))^* = L^2(m)$ so that each element of the dual is a random variable. That is, a linear functional on $\orliczof {\exp_2^*} m$ of norm $k$ is of the form $f \mapsto \scalarat m f g$ with $\avalof {\scalarat m f g} \leq k \normat {\orliczof {\exp_2^*} m} f$. The dual is $\orliczof {\exp_2} m$ with the Orlicz norm. In conclusion
\begin{equation}
  \orliczof {\exp_2} \mu = {\orliczof {\exp_2^*} \mu} ^* \quad \text{and} \quad \orliczof {\cosh_2^*} \mu = {\orliczof {\cosh_2^*} \mu} ^* \ .  
\end{equation}
The reverse duality, reflexivity, does not hold unless the sample space is finite. The dual of $\orliczof {\exp_2} \mu$ contains functionals that are not representable as functions. This is similar to the well known case $(L^1(m))^* = L^\infty(m)$ and $(L^\infty(m))^* \subsetneq L^1(m)$. See the general references already cited above.

\begin{remark}[Analytic bi-lateral Laplace transform] \label{remark:analytic-laplace}
  The random variable $u$ belongs to $\orliczof {\exp_2} m$ if, and only if, the Laplace transform of the image probability measure $u_*(m)$ is finite on an open interval containing 0. In such a case, the Laplace transform itself is analytic at 0. See, for example, \cite{malliavin:1995}
\end{remark}

\begin{remark}[Densities with finite entropy] \label{remark:entropy} Let us discuss the relation between the following properties: 1) the probability density $p$ belongs to the conjugate space $\orliczof {\exp_2^*} m$; 2) the density $p$ has integrable logarithm $-\infty < \int \log p \ dm \leq 0$; 3) the density $p$ has finite entropy $0 \leq \int p \log p \ dm < +\infty$. A classical reference is \cite{cover|thomas:2006}. The function $y \mapsto y \log y$, $y > 0$, is convex. Hence, the increment is bounded by the derivatives at the extreme points, $\log y + 1 \leq (1+y) \log (1+y) - y \log y \leq \log(1+y) + 1 \leq y + 1$. From the upper bound, integration gives $\int \exp_2^*(p) \ dm \leq \int p \log p \ dm + 1$, that is, 3) implies 1). From the lower bound, if 2) holds, then $\int p \log p \ dm \leq \int \exp_2^*(p) \ dm - \int \log p \ dm$, that is, 1) and 2) imply 3).
\end{remark}

\subsection{Maximal exponential model}\label{sec:MEM}

In this section, we apply the general methods of Sec.~\ref{sec:nonparametric-statistical-affine-manifolds} to the specific case of a Banach manifold modeled on the Orlicz space $\orliczof {\cosh_2} \mu$. The idea to study the geometry of statistical models by embedding them in a larger exponential family is due to \cite{efron:1975} and \cite{efron:1978}. The possibility of a non-parametric set-up for the statistical bundle was suggested first in \cite{dawid:75} and \cite{dawid:1977as}. See also the review paper \cite{kass:89}. The idea of considering the largest possible exponential model was discussed first in \cite{pistone|sempi:95}, \cite{cena|pistone:2007}, \cite{pistone|rogantin:99}. The extension to the statistical bundle appeared later; see \cite{gibilisco|pistone:98}, \cite{pistone:2013GSI} and \cite{pistone:2018}.

We shall first define the moment functional and the cumulant functional. These are non-parametric versions of the moment-generating and cumulant-generating functions, respectively. Given a probability measure $\mu$ on the measurable space $(\Omega,\mathcal B)$, we define $B_\mu = \setof{u \in \orliczof {\cosh_2} \mu}{\int u \ d\mu = 0}$, where the Orlicz space $\orliczof{\cosh_2}{\mu}$ is defined in sec.~\ref{sec:Orlicz-space}. $B_\mu$ is a Banach space when the Luxemburg norm of $\orliczof {\cosh_2} \mu$ is restricted to the sub-space. The \emph{moment functional} is the convex mapping
\begin{equation}
    M_\mu \colon B_\mu \ni u \mapsto \int \euler^u \ d\mu \in ]0,\infty] \ .
  \end{equation}
  
The \emph{proper domain} of $M_\mu$, $\setof{u \in B_\mu}{M_\mu(u) < \infty}$, is a convex subset of $B_\mu$ that contains the unit ball of $\orliczof {\cosh_2} \mu$. In fact, $\int \cosh_2 u \ d\mu \leq 1$ implies  $\int \euler^u \ d\mu = M_\mu(u) \leq 4$. It follows that the interior of the proper domain of $M_\mu$ is an open convex set, $\mathcal S(\mu) = \setof{u \in B_\mu}{k_\mu(u) < \infty}^\circ$.

The \emph{cumulant functional} is defined by for $u \in \mathcal S(\mu)$ by $K_\mu(u) = \log M_\mu(u)$. For all $u \in \mathcal S(\mu)$ and $h \in B_\mu$, the mapping $t \mapsto M_\mu(h+th) = \int \euler^{th} \ \euler^u \ d\mu$ is the Laplace transform of the random variable $h$ for the finite measure $\euler^u \cdot d\mu$ and it is defined on a neighborhood of 0. It follows from standard results that the mapping is infinitely differentiable at $t=0$ with $k$'th derivative. See, for example, \cite{brown:86}.

\begin{equation} \left. \frac{\partial M_\mu(u+\sum_{j=1}^k t_j h_j)}{\partial t_1 \cdots \partial t_k} \right|_{t_,\dots,t_k=0} =
\int h_1 \cdots h_k \ \euler^u \ d\mu \ .
\end{equation}

The multi-linear mapping $B_\mu \times \cdots \times B_\mu \ni (h_1, \dots, h_k) \to h_1 \cdots h_k \euler^u$ is bounded into $L^1(\mu)$. In fact,
\begin{equation}
  \int \avalof{h_1 \cdots h_k} \euler^u \ d\mu \leq \left(\int \avalof{h_1 \cdots h_k}^n \ d\mu\right)^{1/n} \left(\int \euler^{n u/(n-1)} \ d\mu\right)^{(n-1)/n} 
\end{equation}
and we can chose $n$ such that $nu/(n-1) \in s_\mu$. For such an $n$, the first factor is bounded by
\begin{equation}
  \int \avalof{h_1 \cdots h_k}^n \ d\mu \leq n!^k \int \euler^{h_1 + \cdots h_k} \ d\mu \ ,
\end{equation}
where the RHS integral is bounded if $h_1 + \cdot + h_k \in \mathcal S(\mu)$. This proves the boundedness. In other words, the mapping $\mathcal S(\mu) \ni u \mapsto M_\mu(u)$ is infinitely Gateaux-differentiable, and the derivatives are continuous linear operators.  In compact notation, the derivative is
\begin{equation}
  d^k M_\mu(u) [h_1 \cdots h_k] = \int h^k \euler^u \ d\mu \ .
\end{equation} 
We shall not discuss stronger differentiability conditions. See, for example, in \cite{cena:2002} the proof that the moment functional is Frech\'et-differentiable on $\mathcal S(\mu)$ and analytic on the open unit ball. The moment functional and the cumulant functional are the normalizing constants for probability densities of exponential form. The support $M$ of the affine manifold of interests here is the \emph{maximal exponential model} $\maxexpat \mu$ consisting of all probability densities of the form
\begin{equation}
  q = M_\mu(u)^{-1} \euler^u = \expof{u - K_\mu(u)} \ , \quad u \in \mathcal S(\mu) \ .
\end{equation}

The following portmanteau theorem is crucial for the consistency of the affine structure of our Orlicz space setup. It shows the existence of a statistical bundle with base $\maxexpat \mu$,  whose fibers are closed subspaces of $\orliczof {\cosh_2} \mu$ and admit a proper cycle of parallel transports.

\begin{proposition} \label{portmanteaux} For all densities $p, q \in \maxexpat \mu$ the following propositions are equivalent.
\begin{enumerate}
\item  $q = \euler^{u - K_p(u)} \cdot p$, where $u \in \orliczof {(\cosh-1)} \mu$, $\int u \ p \ d\mu = 0$, and $u$ belongs to the interior of the proper domain of the convex function $K_p = \int \euler^u \ p \ d\mu$. 
\item\label{port-1} An open exponential arc connects the densities $p$ and $q$, that is, there exists a one-dimensional exponential family $r_\theta \propto \euler^{\theta U}$ with $\theta \in I$, $r_0 = p$, $r_1 = q$, and $[0,1] \subset I$.
\item\label{port-2} $\orliczof {\cosh_2} p = \orliczof {\cosh_2} q$ and the norms are equivalent;
\item\label{port-3} $p/q \in \cup_{\alpha > 1} L^\alpha(q)$ and $q/p \in \cup_{\alpha>1} L^\alpha(p)$.
\item\label{port-4} The mapping $v \mapsto \frac q p v$ is an isomorphism of $\orliczof {\cosh_2^*} p$ onto $\orliczof {\cosh_2^*} q$\end{enumerate}

\end{proposition}
\begin{proof}[Main argument only] We give only part of the proof. See \cite{cena|pistone:2007,santacroce|siri|trivellato:2016,santacroce|siri|trivellato:2018,siri|trivellato:2021-SPL} for a detailed proof and some further developments. Precisely, we prove a generalization of the implication \ref{port-1}$\Rightarrow$\ref{port-2}. Let $F$ be logarithmically convex on $\R$, such that $\Phi = F - 1$ is a Young function. For example, the assumption holds for both $F(x)=\cosh x$ and $F(x) = \euler^{x^2/2}$. For all real $A$ and $B$, the function
\begin{equation*}
  \R^2 \ni  (\lambda,t) \mapsto F(\lambda A) \euler^{tB} = \expof{\log F(\lambda
    A) + tB}  
\end{equation*}
is convex, and so is the integral
\begin{equation*}
  C(\lambda,t) = \int F(\lambda f(x)) \ \euler^{tu(x)} \ p(x) \ \mu(dx) \ ,
\end{equation*}
where $f \in \orliczof \Phi \gaussdensity$ with and $u \in \orliczof {\cosh_2} p$ with $\int u(x) \ p(x) \ dx = 0$. Without the restriction of generality, assume $\normat {\orliczof \Phi p} f =1$. Let us derive two marginal inequalities. First, for $t=0$, the definition of Luxemburg norm gives
\begin{equation*}
  C(\lambda,0) = \int F(\lambda f) \ p(x) \ \mu(dx)
  \leq 2 \ , \quad -1 \leq \lambda \leq 1 \ .
\end{equation*}

Second, for $\lambda = 0$, consider $K_p(tu) = \log \int \euler^{tu} \ p(x) \mu(dx)$, where $t$ belongs to an an open interval $I$ containing $[0,1]$ and such that $K_p(tu) < + \infty$. It follows that
\begin{equation*}
  C(0,t) = \int \euler^{tu} \ p(x) \mu(dx) = \euler^{K_p(tu)} <
  + \infty \ .
\end{equation*}

Choose a $t>1$ in $I$ and consider the convex combination
\begin{equation*}
  \left(\frac{t-1}t,1\right) = \frac{t-1}t (1,0)+\frac1t (0,t)
\end{equation*}
and the inequality
\begin{equation*}
  C\left(\frac{t-1}t,1\right) \leq \frac{t-1}t C(1,0)+\frac1t C(0,t)
  \leq 2 \frac{t-1}t + \frac1t \euler^{K_1(tu)} \ .
\end{equation*}
Now,
\begin{multline*}
  \int \Phi\left(\frac{t-1}t f(x)\right) \euler^{u(x)-K_1(u)}
  p(x) \ \mu(dx) = \\ \int F\left(\frac{t-1}t f(x)\right)
  \euler^{u(x)-K_p(u)} p(x) \ \,u(dx) - 1 = \\
  \euler^{-K_p(u)} C\left(\frac{t-1}t,1\right) - 1 \leq
   \euler^{-K_p(u)} \left(2\frac{t-1}t + \frac1t
     \euler^{K_p(tu)}\right) - 1 \ .\end{multline*}

 As the RHS is finite, we have proved that $f \in \orliczof \Phi q$ for $q = \euler^{u - K_p(u)}$. Conversely, a similar argument shows the other implication. We have proved that all Orlicz spaces $\orliczof \Phi p$, $p \in \maxexpat \gaussdensity$ are equal. In turn, equality of spaces implies the equivalence norms. It is possible to derive explicit bounds by choosing a $t$ such that the RHS is smaller or equal to 1. \smartqed
\end{proof}

If $p,q$ are densities in the maximal exponential model, then there is an $\epsilon > 0$ such that the combination $r = (1-\lambda) p + \lambda q$, $-\epsilon < \lambda < 1+\epsilon$ is there too. That is, \emph{the maximal exponential model is a convex set and is open on lines}. This is proved by Prop.~\ref{portmanteaux}\eqref{port-3}. See the detailed proof in \cite{santacroce|siri|trivellato:2016}.

Let us recall the notations of Sec.~\ref{sec:nonparametric-statistical-affine-manifolds}: 
\begin{gather}
  B_p = \setof{u \in \orliczof {\cosh_2}  \mu}{\int u \ p \ d\mu = 0} \label{eq:bsubp}\\
  \euler_p(u) = \expof{u - K_p(u)} \cdot p \label{eq:esubp} \\
  \etransport p q u = u - \int u \ p \ d\mu \label{eq:etransp} \quad \mtransport p q u = \frac q p u
\end{gather}

The spaces $B_p$, $p \in \maxexpat \mu$, of Eq. ~\eqref{eq:bsubp} will be the fibers of the statistical bundle. Given $p, q \in \maxexpat \mu$, prop.~\ref{portmanteaux}\eqref{port-2} shows that the Banach spaces $B_p$ and $B_q$ are vector sub-spaces of co-dimension 1 of the two isomorphic space $\orliczof {\cosh_2} p \approx \orliczof {\cosh_2} q$. The mapping $\etransport p q \colon B_p \to B_q$ of Eq.~\eqref{eq:etransp} is such an automorphism.

According to prop.~\ref{portmanteaux}\eqref{port-1}, for each given $p \in \maxexpat \mu$, every other $q \in \maxexpat \mu$ is of the form $\euler_p(u)$ for some $u \in B_p$. Precisely, $\log \frac q p = u - K_p(u)$ with $\int \log \frac q p \ p \ d\mu = - K_p(u)$ and $u = \log \frac q p - \int \log \frac q p \ p \ d\mu$. It follows that the mapping
\begin{equation} \label{eq:exp-displacement}
\maxexpat \mu ^ 2 \ni (p,q) \mapsto \displacement (p,q) =  \log \frac q p - \int \log \frac q p \ p \ d\mu \in B_p
\end{equation}
is well defined and, moreover,
\begin{equation}
 s_p \colon \maxexpat \mu \ni q \mapsto \displacement(p,q) \in B_p 
\end{equation}
is 1-to-1 with image  the open set $\setof{u \in B_p}{K_p(u) < \infty}^\circ$.

Eq.~\eqref{eq:etransp} defines a linear continuous invertible operator from $B_p$ onto $B_q$. The cocycle properties hold: $\etransport p p$ is the identity and $\etransport q r \etransport p q = \etransport p r$. In turn, we see that $\left(\etransport p q\right) ^ {-1} = \etransport q p$.

We have thus proved that the statistical bundle
\begin{equation}
  S\maxexpat \mu = \setof{(p,u)}{p \in \maxexpat \mu, u \in B_p}
\end{equation}
admits the family of parallel transports $\etransport p q \colon B_p \to B_q$. For this family, the map $\displacement$ of eq.~\eqref{eq:exp-displacement} is an affine displacement. Moreover, the image of the chart is an open set. All requirements for an affine statistical manifold hold.

Let us discuss the duality. Define 
\begin{equation}
  \pstar B_b = \setof{v \in \orliczof {\exp_2^*} \mu}{ \int v p \ d\mu
    = 0} \ . 
\end{equation} 
It is a Banach space for the restriction of the Luxemburg norm. We use the pre-script notation to remember that $B_p$ is the dual of $\pstar B_p$, that is, $\left(\pstar B_p\right)^* = B_p$. In the pairing $B_p$, $\pstar B_p$, it holds
\begin{equation}
  \scalarat q {\etransport p q u} v = \scalarat p u {\mtransport q p v} \ ,
\end{equation}
for all $p,q \in \maxexpat \mu$, $u \in B_p$ and $v \in \pstar B_p$.

We can define the conjugate affine system with
\begin{equation}
 \pstar s_p \colon q \mapsto \frac q p - 1 \ .
\end{equation}
Let us check that $\pstar s_p(q) \in \pstar B_p$. Clearly, $\int \left(\frac q p - 1\right)\ p \ d\mu = 0$. From $\exp_2^*(y) = (1+y) \log(1+y) - y$, we find \begin{equation}\int \exp_2^*(s_p(q)) \ p \ d\mu = \int \frac q p \log \frac q p p \ d\mu - \int \frac q p \ d\mu + 1 = - \int \log \frac p q \ q \ d\mu \ .
\end{equation}
The last term of equality is finite because it is the opposite of the normalizing constant of the exponential representation of $p$ for $q$. The existence of a family of transports is shown in Prop.~\ref{portmanteaux}\eqref{port-4}. 

Below, we give the first three derivatives of the cumulant functional. As the values of the Gateaux derivatives are directional, all equalities below reduce to well-known properties of the usual cumulant generating functions.
\begin{equation}\label{eq:K1}
  d K_p(u) [h] = \int h \ \euler^{u - K_p(u)} \ d\mu = \int h \ \euler_p(u) \ d\mu \ . \end{equation}

\begin{multline} \label{eq:K3}   
d^2 K_p(u) [h_1,h_2] = \int (\etransport p {\euler_p(u)} h_1)(\etransport p {\euler_p(u)} h_2) \ \euler_p(u) \ d\mu = \\ \scalarat {\euler_p(u)}{\etransport p {\euler_p(u)} h_1}{\etransport p {\euler_p(u)} h_2} \ . 
\end{multline}
\begin{equation}\label{eq:K4}
d^3K_p(u) [h_1,h_2,h_3] =
    \int (\etransport p {\euler_p(u)} h_1)(\etransport p {\euler_p(u)} h_2)(\etransport p {\euler_p(u)} h_3) \ \euler_p(u) \ d\mu \ .
    \end{equation}

\begin{remark} The equations above show that the expected value $\expectat q h$, the covariance $\covat q {h_1} {h_2}$, and the triple covariance $\tricovat q {h_1}{h_2}{h_3}$, all depend on a convex function, namely, the cumulant functional. What we have here is a special case of Hessian geometry. See \cite{shima:2007}.
\end{remark} 
 
 \begin{remark}[Entropy: cf. ex.~\ref{ex:running-4} and Rem.~\ref{remark:entropy}] The entropy is finite on all of $q \in \maxexpat \mu$ because $-\entropyof q = \KL q 1$ and its expression in terms of $K_1$ and $dK_1$ is
 \begin{equation}
     \int q \log q \ d\mu = \int (u - K_1(u)) \euler_1(u) \ d\mu = dK_1(u)[u] - K_1(u) \ .
 \end{equation}
 
 The random variable $\log q = u - K_1(u)$ is integrable, with $\int \log q \ d\mu = -K_p(u)$. It follows from the argument in rem.~\ref{remark:entropy} that each density $p \in \maxexpat \mu$ is an element of $\orliczof {\exp_2^*} \mu$.

We have the mapping between convex sets 
\begin{equation}
B_1 \supset \mathcal S(\mu)   \ni u \mapsto \euler_1(u) \in \maxexpat \mu \subset \pstar B_1 + 1 \ .    
\end{equation}
Notice that $p \mapsto \pstar s(p) = p - 1 \in \pstar B_1$ is not a chart of $\maxexpat \mu$ because the image is not always an open set. It is the restriction of an affine chart defined on a larger base set, the affine subspace of $\orliczof {\cosh_2^*} \mu$ generated the maximal exponential model, that is, $\pstar B_1 + 1$.
\end{remark}
  
 We conclude this section with examples showing classical topics of IG in our formalism.
 
\begin{example}[First and second variation of the KL divergence] Given $p, q \in \maxexpat m$, we can write $q = \euler^{u - K_p(u)} \cdot p = \euler_p(u)$, that is, $u = s_p(q)$. It follows that the value at $u$ of the expression of the Kullback-Leibler divergence $q \mapsto \KL q p$ in the chart $s_p$ is
   \begin{multline}
     \KL q p = \int \log \frac q p \ q \ dm = \int (u - K_p(u)) \
     \euler_p(u) \ p \ dm = \\ \int u \
     \euler_p(u) \ p \ dm - K_p(u) = d K_p(u) [u] - K_p(u) \ .
   \end{multline}
The derivative in the direction $h$ is
   \begin{multline}
 d (d K_p(u) [u] - K_p(u)) [h] = d^2 K_p(u) [u,h] + d K_p(u) [h] -
 dK_p(u) [h] = \\ d^2 K_p(u) [u,h] =  \scalarat q {\etransport p q u}{\etransport p q h} 
 = \scalarat p {\mtransport q p \etransport p q u}{h} \ .
\end{multline}

The second derivative in the directions $h$ and $k$ is
\begin{multline}
d^2 (d K_p(u) [u] - K_p(u)) [h,k] =  \\ d (d^2 K_p(u) [u,h]) [k] =  d^3
K_p(u) [u,h,k] + d^2 K_p(u)[k,h] = \\ \int (\etransport p {q} u) (\etransport p {q} h)(\etransport p {q} k) \ q \ d\mu +  \int (\etransport p {q} h)(\etransport p {q} k) \ q \ d\mu
\end{multline}

Both $\etransport p {q} h$ and $\etransport p {q} k$ are in the fiber
$B_q \subset \orliczof {\cosh_2} {q}$ but, in general, their product
$(\etransport p {q} h)(\etransport p {q} k)$ is an element of
$\cap_{\alpha > 1} L^\alpha(q)$ only. If actually $h,k \in \orliczof
{\operatorname{gauss}_2} q$, a simple algebraic expansion presents the symmetric part of the bilinear operator as a
function of the product $hk - \int hk \ p \ dm \in B_p$.
\end{example}

\begin{example}[Sub-models with constant expectation] Let be given a constant $b \in \R$, a reference density $p \in \maxexpat m$, and a random variable $f \in \orliczof {\cosh_2} m$ such that $\int f \ p \ dm = b$.  Consider the subset $E(f,p)$ of the maximal exponential model $\maxexpat m$ consisting of all densities $q$ such that $\int f \ q \ dm = \int f \ p \ dm$. It is a relatively open convex set. Note that $f - b \in B_p$ and the condition can be equivalently rewritten in terms of the coordinate $u = s_p(q)$. Namely, $\euler_p(u) = \euler^{u - K_p(u)} \cdot p = q \in E(B,p)$, if, and only if,
\begin{equation}
0 = \int (f - b) \ \euler_p(u) \ dm = dK_p(u)[f - b] \ .   
\end{equation}
Any tangent vector $h$ satisfies $d^2K_p(u)[h,f-b] = \covat q h {f-b} = 0$ For a smooth $F \colon \maxexpat m \to \R$, the expression at $p$ is $F_p(u) = F \circ \euler_p(u)$. Any extremal point of $F$ restricted to $E(f,p)$ satisfies $DF_p(u)[h] = 0$.

Consider the entropy $\entropyof q = \int \log q \ q \ dm$. As seen in Rem.~\ref{remark:entropy}, we have
$\Entropy_p(u) = dK_1(u)[u] - K_1(u)$ and $D\Entropy_1(u)[h] = \covat
{\euler_1(u)} h u$, so that the stationarity condition is $\covat
{\euler_1(u)} h f = 0$ whenever $ \covat {\euler_1(u)} h {f - b}$. In
conclusion $u \propto (f - b)$ and the stationary point has the form
$q = \euler^{\theta (f - b) - K_p(\theta(f -b))}$.

A similar argument holds for a finite number of random variables with given expected values. In the case of an infinite dimensional subspace of random variables, the subspace must be splitting.
\end{example}

\begin{example}[Pythagorean theorem for the KL divergence]
The Kullback-Leibler divergence is $\KL p q = \int \log \frac p q \ p \ dm$. Consider positive densities $q = \euler^{u - K_p(u)} \cdot p$, $q \in \maxexpat p$, and $r = (1+v) \cdot p$, $v \in \pstar B_p$. A simple computation shows that
\begin{equation*}
  \KL r q + \scalarat p u v = \KL r p + \KL p q \ .
\end{equation*}
The case where $\scalarat p u v = 0$ is sometimes called Pythagorean
Theorem for divergences. The result implies a conjugation statement. In fact, $ \scalarat p u v \leq \KL r p + \KL p q$ and $r=q$ gives $\scalarat p u v = \KL q p + \KL p q$.
\end{example}

\subsection{Covariant derivatives, tensor bundle, acceleration}

In discussing higher-order geometry, one needs to define bundles whose fibers are the product of multiple copies of the mixture and exponential fibers.

As a first example, the \emph{full bundle} is
\begin{equation} \label{fullbundle}
  \fullbundleat \mu = \setof{(q,\eta,w)}{q \in \maxexpat \mu, \eta \in
    \mixfiberat q \mu, w \in \expfiberat q \mu} \ .
\end{equation}

 There is a duality pairing $\mixfiberat q \mu \times \expfiberat q \mu \ni (\eta,w) \mapsto \scalarat q \eta w$ and the dual of $\mixfiberat q \mu$ is $\expfiberat q \mu$. The full bundle is our setup to discuss second-order geometry, such as covariant derivatives. 

More generally, ${}^h\!S^k$ will denote the case with $h$ mixture factors and $k$ exponential factors.  We use both notations for expected values in the following sections, $\int F \ p d\mu = \expectat p F$.
\\

Let us compute the expression of the velocity at time $t$ of a smooth
curve in the exponential bundle:
\begin{equation}\label{statcurve}
t \mapsto \gamma(t) = (q(t), w(t)) \in \expbundleat \mu = \tensorat 0 1 \mu \ ,  
\end{equation}
where $q(t) \in \maxexpat \mu$ and $w(t) \in \expfiberat {q(t)} \mu$ is a $q(t)$-centered random variable in the Orlicz space $\orliczof {\cosh_2} {\mu}$. 

In the chart centered at $p$, the expression of the curve is
\begin{equation}
\gamma_{p}(t) = s_p(\gamma(t))=\left(s_p(q(t)),\etransport {q(t)}p w(t)\right) \ ,  
\end{equation}
and, consequently, the time derivative has two components, 
\begin{multline}\label{eq:deriv1}
  \derivby t  s_p(q(t)) = \derivby t \left(\log \frac {q(t)} p - \expectat p {\log \frac {q(t)} p}\right) = \frac {\dot q(t)}{q(t)} - \expectat p {\frac {\dot q(t)}{q(t)}} = \\ \etransport {q(t)}{p} \frac {\dot q(t)}{q(t)} = \etransport {q(t)} p \derivby t \log q(t) \ ,
\end{multline}
and 
\begin{equation}\label{eq:deriv2}
  \derivby t  \etransport {q(t)} p w(t) = \derivby t \left(w(t) - \expectat p {w(t)}\right) = \dot w(t) - \expectat p {\dot w(t)} \ . 
\end{equation}

By expressing the tangent at each time $t$ in the chart centred at the current position $q(t)$, from the first component, we obtain the \emph{velocity},
\begin{equation}\label{eq:exponential-velocity}
\velocity q(t) =  \etransport p {q(t)} \derivby t s_p(q(t)) = \dot u(t) - \expectat {q(t)}{\dot u(t)} = \derivby t \log q(t) = \frac {\dot q(t)}{q(t)} \ .
\end{equation}

Notice that $t \mapsto (q(t),\velocity q(t))$ is a curve in the exponential statistical bundle whose expression in the chart centered at $p$ is $t \mapsto (u(t),\dot u(t))$. The mapping $q \mapsto (q,\velocity q)$ is the \emph{lift} of the curve to the statistical bundle.

 Let us turn to the interpretation of the second component in
\eqref{eq:deriv2}. In terms of the exponential parallel transport in \eqref{expart}, we define an exponential \emph{covariant} derivative by setting
\begin{multline}\label{eq:exponential-derivative}
\Derivby t w(t) = \etransport p {q(t)} \derivby t \etransport {q(t)} p
w(t) = \\ \etransport p {q(t)} \left( \dot w(t) - \expectat p {\dot w(t)}\right) = \dot
w(t) - \expectat {q(t)} {\dot w(t)} \ .  
\end{multline}

The notation $\,\Derivby t \,$ will generically denote the covariant time derivative in a given transport or connection, whose choice will depend on the context. When necessary, we use $\eDeriv$, $\mDeriv$, or similar notations.

Let us now do the computation in the \emph{mixture bundle}. The curve is  
\begin{equation}
t \mapsto \zeta(t) = (q(t),\eta(t)) \in \mixbundleat \mu = \tensorat 1 0 \mu
\end{equation}
The computation in the first component is the same as above. The expression of the second component in the chart
is $\mtransport {q(t)} p \eta(t) = \frac {q(t)} p \eta(t)$. The derivation is
\begin{equation}
\derivby t \mtransport {q(t)} p \eta(t) = \derivby t \frac {q(t)} p
\eta(t) = \frac 1p\left(\dot q(t) \eta(t) + q(t)\dot \eta(t)\right) \ ,
\end{equation}
which, in turn, defines the mixture covariant derivative as 
\begin{multline}\label{eq:mixture-derivative}
\Derivby t \eta(t) = \mtransport p {q(t)} \derivby t \mtransport {q(t)} p
\eta(t) = \\ \frac {p}{q(t)} \frac 1p \left(\dot q(t) \eta(t) + q(t)\dot \eta(t)\right)
= \velocity q(t) \eta(t) + \dot \eta(t) \ .\end{multline}

A basic computation in the full statistical bundle is the variation of the duality
pairing. The covariant derivatives in eq. \eqref{eq:exponential-derivative} and eq. \eqref{eq:mixture-derivative} are compatible with the duality pairing.

\begin{proposition}[Duality of the covariant derivatives]\label{prop:covariantduality}
For each smooth curve in the full statistical bundle,
\begin{equation}
  t \mapsto (q(t),\eta(t),w(t)) \in \fullbundleat \mu \ ,
\end{equation}
it holds
\begin{equation}  \label{eq:d-inner}
  \derivby t \scalarat {q(t)} {\eta(t)}{w(t)} = 
\scalarat {q(t)} {\mDerivby t \eta(t)}{w(t)} + \scalarat {q(t)}
{\eta(t)} {\eDerivby t w(t)} \ .
\end{equation}
\end{proposition}

\begin{proof}
  In an arbitrary reference density $p$,
  \begin{multline*}
   \derivby t \scalarat {q(t)} {\eta(t)}{w(t)} =   \derivby t
   \scalarat {p} {\mtransport {q(t)} p \eta(t)}{\etransport {q(t)} p
     w(t)} = \\ \scalarat {p} {\derivby t \mtransport {q(t)} p \eta(t)}{\etransport
     {q(t)} p w(t)} + \scalarat {p} {\mtransport {q(t)} p
     \eta(t)}{\derivby t \etransport {q(t)} p w(t)} = \\ \scalarat {q(t)}
   {\mtransport p {q(t)} \derivby t \mtransport {q(t)} p
     \eta(t)}{w(t)} + \scalarat {p} {\eta(t)}{\etransport p {q(t)} \derivby t \etransport {q(t)} p w(t)} \ ,   
 \end{multline*}
which is Eq.~\eqref{eq:d-inner}. \smartqed
\end{proof}

\begin{remark}[Riemannian derivative]
In \cite{amari|nagaoka:2000} another covariant derivative is defined for $t \mapsto (q(t),w(t)) \in \expbundleat \mu$. Because of the embedding $\expbundleat \mu \subset \mixbundleat \mu$, we can define 
\begin{equation}\label{eq:0acc}
  \hDerivby t w(t) = \frac12 \left(\mDerivby t w(t) + \eDerivby t w(t)\right) \ .
\end{equation}
The remarkable property of this derivative is its compatibility with the inner product. If $t \mapsto (q(t),v(t),w(t)) \in \tensorat 0 2 \mu$, a straightforward computation shows that
\begin{equation}
\derivby t \scalarat {q(t)} {v(t)}{w(t)} = \scalarat {q(t)} {\hDerivby t  v(t)}{w(t)} + \scalarat {q(t)} {v(t)}{\hDerivby t w(t)} \ . 
\end{equation}
\end{remark}

Let us show that both the covariant derivatives we have defined in Eq.s~\eqref{eq:exponential-derivative} and \eqref{eq:mixture-derivative} deserve their name. We adapt the presentation of \cite[sec.~2.2]{docarmo:1992}. All curves and fields are assumed to be smooth.

\begin{proposition}
Both exponential and mixture covariant derivatives satisfy the following equations.
\begin{gather}
\Derivby t (X(t) + Y(t)) = \Derivby t X(t) + \Derivby t Y(t) \\
\Derivby t f(t) v(t) = \dot f(t) X(t) + f(t) \Derivby t X(t)
\end{gather}
\end{proposition}

\begin{proof}
Both equations follow immediately from the definitions, that is, $\Derivby t X(t) = \transport p {q(t)} \derivby t \transport {q(t)} p X(t)$.  \smartqed
\end{proof}

\begin{remark}
Let us briefly justify our unusual presentation of a classical topic such as covariant derivatives. As the manifold structure we discuss is an affine space with global charts, the geometry of the tangent bundle of the manifold follows in terms of explicitly defined parallel transports on its expression as a statistical bundle. See \cite{gibilisco|pistone:98}. Because of the non-parametric set-up and of the statistical application of interest, the covariant derivatives are operations on smooth curves.

The usual covariant derivatives of vector fields, as described, for example, in \cite[Ch.~VIII]{lang:1995}, could be defined, whether it is needed, as follows. If $F$ is a smooth section of the statistical bundle, $F(q)$ in the fiber at $q$, for each smooth curve $t \mapsto q(t)$ one could compute $\Derivby t F(q(t))$ and looks at its representation as a linear operator $d F (q(t))$ applied to $\velocity q(t)$. Here $DF$ would be the covariant derivative defined on smooth sections. As we have a specific representation of the linear operators of fibers represented precisely by the full statistical bundle $\tensorat 1 1 \mu$, we prefer to talk about covariant gradients instead of covariant derivatives.
\end{remark}

We now define the \emph{second statistical bundle} to be
\begin{equation}
\tensorat 0 3 \mu = \setof{(q,w_1,w_2,w_3)}{(q \in \maxexpat \mu,w_1,w_2,w_3 \in \expfiberat q \mu} \ ,
\end{equation}
with charts centered at $p \in \maxexpat \mu$ defined by
\begin{equation}
  s_p(q,w_1,w_2,w_3) = \left(s_p(q),\etransport q p w_1,\etransport q p w_2,\etransport q p w_3\right) \ .  
\end{equation}

The second bundle is an expression of the statistical bundle of the exponential statistical bundle, $S \expbundleat \mu$. That is, for each curve $t \mapsto \gamma(t) =
(q(t),w(t)) \in \expbundleat \mu$, we define its \emph{velocity
  at $t$} to be
\begin{equation}
  \velocity \gamma (t) =  \left(q(t),w(t),\velocity q(t),\Derivby t w(t)\right) \ ,
\end{equation}
because $t \mapsto \velocity \gamma(t)$ is a curve in the second statistical bundle and that its expression in the chart at $p$ has the last two components equal to the values given in \eqref{eq:deriv1} and \eqref{eq:deriv2}, respectively. 

For each smooth curve $t \mapsto q(t)$, the velocity of
its  lift $t \mapsto \chi(t) = (q(t),\velocity
q(t))$ is 
\begin{equation}\label{eq:velocity-2}
\velocity \chi(t) = \left(q(t),\velocity q(t), \velocity q(t), \acceleration q(t)\right) \ ,  
\end{equation}
where $\acceleration q(t)$ defines the \emph{exponential acceleration} at $t$, 
\begin{equation}\label{eq:acceleration}
\acceleration q(t) = \Derivby t \velocity q(t) = \derivby t \frac{\dot q(t)}{q(t)}   - \expectat {q(t)} {\derivby t \frac{\dot q(t)}{q(t)}} = \frac {\ddot q(t)}{q(t)} - \Big(\velocity q(t)^2 - \expectat {q(t)} {\velocity q(t)^2}\Big) \ .\end{equation}

As the two middle components of the RHS of Eq.~\eqref{eq:velocity-2} are equal, the acceleration is defined in $\tensorat 0 2 \mu$.

\begin{remark}[e-Geodesic] The acceleration defined above has the one-dimensional exponential families as a (differential) geodesic. Every exponential curve $t \mapsto q(t) = \euler_p(t u)$ has velocity $\velocity q(t) = u - d K(t u)[u]$ so that the acceleration is $\acceleration q(t) = 0$. Conversely, if one writes $v(t) = \log q(t)$, then
\begin{equation}
 0 = \acceleration q(t) = \ddot v(t) + \expectat {q(t)} {\ddot v(t)} \ ,
\end{equation}
so that $v(x;t) = t v(x) + c$.
\end{remark}

We have defined the \emph{exponential acceleration} $\eDerivby t \velocity q(t) = \acceleration
q(t)$ via exponential transport in \eqref{eq:acceleration}. Similarly, we  define the \emph{mixture acceleration}, via mixture transport, as
\begin{equation}\label{eq:macc}
 \mDerivby t \velocity q(t) = \transport p {q(t)} \derivby t  \transport {q(t)} p
\velocity q(t) = \ddot q(t)/q(t) \ .
\end{equation}

\begin{example}[The SIR model as a second order equation]
The Kermack and McKendrick SIR model is a differential equation on the positive probability densities on $\set{S,I,R}$. As an equation in  the statistical bundle is
\begin{equation*}
  \begin{cases}
    \velocity p(S;t) &= -\beta p(I;t)  \\
    \velocity p(I;t) &= \beta p(S;t) - \gamma \\
    \velocity p(R;t) &= \gamma  p(I;t)/p(R;t)
  \end{cases}
\end{equation*}

The mixture acceleration $\macc p = \ddot p / p$ as a function of $(p(t),\velocity p(t))$ is linear in each fiber:
  \begin{equation*}
    \macc p(t) = \begin{bmatrix} - \beta p(I;t) & - \beta p(I;t) & 0 \cr
      \beta p(S;t) & (\beta p(S;t) - \gamma) & 0 \cr
      0 & \gamma & 0 
    \end{bmatrix} \velocity p(t) 
  \end{equation*}
\end{example}

\subsection{Gradients on the Tensor Bundles}

Given a scalar field $F \colon \maxexpat \mu \to \R$ the \emph{gradient} of $F$ is the section $q \mapsto \Grad F(q)$ of the mixture bundle $\mixbundleat \mu$ such that for all smooth curve $t \mapsto q(t) \in \maxexpat \mu$ it holds
\begin{equation}\label{eq:naturalgradient}
  \derivby t F(q(t)) = \scalarat {q(t)} {\Grad F(q(t))}{\velocity  q(t)} \ .
\end{equation}

\begin{example}[Gradient of the entropy]
  The computation of the natural gradient does not require, in some cases, the computation in charts. For example, if conditions for existence and smoothness are satisfied, the derivative of the entropy function $\entropyof {q(t)}= \expectat 1 {q(t) \log q(t)}$ along the curve $t \mapsto q(t)$ is 
\begin{multline*}
    \derivby t \entropyof{q(t)} = -\derivby t \expectat 1 {q(t) \log q(t)} = -\expectat 1 {\dot q(t)(\log q(t) + 1)} = \\  - \expectat {q(t)} {\log q(t) \velocity q(t)} = \scalarat {q(t)} {- \log q(t) - \entropyof {q(t)}}{\velocity q(t)} \ ,
\end{multline*}
hence $\Grad \entropyof q = - \log q - \entropyof q$.

More precisely, in coordinates, we have
\begin{multline*}
    \entropyof {\euler_p(u(t))}=- \int (u(t) - K_p(u(t)) \ \euler_p(u(t)) \ d\mu = \\ - dK_p(u(t))[u(t)] + K_p(u(t)] \ . 
\end{multline*}
With cancellations and centering, the derivative is
\begin{equation*}
    d^2 K_p(u(t)][u(t),\dot u(t)] = \scalarat {q(t)}  {-\log q(t) - \expectat {q(t)} {-\log q(t)}} {\velocity q(t)} \ .
\end{equation*}
\end{example}

In a chart, the gradient is expressed as a function of the ordinary gradient $\nabla F_p$ of $F$. In the generic chart at $p$, with $q = \euler_p(u)$ and $F(q) = F_p(u)$, it holds
\begin{multline*}
\scalarat {q(t)} {\Grad F(q(t))}{\velocity q(t)} = \derivby t F(q(t)) = \derivby t F_p(u(t))
  = DF_p(u(t))[\dot u(t)] = \\ 
  DF_p(u(t))[\transport {q(t)} p \velocity q(t)] 
  = \scalarat {p} {p^{-1} \nabla F_p(u(t)) }{\transport {q(t)} p \velocity q(t)} = \\ \scalarat {q(t)} {\transport p {q(t)} p^{-1} \nabla F_p(u(t)) }{\velocity q(t)} 
  =\scalarat {q(t)} {q^{-1} \nabla F_p(u(t)) }{\velocity q(t)} =  \\ \scalarat {q(t)} {q^{-1} \nabla F_p(u(t) ) - \expectat {q(t)} {q^{-1} \nabla F_p(u(t) )} }{\velocity q(t)} \ ,
\end{multline*} 
where, in the pairing, mixture and exponential transports consistently act on the fibers.

\begin{remark}[Natural gradient] The definition of gradient above is a non-parametric version of the natural gradient introduced by S-I Amari; see \cite{amari:1998natural}. Consider a $d$-dimensional statistical model $\Theta \ni \theta \mapsto q(\theta) \in \maxexpat \mu$, $\Theta \subset \R^d$ open. The variation along the curve $\theta_i \mapsto q(\theta)$ is
  \begin{equation*}
    \pderivby {\theta_i} F(q(\theta)) = \scalarat {q(\theta(t))} {\Grad F(\theta)} {\pderivby {\theta_i} \log q(\theta)} \ .
  \end{equation*}
  Assume there is a $q(\theta)$-orthogonal projection of $\Grad F(q(\theta))$ onto the space generated by $\pderivby {\theta_1} \log q(\theta),\dots, \pderivby {\theta_d} \log q(\theta)$.Hence,
   \begin{multline*}
    \pderivby {\theta_i} F(q(\theta)) = \sum_{j=1}^d \widehat F_{j}(\theta) \scalarat {q(\theta(t))} {\pderivby {\theta_j} \log q(\theta)} {\pderivby {\theta_i} \log q(\theta)} = \\ \sum_{j=1}^d \widehat F_j(\theta) I_{ij}(\theta) \ ,
  \end{multline*}
 where the matrix $I(\theta)$ is the \emph{Fisher matrix} of the given model. The last equation presents the natural gradient $\widehat F_j(\theta)$ as the inverse Fisher matrix applied to the ordinary gradient for the parameters. Notice that the non-parametric setup clarifies an explicit set of assumptions to justify the computation.   
\end{remark}

\begin{example}[Gradient Flow of the entropy: cf. ex.~\ref{ex:running-4}]\label{entropy}
  The integral curves of the gradient flow equation
\begin{equation*}
\velocity q(t) = \Grad \entropyof {q(t)}  
\end{equation*}
are exponential families of the form $q(t) \propto
q(0)^{\euler^{-t}}$. If we write the equation in $\R^N$,
we get the quasi-linear ODE
\begin{equation*}
 \derivby t \log q(t) = -\log q(t) - \entropyof {q(t)} \ ,  
\end{equation*}
and, in turn,
\begin{equation*}
  \log q(t) = \euler^{-t} \log q(0) - \euler^{-t} \int_0^t \euler^s
  \entropyof {q(t)} \ ds \ .
\end{equation*}
The behavior as $t \to \pm \infty$ and other properties follow quickly; see \cite{pistone:2013Entropy}.

Given a section $q \mapsto F(q) \in \expfiberat q \mu$, the variation
of the entropy along the integral curves, $\velocity q(t) = F(q(t))$, is
\begin{multline*}
  \derivby t \entropyof{q(t)} =  \scalarat {q(t)} {\Grad \entropyof
    {q(t)}} {F(q(t))} \\ =  - \scalarat {q(t)} {\log q(t) + \Grad \entropyof
    {q(t)}} {F(q(t))} \ .
  \end{multline*}
  For example, the condition for entropy production becomes
  \begin{equation*}
    \scalarat {q} {\log q + \entropyof {q(t)}} {F(q)} = \expectat q {\log q F(q)} < 0 \ .
  \end{equation*}
\end{example}

The definition of the natural gradient extends to functions defined on the full statistical bundle $\fullbundleat \mu$. Consider a real function $F \colon \fullbundleat \mu \times \mathcal D \to \R$, where $\mathcal D$ is a domain of $\R^k$. For a generic smooth curve
\begin{equation}
  t \mapsto (q(t),\eta(t),w(t),c(t)) \in \fullbundleat \mu \times
  \mathcal D \ ,
\end{equation}

we want to write 
\begin{multline}\label{eq:totalderivative}
  \derivby t F\big(q(t),\eta(t),w(t),c(t)\big) = \\ \scalarat {q(t)} {
    \Grad F\big(q(t),\eta(t),w(t),c(t)\big) }{\velocity q(t)} + 
\scalarat
  {q(t)} {\Derivby t \eta(t)} { \Grad_\text{m}
    F\big(q(t),\eta(t),w(t),c(t)\big)} + \\   \scalarat {q(t)} {\Grad_\text{e} F\big(q(t),\eta(t),w(t),c(t)\big)}{\Derivby t w(t)}
  + \nabla F\big(q(t),\eta(t),w(t),c(t)\big) \cdot \dot c(t) \ ,
\end{multline}
where the four components of the gradient are
\begin{equation}
\fullbundleat \mu \times \mathcal D \ni (q,\eta,w,c) \mapsto \begin{cases} (q,\Grad F\big(q,\eta,w,c\big))
  \in \mixfiberat q \mu \\
(q,\Grad_\text{m} F\big(q,\eta,w,c\big))
  \in \expfiberat q \mu \\  (q,\Grad_\text{e} F\big(q,\eta,w,c\big))
               \in \mixfiberat q \mu \\  (q,\nabla  F\big(q,\eta,w,c\big))
  \in \maxexpat \mu \times \R^k
\end{cases}
\end{equation}

\begin{proposition}\label{prop:totalderivative} 
The total derivative \eqref{eq:totalderivative} holds true, where
\begin{enumerate}
\item \label{prop:totalderivative1} $\Grad F\big(q,\eta,w,c\big)$ is the natural gradient of
  \begin{equation}
    q
  \mapsto F(q,\transport p q \zeta,\transport p q v,c) \ ,
  \end{equation}
that is, with the representation in $p$-chart
\begin{equation}
  F_p(u,\zeta,w,c) = F(\euler_p(u),\transport p {\euler_p(u)} \zeta,\transport p {\euler_p(u)} v,c) \ ,
\end{equation}
it is defined by
\begin{equation}
      \scalarat{q}{\Grad F(q,\zeta,w,c)}{\velocity q} = d_1F_p(u,\zeta,w,c) \left[\transport{q}{p} \velocity q\right] \ , \quad (q,\velocity q) \in \expbundleat \mu \ ;
  \end{equation}
\item \label{prop:totalderivative2} $\Grad_\text{m} F\big(q,\eta,w,c\big)$ and $\Grad_\text{e}
  F\big(q,\eta,w,c\big)$ are the fiber gradients;
\item $\nabla F\big(q,\eta,w,c\big)$ is the Euclidean gradient w.r.t.~the last variable.
\end{enumerate}
\end{proposition}

\begin{proof}
  Let us fix a reference density $p$ and express both the given function and the
generic curve in the chart at $p$. We can write the total derivative as
\begin{multline*}
  \derivby t F \big(q(t),\eta(t),w(t),c(t)\big) = \derivby t
   F_p(u(t),\zeta(t),v(t),c(t)) = \\
  D_1 F_p \big(u(t),\zeta(t),v(t),c(t)\big) \big[\dot u(t) \big] + D_2
  F_p \big(u(t),\zeta(t),v(t),c(t)\big) \big[\dot \zeta(t) \big] + \\
   D_3 F_p\big(u(t),\zeta(t),v(t),c(t)\big) \big[\dot v(t) \big] +
   D_4 F_p\big(u(t),\zeta(t),v(t),c(t)\big) \big[\dot c(t) \big] \ . 
\end{multline*}

In the equation above, $D_j$, with $j=1,\dots,4$, denotes the partial derivative with
respect to the $j$-th variable of $F_p$,which is
intended to provide a linear
operator represented by the appropriate dual vector, that is,
the value of the proper gradient.

The last term of the total derivative does not require any comment, and we can write it as the ordinary
Euclidean gradient:
\begin{equation*}
  D_4 F_p\big(u(t),\zeta(t),v(t),c(t)\big) \big[\dot c(t) \big] =
  \nabla F_p\big(u(t),\zeta(t),v(t),c(t)\big) \cdot \dot c(t) \ .
\end{equation*}

Let us consider the second and third terms together. It is a
computation of the fiber derivative and does not involve the
representation in the chart. Given
$\alpha \in \mixfiberat p \mu$ and $\beta \in \expfiberat p \mu$, that
is, $(\alpha,\beta) \in \fullfiberat p \mu$, we have
\begin{multline*}
  D_2 F_p(u,\zeta,v,c) [\alpha] + D_3 F_p(u,\zeta,v,c) [\beta] =
  \left. \derivby t F_p(u,\zeta+t\alpha,w+t\beta,c) \right|_{t=0} = \\
  \left. \derivby t F(q,\eta +t \transport p q \alpha,v + t
    \transport p q \beta,c) \right|_{t=0} = \mathbb F
  F(q,\eta,w,c)[(\transport p q \alpha,\transport p q
  \beta)] = \\
  \scalarat q {\transport p q \alpha} {\Grad_\text{m} F(q,\eta,w,c)}
  + \scalarat q {\Grad_\text{e} F(q,\eta,w,c)} {\transport p q \beta}
  \ ,
\end{multline*}
where $\mathbb F$ denotes the fiber derivative in $\fullfiberat q
\mu$, which is expressed, in turn, with the relevant gradients. Consider that the inner
product always has $\mixfiberat q \mu$ first, followed by $\expfiberat
q \mu$ and that the subscript to the $\Grad$ symbol displays which
component of the full bundle is considered.

We have that
\begin{equation*}
  \Derivby t w(t) = \transport p {q(t)} \dot v(t) \ , \quad \Derivby
  t \eta (t) = \transport p {q(t)} \dot \zeta(t) \ .
\end{equation*}
Putting together all the contributions, we have proved that
\begin{multline*}
  \derivby t F\big(q(t), \eta(t),w(t),c(t)\big) =  D_1 F_p
                                                   \big(u(t),\zeta(t),v(t),c(t)\big)
                                                  \big[\transport
                                                   {\euler_p(u(t))} p
                                                   \velocity q(t)  \big]  + \\
                                                 \scalarat
                                                   {q(t)} {\Derivby t \eta(t)} { \Grad_\text{m}
                                                   F\big(q(t),\eta(t),w(t),c(t)\big)} + \\  \scalarat {q(t)} {\Grad_\text{e} F\big(q(t),\eta(t),w(t),c(t)\big)}{\Derivby t w(t)}
                                                                                               + \\ \nabla F\big(q(t),\eta(t),w(t),c(t)\big) \cdot \dot c(t) \ ,
\end{multline*}

To identify the first term in the total derivative above, consider the geodesic'' case,
\begin{equation*}
  q(t) = \euler_p(u(t)), \quad \eta(t) =
  \transport p {\euler_p(u(t))} \zeta,  \quad w(t) = \transport p {\euler_p(u(t))} v, \quad c(t) = c \ ,
\end{equation*}
so that the first term reduces to $D_1 F_p(u(t),\zeta,v,c) [\transport
                                                   {\euler_p(u(t))} p
                                                   \velocity q(t) ]$. It
follows that the proper way to compute the first gradient is to
consider the function on $\maxexpat \mu$ defined by 
 \begin{equation*}
 q \mapsto F_{\zeta,v,c}(q) = F(q,\transport p q \zeta,\transport p q v,c)
 \end{equation*}
which  has a natural gradient whose chart representation is precisely
that first term. \smartqed
\end{proof}

We have concluded the computation of the total derivative of a parametric function on the full bundle. Notice that the computation of the natural gradient for $\Grad_\text{m} F\big(q,\eta,w,c\big)$ and $\Grad_\text{e}
  F\big(q,\eta,w,c\big)$ is done by fixing the variables in the fibers to be translations of fixed ones.

We are going to discuss the following examples of gradient flow on the full statistical bundle: the scalar function $L(q,w) = \frac12 \scalarat q w w$; the \emph{cumulant} function $L(q,w) = K_q(w)$; the \emph{conjugate cumulant} function $H(q,\eta) = \expectat q {(1+\eta) \log (1 + \eta)}$.

\begin{example}[Scalar function $\frac12 \scalarat q w w$]\label{ex:quadratic-lagrangian}
On the statistical bundle, consider the scalar function given by the pairing 
\begin{equation*}
 L(q,w) = \frac12 \scalarat q w w   \ . 
\end{equation*}

In chart, we have  $L_p(u,v) = L\left(\euler_p(u),\etransport p {\euler_p(u)} v\right) = d^2K_p(u)[v,v]$  from Eq.~\eqref{eq:K3}, where $K_p(u)$ is the expression in the chart at $p$ of Kullback–Leibler divergence of $q\mapsto \KL p q$. 

From Eq.~\eqref{eq:K4}, we write the derivative with respect to $u$ in the direction $h$ as
  \begin{multline*}
\frac 12 d^3 K_p(u)[v,v,h] = \frac12 \expectat {\euler_p(u)}{\left(v -  \expectat {\euler_p(u)} v\right)^2 \etransport p {\euler_p(u)} h} = \\ \frac 12 \scalarat q {w^2 - \expectat q w^2}{ \etransport p q h}
  \end{multline*}
  which, in turn, identifies the gradient as $\Grad \frac12 \scalarat q w w = \frac12 (w^2 - \expectat q {w^2}) \in \mixbundleat \mu$. The exponential gradient is $\Grad_{\text{e}} \frac 12 \scalarat q w w = w$.
\end{example}

\begin{example}[Cumulant functional]\label{ex:cumulant-lagrangian}
If  $L(q,w) = K_q(w)$, then  
  \begin{multline*}
  L_p(u,v) =
    K_{\euler_p(u)}\left(\etransport p {\euler_p(u)} v\right) = 
    \log \expectat {\euler_p(u)}  {\euler^{v - \expectat {\euler_p(u)} v}}x = \\
    \log\expectat p {\euler^{u - K_p(u) + v - \expectat {\euler_p(u)} v}} =
    \log\expectat p {\euler^{u + v - K_p(u)- dK_p(u)[v] }} = \\
    K_p(u+v) - K_p(u) - dK_p(u)[v] \ .
  \end{multline*}
The derivative with respect to $u$ in the direction $h$ is
  \begin{multline*}
    dK_p(u+v)[h] - dK_p(u)[h] - d^2K_p(u)[v,h] = \\ \expectat
    {\euler_p(u+v)} h - \expectat {\euler_p(u)} h - \expectat {\euler_p(u)}
    {\left(\transport p {\euler_p(u)} v\right) \left(\transport p {\euler_p(u)} h\right) } =
    \\ \expectat {\euler_p(u)} {\frac {\euler_p(u+v)}{\euler_p(u)} h} - \expectat {\euler_p(u)} h -
    \expectat {\euler_p(u)} {w \left(\transport p {\euler_p(u)} h\right)} = \\
 \expectat {\euler_p(u)} {\frac {\euler_p(u+v)}{\euler_p(u)} \left(\transport p
     {\euler_p(u)} h\right)} - \scalarat  q {w} {\transport p {\euler_p(u)} h}   \ .
  \end{multline*}
 The expected value of the factor $\frac {\euler_p(u+v)}{\euler_p(u)}$ in the first term of the RHS equals
  \begin{multline*}
    \expectat {\euler_p(u)} {\euler^{v - (K_p(u+v)-K_p(u))}
      \left(\etransport p {\euler_p(u)} h\right)} = \\ \expectat {\euler_p(u)}
    {\euler^{v - (K_{\euler_p(u)}(\transport p {\euler_p(u)} v) + dK_p(u)[v])}
      \left(\transport p
        {\euler_p(u)} h\right)} = \\
    \expectat {\euler_p(u)} {\euler^{\etransport p {\euler_p(u)} v -
        K_{\euler_p(u)}(\transport p {\euler_p(u)} v)} \left(\etransport p
        {\euler_p(u)} h\right)} = \\
    \expectat q {\euler^{w - K_q(w)} \left(\transport p {\euler_p(u)}
        h\right)} = \scalarat q {\frac{e_q(w)}{q} - 1}{\etransport p
        {\euler_p(u)} h}  \ .
  \end{multline*}
In conclusion, $\Grad K_q(w) = \left(\frac{e_q(w)}{q} - 1\right) -
w$. The expectation gradient is
\begin{equation*}
  \Grad_{\text{e}} K_q(w) = \frac{\euler_q(w)}{q} - 1 \ .
\end{equation*}
\end{example}

\begin{example}[Conjugate cumulant functional]\label{ex:cumulant-hamiltonian}
The conjugate cumulant functional
\begin{equation*}
 \mixbundleat \mu \colon (q,\eta) \mapsto H(q,\eta) = \expectat q {(1+\eta) \log (1 + \eta)} \ , \quad \eta > -1 \ ,   
\end{equation*}
is the Legendre transform of the cumulant function $K_q$,
\begin{equation*}
  H(q,\eta) = \scalarat q \eta {(\Grad K_q)^{-1}(\eta)} -
  K_q\left((\Grad K_q)^{-1}(\eta)\right) \ .
\end{equation*}
In particular, the fiber gradient of $H_q$ is $\Grad_{\text{m}} H(q,\eta) = \log(1+\eta) - \expectat q {\log(1+\eta)}$ which is the inverse of the fiber gradient of $K_q$. Notice that $r = (1+\eta)q$ is a density, and $\KL r q = H(q,\eta)$. 

Let us compute the gradient. The expression of the conjugate cumulant functional in the chart at $p$ is 
\begin{multline*}
 H_p(u,\zeta) = \expectat {\euler_p(u)} {\left(1 + \frac p {\euler_p(u)} \zeta \right) \log \left(1 + \frac p {\euler_p(u)} \zeta \right)}  = \\ \expectat p {\left(\frac{\euler_p(u)} p  + \zeta \right) \log \left(1 + \frac p {\euler_p(u)} \zeta \right)} \ .  
\end{multline*}
As, for $h \in \expfiberat p \mu$,
\begin{gather*}
  d \left(\frac{\euler_p(u)} p  + \zeta \right) [h]= \frac{\euler_p(u)}p
  \mtransport p {\euler_p(u)} h \ , \\  d \left(1 + \frac
    p {\euler_p(u)} \zeta \right) [h]= - \frac p{\euler_p(u)} \zeta \mtransport p
  {\euler_p(u)}     h \ ,
\end{gather*}
the derivative of $H_p$ with respect to $u$ in the direction $h$ is given by
\begin{multline*}
  DH_p(u,\zeta) [h]= \expectat p {\left(\frac{\euler_p(u)}p \mtransport p
      {\euler_p(u)} h\right) \log \left(1 + \frac p {\euler_p(u)}
      \zeta \right)} - \\
  \expectat p {\left(\frac{\euler_p(u)} p + \zeta \right) \left(1 + \frac p
      {\euler_p(u)} \zeta \right)^{-1}\frac p{\euler_p(u)} \zeta \mtransport p
    {\euler_p(u)}     h} = \\
  \expectat q {\log(1+\eta)  \mtransport p {\euler_p(u)} h} - \expectat q
  {\zeta \transport p {\euler_p(u)} h}\ ,
\end{multline*}
hence $\Grad H(q,\eta) = \log(1+\eta) - \expectat q {\log(1+\eta)} - \eta$.
\end{example}

\subsection{Lagrangian and Hamiltonian formalisms on the full Statistical Bundle}
\label{sec:mechanics}
The dually affine geometry of the statistical bundle is naturally well suited for describing the dynamics of probability densities in a Lagrangian and Hamiltonian formalism; see \cite{pistone:2018lagrange, chirco|malago|pistone:2020}. This is apparent from the previous examples.

The Lagrangian formulation of mechanics derives the fundamental laws of force balance from variational principles. In our context, the exponential model $\maxexpat \mu$ corresponds to the configuration space, while the statistical bundle is associated with the velocity phase space.
For a given smooth curve $q \colon [0,1] \ni t \mapsto q(t)$ in $\maxexpat \mu$ and its lift $t \mapsto (q(t),\velocity q(t)) \in \expbundleat \mu$, we introduce a generic Lagrangian function
\begin{equation}
  L(q(t),\velocity q(t)) \colon \expbundleat \mu \times [0,1] \to \R
\end{equation}
and define an action as the integral of the Lagrangian along the curve over the fixed time interval $[0,1]$,
\begin{equation}
q \mapsto   \mathcal{A}(q) = \int_{0}^{1} L(q(t),\velocity q(t),t) \ dt \ .
\end{equation}

Hamilton’s principle states that this function has a critical point at a solution within the space of curves on $\maxexpat \mu$. We have

\begin{proposition}[Euler-Lagrange equation]
  If $q$ is an extremal of the action integral, then 
  \begin{equation}\label{eq:Euler-Lagrange}
  \Derivby t \Grad_\text{e} L(q(t),\velocity q(t),t) = \Grad L(q(t),\velocity q(t),t) \ .
\end{equation}
\end{proposition}

\begin{proof}
Let us express the action integral in the exponential chart $s_p$ centered at $p$. If $q(t) = \euler^{u(t)-K_p(u(t))}\cdot p$, with $t \mapsto u(t) \in \expfiberat p \mu$, we have
\begin{equation*}
L(q(t),\velocity q(t),t) = L\left(\euler_p(u(t)),\etransport p {\euler_p(u(t))} \dot
u(t),t\right) = L_p(u(t),\dot u(t),t) \ , 
\end{equation*}
so that the expression of the Euler-Lagrange equation in a chart is given by 
\begin{equation}\label{eq:EL-chart}
   D_1 L_p(u(t),\dot u(t),t)[h] = \derivby t D_2 L_p(u(t),\dot u(t),t)[h] \ . \quad t \in [0,1] \ , h \in \expfiberat p \mu \ . 
\end{equation}
Consider first the RHS of eq.~\eqref{eq:EL-chart}. From Proposition \eqref{prop:totalderivative} we have
  \begin{equation*}
  D_1 L_p(u(t),\dot u(t),t)[h] = \scalarat {q(t)}
  {\Grad L(q(t),\velocity q(t),t)}{\etransport p {q(t)} h} \ .
  \end{equation*}
On the left-hand side, we have
  \begin{equation*}
  D_2 L_p(u(t),\dot u(t))[h] = \scalarat {q(t)}
  {\Grad_\text{e}L(q(t),\velocity q(t),t)}{\etransport p {q(t)} h} \ .
  \end{equation*}
The derivation formula of \eqref{eq:d-inner} gives
\begin{multline*}
\derivby t   D_2 L_p(u(t),\dot u(t),t)[h] = \derivby t \scalarat {q(t)}
  {\Grad_\text{e} L(q(t),\velocity q(t),t)}{\etransport p {q(t)} h} = \\ \scalarat {q(t)}
  {\Derivby t \Grad_\text{e} L(q(t),\velocity q(t),t)}{\etransport p {q(t)} h} + \scalarat{q(t)}{\Grad_\text{e} L(q(t),\velocity q(t),t)}{\Derivby t \etransport{p}{q(t)} h} = \\ \scalarat {q(t)}
  {\Derivby t \Grad_\text{e} L(q(t),\velocity q(t),t)}{\etransport p {q(t)} h} \ ,
\end{multline*}
because $\Derivby t \etransport{p}{q(t)} h = 0$. As $h$ is arbitrary, the conclusion follows.
\end{proof}

\subsubsection{Hamiltonian mechanics}
\label{sec:legendre-transform}
At each fixed density $q \in \maxexpat \mu$, and each time $t$, the
partial mapping
$\expfiberat q \mu \ni w \mapsto L_{q,t}(w) = L(q, w, t)$ is defined on
the vector space $\expfiberat q \mu$, and its gradient mapping in the
duality of $\mixfiberat q \mu \times \expfiberat q \mu$ is
$w \mapsto \Grad_\text{e} L(q,w,t)$. The standard Legendre transform
the argument provides the intrinsic form of the Hamilton equations under
the following assumption.

\begin{assumption}\label{h1}
  We restrict our attention to Lagrangians such that the
  fiber gradient mapping at $q$,
  $w \mapsto \eta = \Grad_\text{e} L_q(w)$ is a 1-to-1 mapping from
  $\expfiberat q \mu$ to $\mixfiberat q \mu$. In particular, this true
  when the partial mappings $w \mapsto L_q(w)$ are strictly convex for
  each $q$. 
\end{assumption}

In the finite-dimensional context, this assumption is
equivalent to the assumption that the fiber gradient is a
diffeomorphism of the statistical bundles
$\Grad_{\text{2}} L \colon \expbundleat \mu \to \mixbundleat
\mu$. This is related to the properties of \emph{regularity} and
hyper-regularity, cf. \cite[sec.~3.6]{abraham|marsden:1978}. The bilinear form
$\mixfiberat q \mu \times \expfiberat q \mu \ni (\eta,w) \mapsto
\scalarat q \eta w = \expectat q {\eta w}$ will always be written in
this order. The Legendre transform of $L_{q,t}$ is defined
for each $\eta \in \mixfiberat q \mu$ of the image of $\Grad_\text{e}
L(q, \cdot, t)$, so that the Hamiltonian is
\begin{equation}
  \label{eq:derived-Hamiltonian}
  H(q,\eta,t) = \scalarat q {\eta}{(\Grad_\text{e} L_{q,t})^{-1}(\eta)} - L(q,(\Grad_\text{e}
  L_{q,t})^{-1}(\eta)) \ .
\end{equation}
If $t \mapsto q(t)$ a solution of Euler-Lagrange
\eqref{eq:Euler-Lagrange}, the curve
$t \mapsto \zeta(t) = (q(t),\eta(t))$ in $\mixbundleat \mu$, where
$\eta(t) = \Grad_\text{e} L(q(t),\velocity q(t),t)$ is the
\emph{momentum}. The mixture bundle $\mixbundleat \mu$ then plays the role of the cotangent bundle in mechanics.

\begin{proposition}[Hamilton equations]\label{prop:Hamiltonequation}
  When \eqref{h1} holds, the momentum curve satisfies the \emph{Hamilton equations}, 
\begin{equation}
  \label{eq:Hamilton}
  \left\{\begin{aligned}
    \Derivby t \eta(t) &= - \Grad H(q(t),\eta(t),t) \\
     \velocity q(t) &= \Grad_\text{m} H(q(t),\eta(t),t) .
  \end{aligned}\right.
\end{equation}
Moreover,
\begin{equation}
  \label{eq:conservationH}
  \derivby t H(q(t),\eta(t),t) = \pderivby t H(q(t),\eta(t),t) \ .
\end{equation}
\end{proposition}
The special intrinsic form of the Hamilton equations follows from the covariant derivatives and the gradients of the statistical bundles.

\begin{example}[Mechanics of $\frac12 \scalarat q w w$]\label{ex:mech-quadratic-lagrangian}\label{eg5}
The scalar function $\frac12 \scalarat q w w$ of Example \ref{ex:quadratic-lagrangian} corresponds to the kinetic energy Lagrangian in mechanics. In this case, as first shown in \cite{pistone:2018lagrange}, the Euler–Lagrange equations are equivalent to the equations of geodesic motion, whose solution coincides with the one-dimensional exponential families.

Now, if $L(q,w) = \frac12 \scalarat q w w$ is our Lagrangian, then via Legendre transform, we obtain the Hamiltonian $H(q,\eta) = \frac 12 \scalarat q \eta \eta$. The gradients are
 
  \begin{align*}
    \Grad H(q,\eta) &= -\frac12 \left(\eta^2 - \expectat q {\eta^2}\right)\\
    \Grad_{\text{m}} H(q,\eta) &= \eta\\
    \Grad L(q,w) &= \frac12 (w^2 - \expectat q {w^2}) \\
    \Grad_{\text{e}} L(q,w) &= w
  \end{align*}

  For $\velocity q=w \in \mixbundleat \mu$, the Euler-Lagrange equation is
  \begin{equation*}
    \Derivby t \velocity q(t) = \frac12 \left(\velocity q(t)^2 -
      \expectat {q(t)} {\velocity q(t)^2}\right) \ ,
  \end{equation*}
where the covariant derivative is computed in $\mixbundleat \mu$, that
is, $\Derivby t \velocity q(t) = \ddot q(t) / q(t)$. In terms of the
exponential acceleration $\acceleration q(t) = \ddot q(t) / q(t) - \left(\velocity q(t)^2 -
      \expectat {q(t)} {\velocity q(t)^2}\right)$, the Euler-Lagrange
      equation reads
      \begin{equation*}
        \acceleration q(t) = - \frac12 \left((\velocity q(t))^2 -
      \expectat {q(t)} {(\velocity q(t))^2}\right) \ ,
      \end{equation*}
consistently with the result in Example 22.

The Hamilton equations are
\begin{equation*}
  \left\{
    \begin{aligned}
      \Derivby t \eta(t) &= \frac12 \left(\eta^2 - \expectat q {\eta^2}\right) \\
      \velocity q(t) &= \eta(t)
    \end{aligned}
\right. \ ,
\end{equation*}
with the covariant derivative again computed in $\mixbundleat \mu$.

The conserved energy is
\begin{equation*}
  H(q(t),\eta(t)) = \frac12 \scalarat {q(t)} {\velocity q(t)}{\velocity q(t)} = \frac12
  \expectat 1 {\frac {{\dot q(t)}^2}{q(t)}} \ .
\end{equation*}
which reflects in the conservation of the \emph{Fisher information}. 

\end{example}

\bibliographystyle{spmpsci}

\begin{thebibliography}{10}
\providecommand{\url}[1]{{#1}}
\providecommand{\urlprefix}{URL }
\expandafter\ifx\csname urlstyle\endcsname\relax
  \providecommand{\doi}[1]{DOI~\discretionary{}{}{}#1}\else
  \providecommand{\doi}{DOI~\discretionary{}{}{}\begingroup
  \urlstyle{rm}\Url}\fi

\bibitem{abraham|marsden:1978}
Abraham, R., Marsden, J.E.: Foundations of mechanics.
\newblock Benjamin/Cummings Publishing Co., Inc., Advanced Book Program,
  Reading, Mass. (1978).
\newblock Second edition, revised and enlarged, With the assistance of Tudor
  Ra\c tiu and Richard Cushman

\bibitem{abraham|marsden|ratiu:1988}
Abraham, R., Marsden, J.E., Ratiu, T.: Manifolds, tensor analysis, and
  applications, \emph{Applied Mathematical Sciences}, vol.~75, second edn.
\newblock Springer-Verlag (1988).
\newblock \doi{10.1007/978-1-4612-1029-0}.
\newblock \urlprefix\url{http://dx.doi.org/10.1007/978-1-4612-1029-0}

\bibitem{absil|mahony|sepulchre:2008}
Absil, P.A., Mahony, R., Sepulchre, R.: Optimization algorithms on matrix
  manifolds.
\newblock Princeton University Press (2008).
\newblock With a foreword by Paul Van Dooren

\bibitem{adams|fournier:2003}
Adams, R.A., Fournier, J.J.F.: Sobolev spaces, \emph{Pure and Applied
  Mathematics (Amsterdam)}, vol. 140, second edn.
\newblock Elsevier/Academic Press, Amsterdam (2003)

\bibitem{aliprantis|border:2006}
Aliprantis, C.D., Border, K.C.: Infinite dimensional analysis, third edn.
\newblock Springer, Berlin (2006).
\newblock A hitchhiker's guide

\bibitem{amari:85}
Amari, S.: Differential-geometrical methods in statistics, \emph{Lecture Notes
  in Statistics}, vol.~28.
\newblock Springer-Verlag (1985)

\bibitem{amari:87}
Amari, S.: Differential geometry in statistical inference.
\newblock In: Proceedings of the 46th Session of the International Statistical
  Institute, Vol. 2 (Tokyo, 1987), \emph{Bulletin de l'Institut International
  de Statistique}, vol.~52, pp. 321--338 (1987)

\bibitem{amari:87dual}
Amari, S.: Dual connections on the {H}ilbert bundles of statistical models.
\newblock In: Geometrization of statistical theory (Lancaster, 1987), pp.
  123--151. ULDM Publ. (1987)

\bibitem{amari|nagaoka:2000}
Amari, S., Nagaoka, H.: Methods of information geometry.
\newblock American Mathematical Society (2000).
\newblock Translated from the 1993 Japanese original by Daishi Harada

\bibitem{amari:1998natural}
Amari, S.I.: Natural gradient works efficiently in learning.
\newblock Neural Computation \textbf{10}(2), 251--276 (1998).
\newblock \doi{10.1162/089976698300017746}.
\newblock \urlprefix\url{http://dx.doi.org/10.1162/089976698300017746}

\bibitem{ay|jost|le|schwachhofer:2017}
Ay, N., Jost, J., L\^e, H.V., Schwachh\"ofer, L.: Information geometry,
  \emph{Ergebnisse der Mathematik und ihrer Grenzgebiete. 3. Folge. A Series of
  Modern Surveys in Mathematics [Results in Mathematics and Related Areas. 3rd
  Series. A Series of Modern Surveys in Mathematics]}, vol.~64.
\newblock Springer, Cham (2017)

\bibitem{bourbaki:71variete}
Bourbaki, N.: Variétés differentielles et analytiques. Fascicule de
  résultats / Paragraphes 1 à 7.
\newblock No. XXXIII in Éléments de mathématiques. Hermann (1971)

\bibitem{brezis:2011fasspde}
Brezis, H.: Functional analysis, {S}obolev spaces and partial differential
  equations.
\newblock Universitext. Springer, New York (2011)

\bibitem{brown:86}
Brown, L.D.: Fundamentals of statistical exponential families with applications
  in statistical decision theory.
\newblock No.~9 in IMS Lecture Notes. Monograph Series. Institute of
  Mathematical Statistics (1986)

\bibitem{buldygin|kozachenko:2000}
Buldygin, V.V., Kozachenko, Y.V.: Metric characterization of random variables
  and random processes, \emph{Translations of Mathematical Monographs}, vol.
  188.
\newblock American Mathematical Society, Providence, RI (2000).
\newblock Translated from the 1998 Russian original by V. Zaiats

\bibitem{burdet|combe|nencka:2001}
Burdet, G., Combe, P., Nencka, H.: On real {H}ilbertian info-manifolds.
\newblock In: Disordered and complex systems ({L}ondon, 2000), \emph{AIP Conf.
  Proc.}, vol. 553, pp. 153--158. Amer. Inst. Phys. (2001)

\bibitem{docarmo:1992}
do~Carmo, M.P.: Riemannian geometry.
\newblock Mathematics: Theory \& Applications. Birkhäuser Boston Inc. (1992).
\newblock Translated from the second Portuguese edition by Francis Flaherty

\bibitem{cena:2002}
Cena, A.: Geometric structures on the non-parametric statistical manifold.
\newblock Ph.D. thesis, Universit\`a degli Studi di Milano (2002)

\bibitem{cena|pistone:2007}
Cena, A., Pistone, G.: Exponential statistical manifold.
\newblock Ann. Inst. Statist. Math. \textbf{59}(1), 27--56 (2007)

\bibitem{chirco|malago|pistone:2020}
Chirco, G., Malagò, L., Pistone, G.: Lagrangian and hamiltonian mechanics for
  probabilities on the statistical manifold (2020).
\newblock \doi{10.48550/ARXIV.2009.09431}.
\newblock \urlprefix\url{https://arxiv.org/abs/2009.09431}

\bibitem{cover|thomas:2006}
Cover, T.M., Thomas, J.A.: Elements of information theory, second edn.
\newblock Wiley-Interscience [John Wiley \& Sons] (2006)

\bibitem{dawid:75}
Dawid, A.P.: Discussion of a paper by {B}radley {E}fron.
\newblock Ann. Statist. \textbf{3}(6), 1231--1234 (1975)

\bibitem{dawid:1977as}
Dawid, A.P.: Further comments on: ``{S}ome comments on a paper by {B}radley
  {E}fron''\ ({A}nn. {S}tatist. {\bf 3} (1975), 1189--1242).
\newblock Ann. Statist. \textbf{5}(6), 1249 (1977)

\bibitem{efron:1975}
Efron, B.: Defining the curvature of a statistical problem (with applications
  to second order efficiency).
\newblock Ann. Statist. \textbf{3}(6), 1189--1242 (1975).
\newblock With a discussion by C. R. Rao, Don A. Pierce, D. R. Cox, D. V.
  Lindley, Lucien LeCam, J. K. Ghosh, J. Pfanzagl, Niels Keiding, A. P. Dawid,
  Jim Reeds and with a reply by the author

\bibitem{efron:1978}
Efron, B.: The geometry of exponential families.
\newblock Ann. Statist. \textbf{6}(2), 362--376 (1978)

\bibitem{ekeland|temam:1999convex2nd}
Ekeland, I., T\'emam, R.: Convex analysis and variational problems,
  \emph{Classics in Applied Mathematics}, vol.~28, english edn.
\newblock Society for Industrial and Applied Mathematics (SIAM) (1999).
\newblock \doi{10.1137/1.9781611971088}.
\newblock \urlprefix\url{http://dx.doi.org/10.1137/1.9781611971088}.
\newblock Translated from the French

\bibitem{gibilisco:2020}
Gibilisco, P.: $l^p$-unit spheres and the $\alpha$-geometries: Questions and
  perspectives.
\newblock Entropy \textbf{22}(12), 1409 (2020).
\newblock \doi{10.3390/e22121409}.
\newblock \urlprefix\url{https://doi.org/10.3390/e22121409}

\bibitem{gibilisco|pistone:98}
Gibilisco, P., Pistone, G.: Connections on non-parametric statistical manifolds
  by {O}rlicz space geometry.
\newblock IDAQP \textbf{1}(2), 325--347 (1998)

\bibitem{hivarinen:2005}
Hyv\"{a}rinen, A.: Estimation of non-normalized statistical models by score
  matching.
\newblock J. Mach. Learn. Res. \textbf{6}, 695--709 (2005)

\bibitem{jost:2017textbook}
Jost, J.: Riemannian geometry and geometric analysis, seventh edn.
\newblock Universitext. Springer, Cham (2017).
\newblock \doi{10.1007/978-3-319-61860-9}.
\newblock \urlprefix\url{https://doi.org/10.1007/978-3-319-61860-9}

\bibitem{kakutani:1948}
Kakutani, S.: On equivalence of infinite product measures.
\newblock The Annals of Mathematics \textbf{49}(1), 214 (1948).
\newblock \doi{10.2307/1969123}.
\newblock \urlprefix\url{http://dx.doi.org/10.2307/1969123}

\bibitem{kass:89}
Kass, R.E.: The geometry of asymptotic inference (with discussion).
\newblock Statistical Science \textbf{4}, 188--234 (1089)

\bibitem{kass|vos:1997}
Kass, R.E., Vos, P.W.: Geometrical foundations of asymptotic inference.
\newblock Wiley Series in Probability and Statistics: Probability and
  Statistics. John Wiley \& Sons, Inc., New York (1997).
\newblock \doi{10.1002/9781118165980}.
\newblock \urlprefix\url{http://dx.doi.org/10.1002/9781118165980}.
\newblock A Wiley-Interscience Publication

\bibitem{klingenberg:1995}
Klingenberg, W.P.A.: Riemannian geometry, \emph{De Gruyter Studies in
  Mathematics}, vol.~1, second edn.
\newblock Walter de Gruyter \& Co., Berlin (1995).
\newblock \doi{10.1515/9783110905120}.
\newblock \urlprefix\url{https://doi.org/10.1515/9783110905120}

\bibitem{landau|lifshits:1980}
Landau, L.D., Lifshits, E.M.: Course of Theoretical Physics. Statistical
  Physics., vol.~V, 3rd edn.
\newblock Butterworth-Heinemann (1980)

\bibitem{lang:1995}
Lang, S.: Differential and {R}iemannian manifolds, \emph{Graduate Texts in
  Mathematics}, vol. 160, third edn.
\newblock Springer-Verlag (1995)

\bibitem{lott:2008calculations}
Lott, J.: Some geometric calculations on {W}asserstein space.
\newblock Comm. Math. Phys. \textbf{277}(2), 423--437 (2008).
\newblock \doi{10.1007/s00220-007-0367-3}.
\newblock \urlprefix\url{https://doi.org/10.1007/s00220-007-0367-3}

\bibitem{malliavin:1995}
Malliavin, P.: Integration and probability, \emph{Graduate Texts in
  Mathematics}, vol. 157.
\newblock Springer-Verlag (1995).
\newblock With the collaboration of H\'el\'ene Airault, Leslie Kay and G\'erard
  Letac, Edited and translated from the French by Kay, With a foreword by Mark
  Pinsky

\bibitem{musielak:1983}
Musielak, J.: Orlicz spaces and modular spaces, \emph{Lecture Notes in
  Mathematics}, vol. 1034.
\newblock Springer-Verlag (1983)

\bibitem{nomizu|sasaki:94}
Nomizu, K., Sasaki, T.: Affine differential geometry: geometry of affine
  immersions.
\newblock No. 111 in Cambridge Tracts in Mathematics. Cambridge University
  Press (1994)

\bibitem{ogouyandjo|wadagni:2020}
Ogouyandjou, C., Wadagni, N.: Wasserstein riemannian geometry on statistical
  manifold.
\newblock International Electronic Journal of Geometry \textbf{13}(2),
  144–151 (2020).
\newblock \doi{10.36890/iejg.689702}.
\newblock \urlprefix\url{http://dx.doi.org/10.36890/IEJG.689702}

\bibitem{otto:2001}
Otto, F.: The geometry of dissipative evolution equations: the porous medium
  equation.
\newblock Comm. Partial Differential Equations \textbf{26}(1-2), 101--174
  (2001).
\newblock \urlprefix\url{../publications/Riemann.ps}

\bibitem{petersen|mueller:2016}
Petersen, A., M\"{u}ller, H.G.: Functional data analysis for density functions
  by transformation to a hilbert space.
\newblock The Annals of Statistics \textbf{44}(1) (2016).
\newblock \doi{10.1214/15-aos1363}.
\newblock \urlprefix\url{https://doi.org/10.1214/15-aos1363}

\bibitem{pistone:2013Entropy}
Pistone, G.: Examples of the application of nonparametric information geometry
  to statistical physics.
\newblock Entropy \textbf{15}(10), 4042--4065 (2013).
\newblock \doi{10.3390/e15104042}.
\newblock \urlprefix\url{http://dx.doi.org/10.3390/e15104042}

\bibitem{pistone:2013GSI}
Pistone, G.: Nonparametric information geometry.
\newblock In: F.~Nielsen, F.~Barbaresco (eds.) Geometric science of
  information, \emph{Lecture Notes in Comput. Sci.}, vol. 8085, pp. 5--36.
  Springer, Heidelberg (2013).
\newblock First International Conference, GSI 2013 Paris, France, August 28-30,
  2013 Proceedings

\bibitem{pistone:2018}
Pistone, G.: Information geometry of the {G}aussian space.
\newblock In: Information geometry and its applications, \emph{Springer Proc.
  Math. Stat.}, vol. 252, pp. 119--155. Springer, Cham (2018)

\bibitem{pistone:2018lagrange}
Pistone, G.: Lagrangian function on the finite state space statistical bundle.
\newblock Entropy \textbf{20}(2), 139 (2018).
\newblock \doi{10.3390/e20020139}.
\newblock \urlprefix\url{http://www.mdpi.com/1099-4300/20/2/139}

\bibitem{pistone:2020-NPCS}
Pistone, G.: Information geometry of the probability simplex: A short course.
\newblock Nonlinear Phenomena in Complex Systems \textbf{23}(2), 221--242
  (2020)

\bibitem{pistone|rogantin:99}
Pistone, G., Rogantin, M.: The exponential statistical manifold: mean
  parameters, orthogonality and space transformations.
\newblock Bernoulli \textbf{5}(4), 721--760 (1999)

\bibitem{pistone|sempi:95}
Pistone, G., Sempi, C.: An infinite-dimensional geometric structure on the
  space of all the probability measures equivalent to a given one.
\newblock Ann. Statist. \textbf{23}(5), 1543--1561 (1995)

\bibitem{rao:45}
Radhakrishna~Rao, C.: Information and the accuracy attainable in the estimation
  of statistical parameters.
\newblock Bull. Calcutta Math. Soc. \textbf{37}, 81--91 (1945)

\bibitem{rudin:1987-3rd}
Rudin, W.: Real and complex analysis, third edn.
\newblock McGraw-Hill Book Co., New York (1987)

\bibitem{santacroce|siri|trivellato:2016}
Santacroce, M., Siri, P., Trivellato, B.: New results on mixture and
  exponential models by {O}rlicz spaces.
\newblock Bernoulli \textbf{22}(3), 1431--1447 (2016).
\newblock \doi{10.3150/15-BEJ698}.
\newblock \urlprefix\url{https://doi.org/10.3150/15-BEJ698}

\bibitem{santacroce|siri|trivellato:2018}
Santacroce, M., Siri, P., Trivellato, B.: Exponential models by {O}rlicz spaces
  and applications.
\newblock J. Appl. Probab. \textbf{55}(3), 682--700 (2018).
\newblock \doi{10.1017/jpr.2018.45}.
\newblock \urlprefix\url{https://doi.org/10.1017/jpr.2018.45}

\bibitem{schwartz:1981}
Schwartz, L.: Cours d'analyse. 1, second edn.
\newblock Hermann, Paris (1981)

\bibitem{shima:2007}
Shima, H.: The geometry of {H}essian structures.
\newblock World Scientific Publishing Co. Pte. Ltd., Hackensack, NJ (2007).
\newblock \doi{10.1142/9789812707536}.
\newblock \urlprefix\url{http://dx.doi.org/10.1142/9789812707536}

\bibitem{siri|trivellato:2021-SPL}
Siri, P., Trivellato, B.: Robust concentration inequalities in maximal
  exponential models.
\newblock Statistics \& Probability Letters \textbf{170}, 109001 (2021).
\newblock \doi{https://doi.org/10.1016/j.spl.2020.109001}.
\newblock
  \urlprefix\url{http://www.sciencedirect.com/science/article/pii/S0167715220303047}

\bibitem{wainwright:2019-hds}
Wainwright, M.J.: High-dimensional statistics: a non-asymptotic viewpoint.
\newblock Cambridge Series in Statistical and Probabilistic Mathematics.
  Cambridge University Press, Cambridge (2019).
\newblock \doi{10.1017/9781108627771}

\bibitem{weaver:2018-la-2nd-ed}
Weaver, N.: Lipschitz algebras.
\newblock World Scientific Publishing Co. Pte. Ltd., Hackensack, NJ (2018).
\newblock Second edition

\bibitem{weyl:1952}
Weyl, H.: Space- time- matter / by Hermann Weyl.
\newblock Dover, New York (1952).
\newblock Translation of the 1921 RAUM ZEIT MATERIE

\bibitem{cenkov:1982}
Čencov, N.N.: Statistical decision rules and optimal inference,
  \emph{Translations of Mathematical Monographs}, vol.~53.
\newblock American Mathematical Society (1982).
\newblock Translation from the Russian edited by Lev J. Leifman

\end{thebibliography}

\end{document}